\documentclass{article}
\usepackage{amsmath}
\usepackage{times}
\usepackage{amssymb}
\usepackage{amsthm}
\usepackage{bm}
\usepackage{graphicx}
\usepackage{caption}
\usepackage{geometry}
\usepackage[round]{natbib}
\usepackage{color}
\usepackage{subfigure}
\usepackage{comment}
\usepackage{enumitem}
\usepackage{mathtools}
\usepackage{thmtools, thm-restate}
\usepackage[ruled]{algorithm2e}
\usepackage{verbatim}

\usepackage[colorlinks,urlcolor=blue,citecolor=blue,linkcolor=blue]{hyperref}
\usepackage[colorinlistoftodos]{todonotes}

\newcommand{\footremember}[2]{%
    \footnote{#2}
    \newcounter{#1}
    \setcounter{#1}{\value{footnote}}%
}
\newcommand{\footrecall}[1]{%
    \footnotemark[\value{#1}]%
}

\newcommand{\innp}[1]{\left\langle #1 \right\rangle}

\newcommand{\norm}[1]{\left\| #1 \right\|}

\newcommand{\mA}{\mathbf{A}}

\newcommand{\zeros}{\textbf{0}}
\newcommand{\vx}{\mathbf{x}}
\newcommand{\tF}{\widetilde{F}}
\newcommand{\tG}{\widetilde{G}}

\newcommand{\cu}{\mathcal{U}}
\newcommand{\cf}{\mathcal{F}}
\newcommand{\cc}{\mathcal{C}}

\newcommand{\ce}{\mathcal{E}}
\newcommand{\co}{\mathcal{O}}

\newcommand{\cv}{\mathcal{V}}

\newcommand{\vy}{\mathbf{y}}

\newcommand{\vv}{\mathbf{v}}

\newcommand{\vb}{\mathbf{b}}

\newcommand{\vu}{\mathbf{u}}

\newcommand{\rr}{\mathbb{R}}
\newcommand{\ee}{\mathbb{E}}
\newcommand{\nn}{\mathbb{N}}
\newcommand{\um}{\vu^{*}}
\newcommand{\ui}{\vu_{0}}

\newcommand{\hF}{\widehat{F}}
\newcommand{\grad}{G_{\eta}}
\newcommand{\proj}{\Pi_{\cu}}

\graphicspath{{imgs/}}

\makeatletter
\def\mathcolor#1#{\@mathcolor{#1}}
\def\@mathcolor#1#2#3{%
  \protect\leavevmode
  \begingroup
    \color#1{#2}#3%
  \endgroup
}
\makeatother

\newcommand*{\vsepfbox}[1]{%
  \begingroup
    \sbox0{\fbox{#1}}%
    \setlength{\fboxrule}{0pt}%
    \mbox{\kern-\fboxsep\fbox{\unhbox0}\kern-\fboxsep}%
  \endgroup
}

\theoremstyle{plain} \numberwithin{equation}{section}
\newtheorem{theorem}{Theorem}[section]
\numberwithin{theorem}{section}
\newtheorem{corollary}[theorem]{Corollary}

\newtheorem{lemma}[theorem]{Lemma}
\newtheorem{proposition}[theorem]{Proposition}

\theoremstyle{definition}

\newtheorem{remark}[theorem]{Remark}

\newtheorem{assumption}{Assumption}

\newcommand{\cb}[1]{{\color{blue}{\textbf{CB:} #1}}}
\newcommand{\jd}[1]{{\color{purple}{\textbf{JD:} #1}}}

\newcommand{\xc}[1]{{\color{orange}{[\textbf{XC:} #1}]}}
\allowdisplaybreaks[4]

\title{Stochastic Halpern Iteration with Variance Reduction \\
                    for Stochastic Monotone Inclusions}
\author{Xufeng Cai\footremember{wisc}{Department of Computer Sciences, University of Wisconsin-Madison. XC (\href{mailto:xcai74@wisc.edu}{xcai74@wisc.edu}), CS (\href{mailto:chaobing.song@wisc.edu}{chaobing.song@wisc.edu}), JD (\href{mailto:jelena@cs.wisc.edu}{jelena@cs.wisc.edu}).}
\and Chaobing Song\footrecall{wisc}
\and Cristóbal Guzmán\footremember{twente}{Institute for Mathematical and Computational Eng., Facultad de Matem\'aticas and Escuela de Ingenier\'ia, Pontificia Universidad Católica de Chile. \href{mailto:crguzmanp@mat.uc.cl}{crguzmanp@mat.uc.cl}.}
\and Jelena Diakonikolas\footrecall{wisc}
}
\date{}

\begin{document}

\maketitle
\begin{abstract}%
We study stochastic monotone inclusion problems, which widely appear in machine learning applications, including robust regression and adversarial learning. We propose novel variants of stochastic Halpern iteration with recursive variance reduction. In the cocoercive---and more generally Lipschitz-monotone---setup, our algorithm attains $\epsilon$ norm of the operator with $\mathcal{O}(\frac{1}{\epsilon^3})$ stochastic operator evaluations, which significantly improves over state of the art $\mathcal{O}(\frac{1}{\epsilon^4})$ stochastic operator evaluations required for existing monotone inclusion solvers applied to the same problem classes. 
We further show how to couple one of the proposed variants of stochastic Halpern iteration with a scheduled restart scheme to solve stochastic monotone inclusion problems with ${\mathcal{O}}(\frac{\log(1/\epsilon)}{\epsilon^2})$ stochastic operator evaluations under additional sharpness or strong monotonicity assumptions.
\end{abstract}


\section{Introduction}\label{sec:intro}
Recent trends in machine learning (ML) involve the study of models whose solutions do not reduce to optimization but rather to {\em equilibrium conditions}. Standard examples include generative adversarial networks, adversarially robust training of ML models, and training of ML models under notions of fairness.  It turns out that several of these equilibrium conditions (including, but not limited to, first-order stationary points, saddle-points, and Nash equilibria of minimax games) can be cast as solutions to a {\em monotone inclusion} problem, which is defined as the problem of computing 
a zero of a (maximal) monotone operator $F: \rr^d\to \rr^d$ (see \eqref{def:MI} for a formal definition). In the context of min-max optimization problems, monotone inclusion reduces to a stationarity condition, which for unconstrained problems boils down to finding a point with small gradient norm.  
%
    
Of particular interest to machine learning are stochastic versions of these problems, in which the operator $F$ is not readily available, but can only be accessed through a stochastic oracle $\hF$. Such are the settings mentioned above, where the definitions of equilibria involve expectations over  continuous high-dimensional spaces. The corresponding  problem, known as the {\em stochastic monotone inclusion}, has not been thoroughly studied, particularly in the context of its stochastic oracle complexity. Understanding stochastic oracle complexity of monotone inclusion in all standard settings with Lipschitz operators, from the algorithmic aspect, is the main motivation of this work. 
    
    
    
    \subsection{Contributions}
    We study three main classes of stochastic monotone inclusion problems with Lipschitz operators, defined by the assumptions made about the operator itself: (i) cocoercive class, which is the most restricted class, but nevertheless fundamental for understanding monotone inclusion, as it relates to the problem of finding a fixed point of a nonexpansive (1-Lipschitz) operator; (ii) Lipschitz monotone class, which is perhaps the most basic class arising in the study of smooth convex-concave min-max optimization problems; and (iii) Lipschitz monotone class with an additional sharpness property of the operator. Sharpness is a widely studied property of optimization problems, often referred to as the ``local error bound'' condition, which is weaker than strong convexity and roughly corresponds to  the problem landscape being curved outside of the solution set (see~\citet{pang1997error} for a survey of classical results).  
    
    From an algorithmic standpoint, we consider variants of classical Halpern iteration~\citep{halpern1967fixed}, which was originally introduced for solving fixed point equations with nonexpansive operators. Variants of this iteration have recently been shown to lead to (near-)optimal  first-order oracle complexity for all aforementioned standard problem classes in \emph{deterministic} settings~\citep{diakonikolas2020halpern,diakonikolas2021potential,Yoon2021OptimalGradientNorm}. 
    However, to the best of our knowledge, stochastic variants of these methods have received very limited attention prior to our work. The only results we are aware of are for a two-step extragradient-like variant of Halpern iteration in negative comonotone Lipschitz settings~\citep{lee2021fast} and which show that when variance of operator estimates is bounded by order-$\frac{\epsilon^2}{k}$ in iteration $k,$ the method attains operator norm $\epsilon$ after $\mathcal{O}(\frac{1}{\epsilon})$ iterations. However, \citet{lee2021fast} does not discuss how such variance control would be obtained. Simple mini-batching, as we show, only leads to $\mathcal{O}(\frac{1}{\epsilon^4})$ stochastic oracle complexity.
    
    We show that existing variants of the Halpern iteration~\citep{diakonikolas2020halpern,tran2021halpern} can be effectively combined with recursive variance reduction~\citep{li2020page} to obtain $\co(\frac{1}{\epsilon^3})$ stochastic oracle complexity in the cocoercive and Lipschitz monotone setups. We then show that the complexity can be further reduced to $\co\big(\frac{1}{\epsilon^2}\log(\frac{1}{\epsilon})\big)$ under an additional sharpness assumption about the operator. 
    The last bound is unimprovable in terms of the dependence on $\epsilon,$ due to existing lower bounds, as we argue for completeness in Section~\ref{appx:lower-bound-reductions}.
    
    To the best of our knowledge, our work is the first to use variance reduction to reduce stochastic oracle complexity of monotone inclusion (small gradient norm in min-max optimization settings), and the attained bounds are the best achieved to date for direct methods.

    \subsection{Techniques}
    
    Inspired by the potential function originally used by~\citet{diakonikolas2020halpern} and later used either in the same or slightly modified form  by~\citet{diakonikolas2021potential,Yoon2021OptimalGradientNorm,tran2021halpern,lee2021fast}, we adapt this potential function-based argument to account for stochastic error terms arising due to the stochastic oracle access to the operator. We first show that in the cocoercive minibatch setting, this argument only leads to $\co(\frac{1}{\epsilon^4})$ stochastic oracle complexity, and it is unclear how to improve it directly, as the analysis appears tight. We then combine the cocoercive variant of Halpern iteration~\citep{diakonikolas2020halpern} with the PAGE estimator~\citep{li2020page} to reduce the stochastic oracle complexity to $\co(\frac{1}{\epsilon^3})$. The same variance reduced estimator is also used in conjunction with the two-step extrapolated variant of Halpern iteration introduced by~\cite{tran2021halpern}, as a direct application of Halpern iteration is not known to converge on the class of Lipschitz monotone operators. 
    
    While the basic ideas in our arguments are simple, their realization requires addressing major technical obstacles. First, the variance reduced estimator that we use \citep{li2020page} was originally devised for smooth nonconvex optimization problems, where it was coupled with a stochastic variant of gradient descent. This is significant, because the proof relies on a descent lemma, which allows cancelling the error arising from the variance of the estimator by the ``descent'' part. Such an argument is not possible in our setting, as there is no objective function to descend on. Instead, our analysis relies on an intricate inductive argument that ensures that the expected norm of the operator is bounded in each iteration, assuming a suitable bound on the variance of the estimator. To obtain our desired result for the variance, we propose a data-dependent batch allocation in PAGE estimator~\citep{li2020page} (see Corollary \ref{cor:variance-bound}), which scales proportionally to the squared distance between successive iterates, similar to \citet{arjevani2020second}. We inductively argue that the squared distance between successive iterates arising in the batch size of the estimator reduces at rate $\frac{1}{k^2}$ \emph{in expectation}. 
    This allows us to further certify that the estimators do not only remain accurate, but their variance decreases as ${\cal O}(\epsilon^2/k)$, where $k$ is the iteration count.
    
    In the context of the potential function argument, unlike in the deterministic settings, we \emph{do not} establish that the potential function is non-increasing, even in expectation. The stochastic error terms that arise due to the stochastic nature of the operator evaluations are controlled by taking slightly smaller step sizes than in the vanilla  methods from~\citet{diakonikolas2020halpern,tran2021halpern}, which allows us to ``leak'' negative quadratic terms that are further used in controlling the stochastic error. The argument for controlling the value of the potential function is itself coupled with the inductive argument for ensuring that the expected operator norm remains bounded. 
    
    Finally, while applying a restarting strategy is standard under sharpness conditions~\citep{roulet2020sharpness}, obtaining the claimed stochastic oracle complexity result of $\co\big(\frac{1}{\epsilon^2}\log(\frac{1}{\epsilon})\big)$ requires a rather technical argument to bound the total number of stochastic queries to the operator. 

    \subsection{Related work}
    






\paragraph{Monotone inclusion and variational inequalities.}

Variational inequality problems were originally devised to deal with approximating equilibria. 
Their systematic study was initiated by~\citet{stampacchia1964formes}. The relationship between variational inequalities and min-max optimization was observed soon after~\citet{rockafellar1970monotone}, while one of the earliest papers to study solving monotone inclusion as a generalization of variational inequalities, convex and min-max optimization, and complementarity problems is~\cite{rockafellar1976monotone}. For a historical overview of this area and an extensive review of classical results, see~\citet{facchinei2007finite}. 

In the case of monotone operators, standard variants of variational inequality problems (see Section~\ref{sec:prelim}) and monotone inclusion are equivalent---their solution sets coincide. This is a consequence of the celebrated Minty Theorem~\citep{minty1962monotone}. However, there is a major difference between these problems when it comes to solving them to a finite accuracy. In particular, on unbounded domains, approximating variational inequalities is meaningless, whereas monotone inclusion remains well-defined. This is most readily seen from the observation that mapping from min-max optimization, variational inequalities correspond to primal-dual gap guarantees, while monotone inclusion corresponds to a guarantee in gradient norm. For a simple bilinear function $f(x, y) = xy$ which has the unique min-max solution at $(x, y) = (0, 0)$, the primal-dual gap is infinite for any point other than $(0, 0),$ while the gradient remains finite and is a good proxy for measuring quality of a  solution. Further, even on bounded domains or using restricted gap functions on unbounded domains as in e.g.,~\cite{nesterov2007dual}, optimal oracle complexity guarantees for approximate monotone inclusion imply optimal complexity guarantees for approximately satisfied variational inequalities (see, e.g.,~\cite{diakonikolas2020halpern}). The opposite does not hold in general. In particular, in deterministic settings, standard algorithms such as the celebrated extragradient~\citep{korpelevich1977extragradient,nemirovski2004prox}, dual extrapolation~\citep{nesterov2007dual}, or Popov's method~\citep{Popov1980} that have the optimal oracle complexity $O(\frac{1}{\epsilon})$ for approximating variational inequalities are suboptimal for monotone inclusion and attain oracle complexity of the order $O(\frac{1}{\epsilon^2})$~\citep{golowich2020last,diakonikolas2021potential}. 

\paragraph{Halpern iteration.} Halpern iteration is a classical fixed point iteration originally introduced by \citet{halpern1967fixed}, and studied extensively in terms of both its asymptotic and non-asymptotic convergence guarantees~~\citep{wittmann1992approximation,leustean2007rates,lieder2019convergence, kohlenbach2011quantitative,kohlenbach2012effective,cheval2022modified}.. The first tight nonasymptotic convergence rate guarantee of $1/t$ was obtained in~\citet{lieder2019convergence,Sabach:2017}. This rate was also matched by an alternative method proposed by~\citet{kim2019accelerated}. 

The usefulness of Halpern iteration for solving monotone inclusion problems was first observed by~\citet{diakonikolas2020halpern},\footnote{Interestingly, the algorithm proposed by~\citet{kim2019accelerated} for cocoercive inclusion coincides with the Halpern iteration for a related nonexpansive operator (see~\citet[Proposition 4.3]{Contreras:2021}).} who showed that its variants can be used to obtain near-optimal oracle complexity results for all standard classes of monotone inclusion problems with Lipschitz operators also studied in this work. The near-tightness (up to poly-logarithmic factors) of the results from~\citet{diakonikolas2020halpern} was certified using lower bound reductions from min-max optimization lower bounds introduced by~\citet{Ouyang2019}. These lower bounds were made tight for the cocoercive setup in~\citet{diakonikolas2021potential}.

The generalization of Halpern iteration from the cocoercive to Lipschitz monotone setup in~\cite{diakonikolas2020halpern} utilized approximating what is known as the resolvent operator, which led to a double-loop algorithm and an additional $\log(1/\epsilon)$ in the resulting complexity. This log factor was shaved off in~\citet{Yoon2021OptimalGradientNorm}, who introduced a two-step variant of Halpern iteration, inspired by the extragradient method of~\citet{korpelevich1977extragradient}. 
The results of \citet{diakonikolas2020halpern,Yoon2021OptimalGradientNorm} were further extended to other classes of Lipschitz operators by~\citet{tran2021halpern,lee2021fast}. 
Except for \cite{lee2021fast} which considered controlled variance as discussed above, all of the existing results only targeted deterministic settings. 

\paragraph{Stochastic settings and variance reduction.} Vanilla stochastic gradient methods have constant variance of stochastic gradients, which creates a bottleneck in the convergence rate. To improve the convergence rate, in the past decade, powerful variance reduction techniques have been proposed. 

For strongly convex finite-sum problems, SAG~\citep{schmidt2017minimizing}, which used a biased stochastic estimator of the full gradient, was the first stochastic gradient method with a linear convergence rate. \citet{johnson2013accelerating} and \citet{defazio2014saga} improved \citet{schmidt2017minimizing} by proposing unbiased estimators of SVRG-type and SAGA-type, respectively. Such unbiased estimators were further combined with Nesterov acceleration~\citep{allen2017katyusha, song2020variance}, or applied to nonconvex finite-sum/infinite-sum problems \citep{reddi2016stochastic,lei2017non}. 
For nonconvex stochastic (infinite-sum) problems, SARAH \citep{nguyen2017sarah} and SPIDER \citep{fang2018spider,zhou2018finding,zhou2018stochastic} estimators were proposed to attain the optimal oracle complexity of $\co(1/\epsilon^3)$ for finding an {$\epsilon$-approximate} stationary point. Both estimators are referred to as ``recursive'' variance reduction estimators, as they are biased when taking expectation w.r.t.~current randomness but unbiased w.r.t.~all the randomness in history. PAGE \citep{li2020page} and STORM \citep{cutkosky2019momentum} significantly simplified SARAH and SPIDER in terms of reducing the number of loops 
and avoiding large minibatches, respectively. \citet{arjevani2020second} further extended this line of work by incorporating second-order information and dynamic batch sizes.

In the setting of min-max optimization and variational inequalities/monotone inclusion, variance reduction has primarily been used for approximating variational inequalities, corresponding to the primal-dual gap in min-max optimization; see, for example~\citet{palaniappan2016stochastic,alacaoglu2021stochastic,iusem2017extragradient,chavdarova2019reducing,carmon2019variance,loizou2021stochastic}. Under strong monotonicity (or sharpness in the case of~\citet{loizou2021stochastic}), such results generalize to monotone inclusion; however, to the best of our knowledge, there have been no results that address monotone inclusion under the weaker assumptions considered in this work. In the context of monotone inclusion with Lipschitz operators, the tightest complexity result that we are aware of is $\co(\frac{1}{\epsilon^4}),$ due to~\citet{diakonikolas2021efficient}, and it applies to a more general class of structured non-monotone Lipschitz operators, for the best iterate. 
The same oracle complexity can be deduced for the last iterate of a two-step variant of Halpern from \citet[Theorem 6.1]{lee2021fast}, using mini-batching. All the results in our work are also for the last iterate. 




\section{Preliminaries}\label{sec:prelim}

We consider a real $d$-dimensional normed space $(\rr^d, \norm{\cdot})$, where $\norm{\cdot}$ is induced by an inner product associated with the space, i.e., $\norm{\cdot} = \sqrt{\innp{\cdot, \cdot}}$. Let $\cu \subseteq \rr^d$ be closed and convex; in the unconstrained case, $\cu \equiv \rr^d$. When $\cu$ is bounded,  $D = \max_{\vu, \vv \in \cu}\|\vu - \vv\|$  denotes its diameter. 
%

\paragraph{Classes of monotone operators.} 
We say that an operator $F: \rr^d \rightarrow \rr^d$ is
\begin{enumerate}[wide=0pt, leftmargin=\parindent]
    \item monotone, if $\forall \vu, \vv \in \rr^d,$
%
$    \innp{F(\vu)- F(\vv), \vu - \vv} \geq 0. $
%
\item $L$-Lipschitz continuous for some $L > 0$, if $\forall \vu, \vv \in \rr^d,$
%
  $  \|F(\vu) - F(\vv)\| \leq L \norm{\vu - \vv}.$
%
\item $\gamma$-cocoercive for some $\gamma>0$, if $\forall \vu, \vv \in \rr^d$, 
%
   $ \innp{F(\vu) - F(\vv), \vu - \vv} \geq \gamma\norm{F(\vu) - F(\vv)}^2.$
%
\item $\mu$-strongly monotone 
for some $\mu>0$, if $\forall \vu, \vv \in \rr^d$, 
 $   \innp{F(\vu) - F(\vv), \vu - \vv} \geq \mu\norm{\vu - \vv}^2.$
\end{enumerate}
Note that we can easily specialize these definitions to the set $\cu$ by restricting  $\vu, \vv$ to be from $\cu$.

Throughout the paper, the minimum assumption that we make about an operator $F$ is that it is monotone and Lipschitz. 
%
%
%
Observe that any $\gamma$-cocoercive operator is monotone and $\frac{1}{\gamma}$-Lipschitz. The converse to this statement does not hold in general. 

%
%

\paragraph{Monotone inclusion and variational inequalities.} 
Monotone inclusion asks for $\vu^*$ such that
\begin{equation}\tag{MI} \label{def:MI}
\mathbf{0} \in F(\mathbf{u}^*)+\partial I_{\mathcal{U}}(\mathbf{u}^*),
\end{equation}
where $I_{\mathcal{U}}$ is the indicator function of the set $\cu$ and $\partial I_{\cu}(\cdot)$ denotes the subdifferential of $I_{\cu}$. 

If $F$ is continuous and monotone, the solution set to \eqref{def:MI} is the same as the solution set of the Stampacchia Variational Inequality (SVI) problem, which asks for $\vu^* \in \cu$ such that 
\begin{equation}\tag{SVI}\label{def:SVI}
(\forall \mathbf{u} \in \mathcal{U}): \quad\left\langle F\left(\mathbf{u}^{*}\right), \mathbf{u}-\mathbf{u}^{*}\right\rangle \geq 0.
\end{equation}
Further, when $F$ is monotone, the solution set of (\ref{def:SVI}) is equivalent to the solution set of the Minty Variational Inequality (MVI) problem consisting in finding $\vu^*$ such that 
\begin{equation}\tag{MVI}\label{def:MVI}
(\forall \mathbf{u} \in \mathcal{U}): \quad\left\langle F(\mathbf{u}), \mathbf{u}^{*}-\mathbf{u}\right\rangle \leq 0.
\end{equation}

We assume throughout the paper that a solution to monotone inclusion (\ref{def:MI}) exists, which implies that solutions to both (\ref{def:SVI}) and (\ref{def:MVI}) exist as well. Existence of solutions follows from standard results and is guaranteed whenever e.g., $\cu$ is compact, or, if there exists a compact set $\cu^{\prime}$ such that $\mathrm{Id} - \frac{1}{L}F$ maps $\cu^{\prime}$ to itself, where $\mathrm{Id}$ is the identity  map~\citep{facchinei2007finite}. As remarked in the introduction, in unbounded setups it is generally not possible to approximate \eqref{def:MVI} and \eqref{def:SVI}, 
whereas approximating \eqref{def:MI} is quite natural: we only need to find $\mathbf{u}$ such that $\zeros \in F(\mathbf{u}) + \partial I_{\cu}(\mathbf{u}) + \mathcal{B}(\epsilon)$, where $\zeros$ denotes the zero vector and $\mathcal{B}(\epsilon)$ denotes the centered ball of radius $\epsilon$.  

\paragraph{Stochastic access to the operator.} 

We consider the stochastic setting for monotone inclusion problems. More specifically, we make the following assumptions for stochastic queries to $F.$ These assumptions are made throughout the paper, without being explicitly invoked. 
\begin{assumption}[Unbiased samples with bounded variance]\label{assmpt:stoch-queries}
For each query point $\vx \in \cu$, we observe $\hF(\vx, z)$ where $z \sim P_z$ is a random variable that satisfies the following assumptions: 
\begin{equation}\notag
\ee_z\big[\hF(\vx, z)\big] = F(\vx) \quad \text{ and }\quad 
    \ee_z\big[\big\|{\hF(\vx, z) - F(\vx)}\big\|^2\big] \leq \sigma^2.
\end{equation}
\end{assumption}

\begin{assumption}[Multi-point oracle]\label{assmpt:multiQuery}
We can query a set of points $(\vx_1, \dots, \vx_n)$ and receive 
\begin{equation*}
    \hF(\vx_1, z), \dots, \hF(\vx_n, z) \quad \text{where} \quad z \sim P_z. 
\end{equation*}
\end{assumption}

\begin{assumption}[Lipschitz in expectation]\label{assmpt:Lipschitz}
$\ee_{z}\big[\big\|{\hF(\vu, z) - \hF(\vv, z)}\big\|^2\big] \leq L^2 \norm{\vu - \vv}^2$, $\forall \vu, \vv \in \cu$. 
\end{assumption}

{We note that complexity results of the paper will bound the total number of queries made to this oracle. In particular, if multiple query points and/or multiple samples $z$ are used in a single iteration, our complexity is given by the sum of all those queries throughout all iterations of the method. Also, Assumption~\ref{assmpt:Lipschitz} is primary with parameter $L$, by which $F$ is also $L$-Lipschitz using Jensen’s inequality.}






\paragraph{PAGE variance-reduced estimator.}
We now summarize a variant of the PAGE estimator, originally developed for smooth nonconvex optimization by~\citet{li2020page}, adapted to our setting. In particular, given queries to $\hF$, 
we define the variance reduced estimator $\tF(\vu_k)$ for $ k \geq 1$ by 
\begin{equation}\label{eq:PAGE}
  \tF(\vu_{k}) = \begin{cases}
     \frac{1}{S_1^{(k)}} \sum_{i = 1}^{S_1^{(k)}} \hF(\vu_{k}, z^{(k)}_i) & \text{w.~p. }  p_{k}, \\
     \tF(\vu_{k - 1}) + \frac{1}{S_2^{(k)}} \sum_{i = 1}^{S_2^{(k)}} \left(\hF(\vu_{k}, z^{(k)}_i) - \hF(\vu_{k - 1}, z^{(k)}_i)\right) & \text{w.~p. } 1 - p_{k},
\end{cases}
\end{equation}
where $p_0 = 1$, $z^{(k)}_i \overset{\text{i.i.d.}}{\sim} P_z$, and $S_1^{(k)}$ and $S_2^{(k)}$ are the sample sizes at iteration $k$. Observe that Assumption~\ref{assmpt:multiQuery} guarantees that we can query $\hF$ at $\vu_k$ and $\vu_{k - 1}$ using the same random seed. Our analysis will make use of conditional expectations, and to that end, we define natural filtration $\cf_k$ by $\cf_{k} := \sigma(\{\tF(\vu_j)\}_{j \leq k})$; namely $\cf_{k}$ contains all the randomness that arises in the definitions of $\tF(\vu_j)$ for $j \leq k.$ 
%
%
Following a similar argument as in~\citet{li2020page}, we  recursively bound the variance of the estimator $\tF$, as summarized in the following lemma. The proof is provided 
in Appendix~\ref{appx:prelim-omitted}. 

\begin{restatable}{lemma}{PAGElemma}
\label{lemma:recursive-variance-bound}
Let $F$ be a monotone operator accessed via stochastic queries $\hF$, under Assumptions~\ref{assmpt:stoch-queries}--\ref{assmpt:Lipschitz}. Then, the variance of 
$\tF$ defined by Eq.~\eqref{eq:PAGE} satisfies the following recursive bound: for all $k \geq 1,$ 
\begin{equation}\notag
    \begin{aligned}
    \ee[\|{\tF(\vu_{k}) - F(\vu_{k})}\|^2] \leq \;&   \frac{p_k\sigma^2}{S_1^{(k)}} + (1 - p_k)\Big(\ee[\|{\tF(\vu_{k- 1}) - F(\vu_{k - 1})}\|^2]
    + \ee\Big[\frac{L^2\|\vu_{k} - \vu_{k - 1}\|^2}{S_2^{(k)}}\Big]\Big).
\end{aligned}
\end{equation}
\end{restatable}

With the choices of $p_k, S_1^{(k)}, S_2^{(k)}$ specified in the following corollary and using induction with the inequality from Lemma~\ref{lemma:recursive-variance-bound}, we obtain the following bound on the variance. 
\begin{restatable}{corollary}{PAGEcoro}
\label{cor:variance-bound}
Given a target error $\epsilon > 0$, if for all $k \geq 1$, $p_k = \frac{2}{k + 1}, S_1^{(k)} \geq \big\lceil\frac{8\sigma^2}{p_k\epsilon^2}\big\rceil, S_2^{(k)} \geq \big\lceil\frac{8L^2\norm{\vu_k - \vu_{k - 1}}^2}{p_k^2\epsilon^2}\big\rceil$, then 
 $   \ee\big[\big\|{\tF(\vu_{k}) - F(\vu_{k})}\big\|^2\big] \leq \frac{\epsilon^2}{k}.$
\end{restatable}

\section{Stochastic Halpern Iteration for Cocoercive Operators}\label{sec:cocoercive}

In this section, we consider the setting of $\frac{1}{L}$-cocoercive operators $F.$ While cocoercivity is a strong assumption that implies that an operator is both Lipschitz and monotone (as discussed in Section~\ref{sec:prelim}), it is nevertheless the most basic setup for studying the Halpern iteration. In particular, while Halpern iteration can be applied directly to the nonexpansive counterpart of a cocoercive operator $F$ (i.e., to the linear transformation $\mathrm{Id} - \frac{2}{L}F$, where $\frac{1}{L}$ is an upper bound on the cocoercivity parameter of $F$), convergence does not seem possible to establish for the more general class of Lipschitz monotone operators. We begin this section by providing a generic proof of stochastic oracle complexity, which we then use to briefly illustrate how to obtain ${\cal O}(\frac{1}{\epsilon^4})$ oracle complexity with a simple minibatch stochastic estimator of $F$. 
We then show how to improve this bound to  ${\cal O}(\frac{1}{\epsilon^3})$   by applying the proposed variant of the PAGE estimator from Eq.~\eqref{eq:PAGE} to Halpern iteration. 


The stochastic variant of Halpern iteration that we consider is defined by
\begin{equation}\label{eq:Halpern}
    \vu_{k + 1} = \lambda_{k + 1}\vu_{0} + (1 - \lambda_{k + 1})\Big(\vu_{k} - \frac{2}{L_{k + 1}}\tF(\vu_{k})\Big),
\end{equation}
where $\tF$ is a stochastic (possibly biased) 
estimator of $F$, $\lambda_{k+1} = \Theta(\frac{1}{k})$ is the step size, and $L_{k+1} \geq L$ is a parameter of the algorithm. 
Compared to the classical iteration $\vu_{k + 1} = \lambda_{k + 1}\vu_{0} + (1 - \lambda_{k + 1})T(\vu_k)$, 
%
%
where $T:\rr^d \to\rr^d$ is a nonexpansive (1-Lipschitz) map~\citep{halpern1967fixed}, $T$ is replaced by the mapping $\mathrm{Id} - \frac{2}{L_{k + 1}}\tF$, which is stochastic and may not be nonexpansive (as the stochastic estimate $\tF$ of $F$ is not guaranteed to be cocoercive even when $F$ is). Compared to the iteration variant considered by~\citet{diakonikolas2020halpern}, the access to the monotone operator is stochastic and we also take slightly larger (by a factor of 2) values of $L_{k+1}$ to bound the stochastic error terms. 


Our argument for bounding the total number of stochastic queries to $F$ is based on the use of the following potential function $\cc_k = \frac{A_k}{L_k}\lVert F(\vu_k) \rVert^2 + B_k\innp{F(\vu_k), \vu_k - \ui}$, 
where $\{ A_{k} \}_{k \geq 1}$ and $\{ B_{k} \}_{k \geq 1}$ are positive and non-decreasing sequences of real numbers, while the step size $\lambda_k$ is defined by $\lambda_k := \frac{B_k}{A_k + B_k}$. 
Such potential function was previously used for the deterministic case of Halpern iteration in \citet{diakonikolas2020halpern,diakonikolas2021potential}.
Observe that even though we make oracle queries to $\hF$, the potential function $\cc_k$ and the final bound we obtain are in terms of the true operator value $F.$ 

Compared to the analysis of Halpern iteration in the deterministic case~\citep{diakonikolas2020halpern,diakonikolas2021potential}, our analysis for the stochastic case needs to account for the  error terms caused by accessing $F$ via stochastic queries and is based on an intricate inductive argument. A generic bound on iteration complexity, under mild assumptions about the estimator $\tF,$ is summarized in Theorem~\ref{thm:cocoercive}. The proof is in 
Appendix~\ref{appx:omitted-cocoercive}.

\begin{restatable}{theorem}{rateCoco}
\label{thm:cocoercive}
Given an arbitrary $\vu_0 \in \rr^d,$ suppose that iterates $\vu_k$ evolve according to Halpern iteration from Eq.~\eqref{eq:Halpern} for $k \geq 1,$ where $L_k = 2L$ and $\lambda_k = \frac{1}{k+1}.$ 
 Assume further that the stochastic estimate $\tF(\vu)$ is unbiased for $\vu = \vu_0$ and $\ee[\|F(\vu_0) - \tF(\vu_0)\|^2] \leq \frac{\epsilon^2}{8}$. 
 Given $\epsilon > 0,$ if for all $k \geq 1,$ we have that 
  $\ee\big[\big\|{F(\vu_k) - \tF(\vu_k)}\big\|^2\big] \leq \frac{\epsilon^2}{k}$, then for all $k \geq 1,$ %
\begin{equation}\label{eq:descent}
    \ee[\norm{F(\vu_k))}] \leq \frac{\Lambda_0}{k} + \Lambda_1\epsilon,
\end{equation}
where $\Lambda_{0} = 76 L\norm{\ui - \um}$ and $\Lambda_1 = 4 \sqrt{\frac{2}{{3}}}$. 
As a result, stochastic Halpern iteration from Eq.~\eqref{eq:Halpern} returns a point $\vu_k$ such that $\ee[\norm{F(\vu_k)}] \leq 4\epsilon$ after at most $N = \lceil \frac{2\Lambda_0}{\epsilon}\rceil = \co\big(\frac{L\|\vu_0 -\vu^*\|}{\epsilon}\big)$ iterations.
\end{restatable}

{We remark that the previous result states an iteration complexity bound under a rather high accuracy assumption for the operator estimators at each iteration. In order to attain these accuracy requirements, we could either use a minibatch at every iteration, or use variance reduction. In what follows we explore both approaches.} We further remark that we made no effort to optimize the constants in the bound above, and thus the constants are likely improvable. 

Finally, observe that due to the required low error for the estimates $\ee[\|F(\vu_k) - \tF(\vu_k)\|^2] \leq\frac{\epsilon^2}{k},$ we can certify by Chebyshev bound that $\mathbb{P}[\|F(\vu_k) - \tF(\vu_k)\| \geq \epsilon] \leq \frac{1}{k}.$ In particular,
after $O(\frac{1}{\epsilon})$ iterations, if we have $\|\tilde{F}(\vu_k)\| \leq \epsilon$ (which holds in expectation), then $\|F(\vu_k)\|$ is also $O(\epsilon)$ with probability at least $1 - \epsilon$. This is particularly important for practical implementations, where a stopping criterion can be based on the value of $\|\tF(\vu_k)\|$, which, unlike $\|F(\vu_k)\|$, can be efficiently evaluated.

\subsection{Stochastic Oracle Complexity With a Simple Mini-batch Estimate}

A direct consequence of Theorem~\ref{thm:cocoercive} is that a simple estimator $\tF(\vu_k) = \frac{1}{S_k}\sum_{i=1}^{S_k} \hF(\vu_k, z_{i}^{(k)})$ leads to the overall $\co(\frac{1}{\epsilon^4})$ oracle complexity, as stated in the following corollary of Theorem~\ref{thm:cocoercive}. 

\begin{restatable}{corollary}{complexityCoco}
Under the assumptions of Theorem~\ref{thm:cocoercive}, if $\tF(\vu_k) = \frac{1}{S_k}\sum_{i=1}^{S_k} \hF(\vu_k, z_{i}^{(k)})$, where $\hF(\vu_k, z_{i}^{(k)})$ satisfies Assumption~\ref{assmpt:stoch-queries} and $z^{(k)}_i \overset{\text{i.i.d.}}{\sim} P_z$, then setting $S_k = \frac{\sigma^2(k+1)}{ \epsilon^2}$ for all $k \geq 0$ guarantees that $\ee[\norm{F(\vu_k)}] \leq 4\epsilon$ after at most $\co\big(\frac{\sigma^2 L^2\|\vu_0 - \vu^*\|^2}{\epsilon^4}\big)$ queries to $\hF$. 
\end{restatable}
\begin{proof}
The averaged operator from the theorem statement is unbiased, by Assumption~\ref{assmpt:stoch-queries}. Further, as by Assumption~\ref{assmpt:stoch-queries}, $\|F(\vu_k) - \hF(\vu_k, z_{i}^{(k)})\|^2 \leq \sigma^2,$ it immediately follows that $\|F(\vu_k) - \tF(\vu_k)\|^2 \leq \frac{\sigma^2}{S_k} = \frac{\epsilon^2}{k + 1}$. Applying Theorem~\ref{thm:cocoercive}, the total number of iterations $N$ of Halpern iteration until $\ee[\|F(\vu_N)\|] \leq 4\epsilon$ is $N = \co(\frac{L\|\vu_0 - \vu^*\|}{\epsilon}).$ To complete the proof, it remains to bound the total number of oracle queries $\hF$ to $F,$ which is simply $\sum_{k=0}^N S_k = \co\big(\frac{N^2\sigma^2}{\epsilon^2}\big) = \co\big(\frac{\sigma^2 L^2 \|\vu_0 - \vu^*\|^2}{\epsilon^4}\big).$  
\end{proof}



\subsection{Improved Oracle Complexity via Variance Reduction}

We now consider using the recursive variance reduction method from Eq.~\eqref{eq:PAGE} to obtain the variance bound required in Theorem~\ref{thm:cocoercive}. 
The algorithm with all its corresponding parameter settings is summarized in Algorithm~\ref{alg:cocoercive}. Of course, in practice, $\|\vu_0 - \vu^*\|$ is not known, and instead of running the algorithm for a fixed number of iterations $N,$ one could run it, for example, until reaching a point with $\|\tF(\vu_k)\| \leq \epsilon.$ {Notice that convergence is guaranteed by Theorem~\ref{thm:cocoercive}; however it does not directly address the problem of the oracle complexity (as batch sizes depend on successive iterate distances). To resolve this issue, we first provide a bound on $\norm{\vu_k - \vu_{k - 1}}$.}

\begin{algorithm}
\caption{Stochastic Halpern-Cocoercive}\label{alg:cocoercive}

\textbf{Input} $\vu_0 \in \rr^d,$ $\|\vu_0 - \vu^*\|,$ $L,$ $\epsilon > 0$, $\sigma$\;

\textbf{Initialization: $\Lambda_0 = \frac{76 L \|\vu_0 - \vu^*\|}{\epsilon}$, $N= \lceil\frac{2\Lambda_0}{\epsilon}\rceil$, $S_1^{(0)} = \lceil\frac{8\sigma^2}{\epsilon^2}\rceil$}\;

$\tF(\vu_0) = \frac{1}{S_1^{(0)}}\sum_{i=1}^{S_1^{(0)}}\hF(\vu_0, z_i^{(0)})$\;

\For{$k = 1:N$}
{
$\vu_k = \frac{1}{k+1}\vu_0 + \frac{k}{k+1}\big(\vu_{k-1} - \frac{1}{L}\tF(\vu_{k-1})\big)$\;

$p_k = \frac{2}{k+1},$ $S_1^{(k)} = \lceil\frac{8\sigma^2}{p_k\epsilon^2}\rceil,$ $S_2^{(k)} = \lceil \frac{8L^2\|\vu_k - \vu_{k-1}\|^2}{{p_k}^2 \epsilon^2}\rceil$\;

Compute $\tF(\vu_k)$ based on Eq.~\eqref{eq:PAGE}
}
\textbf{Return:} $\vu_N$
\end{algorithm}
\begin{restatable}{lemma}{ukCoco}
\label{lemma:uk}
Given an arbitrary initial point $\vu_0 \in \rr^d,$ let $\{ \vu_k \}_{k \geq 1}$ be the sequence of points produced by Algorithm~\ref{alg:cocoercive}. 
Assume further that $\lambda_k = \frac{1}{k + 1}$, $L_k = 2L$ for all $k \geq 0$. 
Then, 
\begin{equation}\label{eq:ukOne}
  \norm{\vu_k - \vu_{k - 1}}^2 \leq \begin{cases}
     \frac{1}{4L^2}\|{\tF(\vu_0)}\|^2 & \text{if }  k = 1, \\
     \frac{2k^2}{L^2(k + 1)^2}\|{\tF(\vu_{k - 1})}\|^2 + \sum_{i = 0}^{k - 2}\frac{2(i + 1)^2}{k(k + 1)^2L^2}\|{\tF(\vu_i)}\|^2 & \text{if } k \geq 2.
\end{cases}
\end{equation}
Moreover, 
if for $1 \leq i \leq k-1$, all of the following conditions hold (same as in Theorem~\ref{thm:cocoercive}): (i) $\ee[\|{F(\vu_i)}\|] \leq \frac{\Lambda_0}{i} + \Lambda_1\epsilon,$ 
where $\Lambda_{0} = 76 L\norm{\ui - \um}$ and $\Lambda_1 = 4 \sqrt{\frac{2}{{3}}}$, (ii)  $\ee\big[\big\|{F(\vu_i) - \tF(\vu_i)}\big\|^2\big] \leq \frac{\epsilon^2}{i}$, and (iii) $\epsilon \leq \frac{\Lambda_0}{k}$, then  
$\ee[\norm{\vu_k - \vu_{k - 1}}^2] = \co\Big(\frac{\norm{\vu_0 - \vu^*}^2}{k^2}\Big).$
\end{restatable}
\begin{proof}
For $k = 1$, $\vu_1 = \frac{1}{2}\vu_0 + \frac{1}{2}\big(\vu_0 - \frac{1}{L}\tF(\vu_0)\big)$, which leads to $\|{\vu_1 - \vu_0}\|^2 = \big\|{-\frac{1}{2L}\tF(\vu_0)}\big\|^2 = \frac{1}{4L^2}\big\|{\tF(\vu_0)}\big\|^2$. For $k \geq 2$,  recursively applying Eq.~\eqref{eq:Halpern}, we have $\vu_k - \vu_{k - 1} = \lambda_k(\vu_0 - \vu_{k - 1}) - \frac{1 - \lambda_k}{L}\tF(\vu_{k - 1}) = \lambda_k(1 - \lambda_{k - 1})(\vu_0 - \vu_{k - 2}) + \frac{\lambda_k(1 - \lambda_{k - 1})}{L}\tF(\vu_{k - 2}) - \frac{1 - \lambda_{k}}{L}\tF(\vu_{k - 1})$, leading to 
\begin{equation*}
\vu_k - \vu_{k - 1} = - \frac{1 - \lambda_k}{L}\tF(\vu_{k - 1}) + \sum_{i = 0}^{k - 2}\frac{\lambda_k}{L}\Big(\prod_{j = i + 1}^{k - 1}(1 - \lambda_j)\Big)\tF(\vu_i).
\end{equation*}
%
Recalling that $\lambda_k = \frac{1}{k + 1}$, we have $\norm{\vu_k - \vu_{k - 1}}^2 = \norm{- \frac{k}{L(k + 1)}\tF(\vu_{k - 1}) + \sum_{i = 0}^{k - 2}\frac{i + 1}{k(k + 1)L}\tF(\vu_i)}^2$, which gives us Inequality~\eqref{eq:ukOne} by applying a generalized variant of Young's inequality $\norm{\sum_{i = 1}^{K}X_i}^2 \leq \sum_{i = 1}^{K}K \norm{X_i}^2$ twice (first to the sum of $- \frac{k}{L(k + 1)}\tF(\vu_{k - 1})$ and the summation term, then to the summation term, while noticing that $\frac{k-1}{k} \leq 1$). 

For the second claim, by the lemma assumptions and the analysis in the proof for Theorem~\ref{thm:cocoercive}, 
we have $\ee[\|{F(\vu_i)}\|^2] = \co(\frac{L^2\norm{\vu_0 - \vu^*}^2}{i^2})$
for $i \leq k-1 \leq \co\big(\frac{1}{\epsilon}\big)$, thus $\ee[\|{\tF(\vu_i)}\|^2] \leq 2\ee[\|{F(\vu_i)}\|]^2 + 2\ee[\|{F(\vu_i) - \tF(\vu_i)}\|^2] = \co(\frac{L^2\|{\vu_0 - \vu^*}\|^2}{i^2})$. Plugging this bound into Inequality~\eqref{eq:ukOne}, we get $\ee[\|{\vu_k - \vu_{k - 1}}\|^2] = \co(\frac{\|{\vu_0 - \vu^*}\|^2}{k^2})$. 
\end{proof}

Using Lemma~\ref{lemma:uk} and making the appropriate parameter settings for the estimator from Eq.~\eqref{eq:PAGE}, it is now possible to apply Theorem~\ref{thm:cocoercive} to obtain the improved ${\cal O}(\frac{1}{\epsilon^3})$ stochastic oracle complexity bound, as stated in the following corollary of Theorem~\ref{thm:cocoercive}. 

\begin{restatable}{corollary}{complexityCocoVR}
\label{thm:complexity}
Given arbitrary $\vu_0 \in \rr^d$ and $\epsilon > 0,$ consider $\vu_N$ returned by Algorithm~\ref{alg:cocoercive}. 
Then, $\ee[\norm{F(\vu_N)}] \leq 4\epsilon$ with 
expected
$\co(\frac{\sigma^2L\norm{\vu_0 - \vu^*} + L^3\norm{\vu_0 - \vu^*}^3}{\epsilon^3})$ oracle queries to $\hF$.
 %
\end{restatable}
\begin{proof}
Let $m_k$ be the number of stochastic queries made by the estimator from Eq.~\eqref{eq:PAGE} in iteration $k$. Using Corollary~\ref{cor:variance-bound}, we have 
\begin{equation*}
\begin{aligned}
    \ee\big[m_{k + 1} | \cf_{k-1}\big] = \;& p_k S_1^{(k)} + 2(1 - p_k)S_2^{(k)}  = p_k\big\lceil\textstyle\frac{8\sigma^2}{p_k\epsilon^2}\big\rceil + 2(1 - p_k)\big\lceil\textstyle\frac{8L^2\norm{\vu_k - \vu_{k - 1}}^2}{p_k^2\epsilon^2}\big\rceil,
\end{aligned}
\end{equation*}
where the first equality holds because $S_2^{(k)}$ is measurable w.r.t.~$\cf_{k-1}$ and the only random choice that remains is the selection of the estimator in Eq.~\eqref{eq:PAGE} determined by probabilities $p_k$ and $1-p_k.$ 

Taking expectation with respect to all randomness on both sides, rearranging the terms, and using the fact that $\lceil x \rceil \leq x + 1$ for any $x \in \rr$, we obtain $\ee[m_{k + 1}] \leq \frac{8\sigma^2}{\epsilon^2} + \frac{16(1 - p_k)L^2\ee[\norm{\vu_k - \vu_{k - 1}}^2]}{p_k^2\epsilon^2} + 2$. Recalling that $p_k = \frac{2}{k + 1} = \co(\frac{1}{k})$ and $\ee[\norm{\vu_k - \vu_{k - 1}}^2] = \co\big(\frac{\norm{\vu_0 - \vu^*}^2}{k^2}\big)$ by Lemma~\ref{lemma:uk}, it follows that  $\ee[m_{k + 1}] = \co\big(\frac{\sigma^2 + L^2\norm{\vu_0 - \vu^*}^2}{\epsilon^2}\big)$. As, by Theorem~\ref{thm:cocoercive}, the total number of iterations to attain $4\epsilon$ norm of the operator in expectation is $N = \left\lceil \frac{2\Lambda_0}{\epsilon} \right\rceil = \co\big(\frac{L\norm{\vu_0 - \vu^*}}{\epsilon}\big)$ and $m_0 = S_1^{(0)} = \co\big(\frac{\sigma^2}{\epsilon^2}\big)$,  the total number of queries to $\hF$ is $\ee[M] = \ee[\sum_{k = 1}^{N}m_k] = \co\big(\frac{\sigma^2L\norm{\vu_0 - \vu^*} + L^3\norm{\vu_0 - \vu^*}^3}{\epsilon^3}\big)$.
\end{proof}

We note in passing that the running time guarantee of this algorithm is of Las Vegas-type: despite its iteration number being surely bounded by $\big\lceil \frac{2\Lambda_0}{\epsilon} \big\rceil = \co\big(\frac{L\|\vu_0 -\vu^*\|}{\epsilon}\big)$, the batch sizes (in particular $S_2^{(k)}$) are random, and are not universally bounded. 
We further argue that Algorithm~\ref{alg:cocoercive} can be extended to {\bf constrained settings} by defining the operator mapping as in \citet{diakonikolas2020halpern} and modifying the variance-reduced stochastic estimator accordingly based on the projection of $\tF$. We show that the newly defined operator mapping is also cocoercive while the variance of the modified estimator is bounded by the variance of $\tF$, so arguments from Theorem~\ref{thm:cocoercive} and Corollary~\ref{thm:complexity} extend to this case. This modified estimator need not be unbiased (as neither is $\tF$); however, this is irrelevant to our analysis as it does not require unbiasedness. 
For completeness, a detailed extension to the constrained case is provided in Appendix~\ref{appx:omitted-coco-constrained}.


\section{Monotone and Lipschitz Setup} \label{sec:two_step_Halpern}

Throughout this section, we assume that $F$ is monotone and $L$-Lipschitz. While the previous section addresses the cocoercive setup using the classical version of Halpern iteration adapted to cocoercive operators, it is unclear how to directly generalize this result to the setting with monotone Lipschitz operators. In the deterministic setting, generalization to monotone Lipschitz operators can be achieved through the use of a resolvent operator (see \citet{diakonikolas2020halpern}). However, such an approach incurs an additional $\log(1/\epsilon)$ factor in the iteration complexity coming from approximating the resolvent and it is further unclear how to generalize it to stochastic settings, as the properties of the stochastic estimate $\tF$ of $F$ do not readily translate into the same or similar properties for the resolvent of $\tF.$ 
Instead of taking the approach based on the resolvent, we  consider a recently proposed two-step variant of Halpern iteration~\citep{tran2021halpern}, adapted here to the stochastic setting. The variant uses extrapolation and  is defined by 
\begin{equation}\label{eq:mono-uk}
\left\{\begin{aligned}
\vv_{k} \quad &:=\lambda_{k} \vu_{0}+\left(1-\lambda_{k}\right) \vu_{k}-\eta_{k} \tF(\vv_{k-1}), \\
\vu_{k+1} &:=\lambda_{k} \vu_{0}+\left(1-\lambda_{k}\right) \vu_{k}-\eta_{k} \tF(\vv_{k}),
\end{aligned}\right.
\end{equation}
where $\lambda_k \in \left[0, 1\right)$, $\eta_k > 0$, and $\tF$ is defined by~\eqref{eq:PAGE}. The resulting algorithm with a complete parameter setting is provided in Algorithm~\ref{alg:monotone}.

\begin{algorithm}
\caption{Extrapolated Stochastic Halpern-Monotone (E-Halpern)}\label{alg:monotone}

\textbf{Input:} $\vu_0 \in \rr^d,$ $\|\vu_0 - \vu^*\|,$ $0 < \eta_0 \leq \frac{1}{3 \sqrt{3} L},$ $L,$ $\epsilon > 0$, $\sigma$\;

\textbf{Initialize:} $\vv_{-1} = \vu_0,$ $S_1^{(-1)} = S_1^{(0)} = \lceil\frac{8\sigma^2}{\epsilon^2}\rceil$, $M = 9L^2$, $\underline{\eta} = \frac{\eta_{0}(1-2 M \eta_{0}^{2})}{1-M \eta_{0}^{2}}$\;

Set $\Lambda_0 = \frac{4(L^2\eta_0\underline{\eta} + 1)\norm{\vu_0 - \vu^*}^2}{\underline{\eta}^2}$, $\Lambda_1 = \frac{5\left(1 + M\underline{\eta}\eta_0\right)}{M\underline{\eta}^2}$, $N= \big\lceil\frac{\sqrt{\Lambda_0}}{\sqrt{\Lambda_1}\epsilon}\big\rceil$\;

$\tF(\vv_{-1}) = \frac{1}{S_1^{(-1)}}\sum_{i=1}^{S_1^{(-1)}}\hF(\vv_{-1}, z_i^{(-1)})$, where $z_i^{(-1)} \stackrel{\text{i.i.d.}}{\sim} {\cal P}_z$\; 

\For{$k = 1:N$}
{
$\vv_{k-1} = \frac{1}{k+1} \vu_{0} + \frac{k}{k+1} \vu_{k-1}-\eta_{k-1} \tF(\vv_{k-2})$\;

$p_{k - 1} = \min(\frac{2}{k}, 1),$ $S_1^{(k - 1)} = \lceil\frac{8\sigma^2}{p_{k - 1}\epsilon^2}\rceil,$ $S_2^{(k - 1)} = \lceil \frac{8L^2\|\vv_{k - 1} - \vv_{k-2}\|^2}{{p_{k - 1}}^2 \epsilon^2}\rceil$\;

Compute $\tF(\vv_{k - 1})$ based on Eq.~\eqref{eq:PAGE}\;

$\vu_{k} = \frac{1}{k+1} \vu_{0} + \frac{k}{k+1}\vu_{k-1}-\eta_{k-1} \tF(\vv_{k-1})$\;

$\eta_k = \frac{(1 - \frac{1}{(k+1)^2} - M{\eta_{k-1}}^2 )(k+1)^2}{(1- M{\eta_{k-1}}^2)k (k+2)}\eta_{k-1}$
 
}
\textbf{Return:} $\vu_N$
\end{algorithm}

To analyze the convergence of the extrapolated Halpern variant from Eq.~\eqref{eq:mono-uk}, we use the potential function $\cv_k = A_{k}\|F(\vu_{k})\|^{2} + B_{k}\innp{F(\vu_{k}), \vu_{k} - \vu_{0}} + c_{k} L^{2}\|\vu_{k} - \vv_{k-1}\|^{2}$, previously used by~\citet{tran2021halpern}, 
%
where $A_k$, $B_k$ and $c_k$ are positive parameters to be determined later. Observe that this is essentially the same potential function as $\cc_k,$ corrected by the quadratic term $c_{k} L^{2}\|\vu_{k} - \vv_{k-1}\|^{2}$ to account for error terms appearing in the analysis of the two-step variant from Eq.~\eqref{eq:mono-uk}. Similarly as in the cocoercive setup, the potential function is not monotonically non-increasing, due to the error terms that arise due to the stochastic access to $F.$ Bounding these error terms requires a careful technical argument,  
and is the main technical contribution of this section. Due to space constraints, the complete technical argument is deferred to Appendix~\ref{appx:monotone}, while the main results are stated below.

\begin{restatable}{theorem}{rateMono}
\label{thm:mono-rate}
Given an arbitrary initial point $\vu_0 \in \rr^d$ 
and target error $\epsilon>0$, assume that the iterates $\vu_k$ evolve according to Algorithm~\ref{alg:monotone} for $k \geq 1$. Then, for all $k \geq 2,$
\begin{equation}\label{ineq:mono-rate}
\begin{aligned}
\ee\left[\norm{F(\vu_k)}^2 + 2L^2\norm{\vu_k - \vv_{k - 1}}^2\right] \leq \;& \frac{\Lambda_0}{(k + 1)(k + 2)} + \Lambda_1\epsilon^2,
\end{aligned}
\end{equation}
where $\Lambda_0 = \frac{4(L^2\eta_0\underline{\eta} + 1)\norm{\vu_0 - \vu^*}^2}{\underline{\eta}^2}$ and $\Lambda_1 = \frac{5\left(1 + M\underline{\eta}\eta_0\right)}{M\underline{\eta}^2}$. 
In particular, 
$\ee\big[\norm{F(\vu_N)}^2 + 2L^2\norm{\vu_N - \vv_{N - 1}}^2\big] \leq 2\Lambda_1\epsilon^2 = \co(\epsilon^2)$ after at most $N = \big\lceil\frac{\sqrt{\Lambda_0}}{\sqrt{\Lambda_1}\epsilon}\big\rceil = \co\big(\frac{L\norm{\vu_0 - \vu^*}}{\epsilon}\big)$ iterations. The total number of oracle queries to $\hF$ is $\co\big(\frac{\sigma^2L\norm{\vu_0 - \vu^*} + L^3\norm{\vu_0 - \vu^*}^3}{\epsilon^3}\big)$ 
in expectation. 
\end{restatable}



\section{Faster Convergence Under a Sharpness Condition} \label{sec:str_monotone}

We now show that by restarting Algorithm~\ref{alg:monotone}, we can achieve the $\co\left(\frac{1}{\epsilon^2}\log\frac{1}{\epsilon}\right)$ oracle complexity under a milder than strong monotonicity {\em $\mu$-sharpness condition}: for all $\vu \in \cu$, $\innp{F(\vu) - F(\vu^*), \vu - \vu^*} \geq \mu\norm{\vu - \vu^*}^2$. The scheme is summarized in Algorithm~\ref{alg:sharp}, and the proof is deferred to Appendix~\ref{appx:str_monotone}.

\begin{algorithm}[ht]
\caption{Restarted Extrapolated Stochastic Halpern-Sharp (Restarted E-Halpern)
}\label{alg:sharp}

\textbf{Input:} $\vv_{-1} = \vu_0 \in \rr^d,$ $\|\vu_0 - \vu^*\|,$ $0 < \eta_0 \leq \frac{1}{3\sqrt{3}L},$ $L,$ $\mu$, $\epsilon > 0$, $\sigma$\;

\textbf{Initialize:} $M = 9L^2,$ $\underline{\eta}=\frac{\eta_{0}\left(1-2 M \eta_{0}^{2}\right)}{1-M \eta_{0}^{2}},$ $N= \Big\lceil\log\Big(\frac{\sqrt{6}\norm{\vu_0 - \vu^*}}{2\epsilon}\Big)\Big\rceil$\;

\For{$k = 1:N$}
{

Call Algorithm~\ref{alg:monotone} with initialization $\vv_{-1}^{(k)} = \vu_0^{(k)} = \vu_{k - 1},$ $\epsilon_k = \frac{\mu\epsilon\sqrt{M\underline{\eta}^2}}{2\sqrt{5\left(1 + M\underline{\eta}\eta_0\right)}},$ and $S_1^{(-1)} = S_1^{(0)} = \lceil\frac{8\sigma^2}{\epsilon_k^2}\rceil$, for $K = \Big\lceil\frac{4\sqrt{L^2\eta_0\underline{\eta} + 1}}{\mu\underline{\eta}}\Big\rceil$ iterations, and return $\vu_k$\;

}
\textbf{Return:} $\vu_N$
\end{algorithm}


\begin{restatable}{theorem}{complexityStrong}\label{thm:complexity-sharp}
Given $F$ that is $L$-Lipschitz and $\mu$-sharp and the precision parameter $\epsilon$, Algorithm~\ref{alg:sharp} outputs $\vu_N$ with $\ee[\norm{\vu_N - \vu^*}^2] \leq \epsilon^2$ as well as $\ee\big[\norm{F(\vu_N)}^2\big] \leq L^2\epsilon^2$ after $N=\co\left(\frac{L}{\mu}\log\frac{\norm{\vu_0 - \vu^*}}{\epsilon}\right)$ 
iterations with at most $\co\Big(\frac{\sigma^2(\mu + L)\log(\norm{\vu_0 - \vu^*}/ \epsilon) + L^3\norm{\vu_0 - \vu^*}^2}{\mu^3\epsilon^2}\Big)$ queries to $\hF$
in expectation.
\end{restatable}

\section{Numerical experiments and discussion}\label{sec:num-exp}
We now illustrate the empirical performance of stochastic Halpern iteration on robust least square problems. 
Specifically, given data matrix $\mA \in \mathbb{R}^{n \times d}$ and noisy observation vector $\vb \in \mathbb{R}^n$ subject to bounded deterministic perturbation $\delta$ with $\norm{\delta} \leq \rho$, robust least square (RLS) minimizes the worst-case residue as $\min_{\vx \in \mathbb{R}^d} \max_{\delta:\norm{\delta} \leq \rho} \norm{\mA\vx  - \vy}^2_2$ with $\vy = \vb + \delta$~\citep{el1997robust}. We consider solving \ref{def:MI} induced from RLS with Lagrangian relaxation
where $\vu = (\vx, \vy)^T$ and $F(\vu) = \big(\nabla_{\vx}L_{\lambda}(\vx, \vy), -\nabla_{\vy}L_{\lambda}(\vx, \vy)\big)^T$ for $L_{\lambda}(\vx, \vy) = \frac{1}{2n}\norm{A\vx - \vy}^2_2 - \frac{\lambda}{2n}\norm{\vy - \vb}_2^2$. 
We use a real-world superconductivity dataset \citep{HAMIDIEH2018346} from UCI Machine Learning Repository \citep{Dua:2019} for our experiment, which is of size $21263 \times 81$. 
To ensure the problem is concave in $\vy,$ we need that $\lambda >1;$ in the experiments, we set $\lambda = 1.5$.
\begin{figure}[!ht]
    \centering
    \subfigure[Comparison on superconductivity dataset.]{\includegraphics[width=0.48\textwidth]{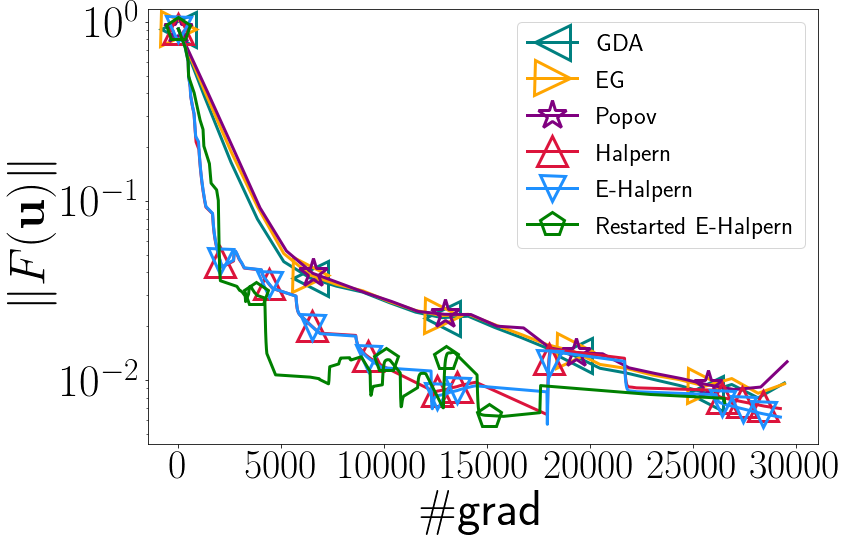}\label{fig:exp1}}
    \subfigure[E-Halpern with different stochastic estimators.]{\includegraphics[width=0.48\textwidth]{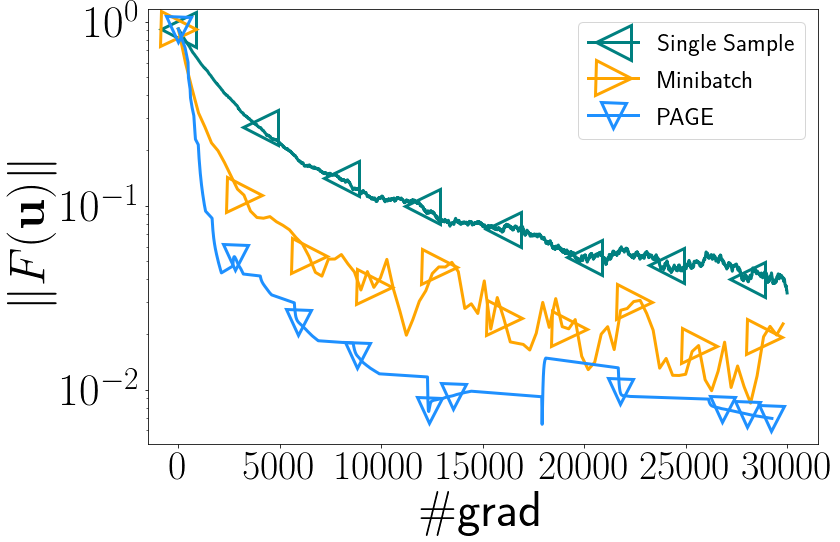}\label{fig:exp2}}
    \caption{Empirical comparison of min-max algorithms on the robust least squares problem.
    }
    \label{fig:exp}
\end{figure}
For the experiment, we compare Halpern, E-Halpern, and Restarted E-Halpern algorithms with gradient descent-ascent (GDA), extragradient (EG)~\citep{korpelevich1977extragradient}, and Popov's method~\citep{Popov1980}
in stochastic settings. Even though our theoretical results for Restarted E-Halpern require scheduled restarts based on known problem parameters, in the implementation, to avoid complicated parameter tuning and illustrate empirical performance, we restart E-Halpern whenever the norm of stochastic estimator $\tF$ used in E-Halpern halves. All Halpern variants are implemented with PAGE estimator considered in our paper; all other algorithms are implemented using minibatches. Additionally, we compare E-Halpern with the PAGE estimator against E-Halpern with  single-sample and mini-batch estimators. 

We report and plot the (empirical) operator norm $\|F(\vu)\|$ against the number of stochastic operator evaluations. 
Note that evaluations of $\|F(\vu)\|$ are only used for plotting but not for running any of the algorithms. 
We use the same random initialization and tune 
the batch sizes and step sizes 
(to the values achieving fastest convergence under noise) for each method by grid search. We use constant batch sizes and constant step sizes for GDA, EG, and Popov. We also choose the batch sizes of PAGE estimator to ensure $\ee[\|F(\vu_k) - \tF(\vu_k)\|^2] \leq \co(\frac{1}{k})$, which handles error accumulation \citep{lee2021fast} and early stagnation of stochastic Halpern iteration. We implement all the algorithms in Python and run each algorithm using one CPU core on a macOS machine with Intel 2.3GHz Dual Core i5 Processor and 8GB RAM.\footnote{Code is available at \href{https://github.com/zephyr-cai/Halpern}{https://github.com/zephyr-cai/Halpern}.} 

We observe that (i) in Figure~\ref{fig:exp1} both Halpern and E-Halpern exhibit faster convergence to approximate stationary points (with much smaller gradient norm after same number of gradient evaluations) than other algorithms, and restarting E-Halpern  provides additional speedup, validating our theoretical insights; 
(ii) in Figure~\ref{fig:exp2}, E-Halpern with PAGE estimator displays faster convergence compared to other two estimators, in agreement with our theoretical analysis.

\section{(Near) Tightness of Stochastic Oracle Complexity Bounds}\label{appx:lower-bound-reductions}
In this section, we briefly discuss lower bound reductions which imply that our results for Lipschitz sharp setups are unimprovable in terms of the dependence on $\epsilon.$
To keep the discussion simple, we only focus on the $\epsilon$ dependence here and unconstrained settings. The near-optimality of our bounds 
is implied by the known lower bound for the optimality gap in $L$-smooth $\mu$-strongly convex stochastic optimization, which is of the order $\Omega(\frac{\sigma^2}{\mu\epsilon})$ in the high noise $\sigma^2$ or low error $\epsilon$ regimes; see, for example, the discussion in~\citet{ghadimi2016accelerated} (the omitted part of the lower bound comes from the deterministic complexity of smooth strongly convex optimization and is less interesting in our context).  The same lower bound implies a lower bound of $\Omega(\frac{\sigma^2}{\epsilon^2})$ for minimizing the gradient of a smooth strongly convex function $f$. Suppose not (for the purpose of contradiction); i.e., suppose that there were an algorithm that constructs a point $\vx$ with $\ee[\|\nabla f(\vx)\|^2] \leq \Bar{\epsilon}^2$ in $o(\frac{\sigma^2}{\bar{\epsilon}^2})$ oracle queries to the stochastic gradient. By $\mu$-strong convexity of $f,$ this would imply that we get $\ee[f(\vx) - \min_{\vu}f(\vu)] \leq \frac{1}{2\mu}\ee[\|\nabla f(\vx)\|^2] \leq \frac{\bar{\epsilon}^2}{2\mu}$ with $o(\frac{\sigma^2}{\bar{\epsilon}^2})$ oracle queries to the stochastic gradient. Setting $\bar{\epsilon} = \sqrt{{\epsilon}{\mu}},$ we get that this would imply oracle complexity $o(\frac{\sigma^2}{\mu \epsilon})$, and we reach a contradiction on the lower bound for the optimality gap. 

Hence, $\Omega(\frac{\sigma^2}{\epsilon^2})$ lower bound applies to the minimization of the gradient of smooth strongly convex functions in stochastic regimes. 
Observe that the gradients of smooth strongly convex functions are Lipschitz and strongly monotone (thus also sharp), so a lower bound for this problem class implies a lower bound for the class of sharp Lipschitz monotone inclusion problems. Thus, we can conclude that our result from Section~\ref{sec:str_monotone} for sharp Lipschitz monotone inclusion problems that gives $\co\Big(\frac{\sigma^2(\mu + L)\log(\norm{\vu_0 - \vu^*}/ \epsilon) + L^3\norm{\vu_0 - \vu^*}^2}{\mu^3\epsilon^2}\Big)$ stochastic oracle complexity  is near-optimal in terms of the dependence on $\sigma$ and $\epsilon$ (but likely not near-optimal in terms of the dependence on the remaining problem parameters).


\section{Conclusion}
We introduced stochastic variance reduced variants of Halpern iteration for addressing monotone inclusion problems. Our work addresses all standard classes of Lipschitz monotone problems and achieves improved stochastic oracle complexity guarantees, all for the last iterate. Subsequent to this work, \cite{chen2022near} obtained near-optimal bounds for the cases considered in this work, by reducing the Lipschitz monotone case to the Lipschitz strongly monotone case, using regularization. It is an open question to obtain such near-optimal bounds with a direct method, without the use of regularization.

\section*{Acknowledgements}

XC and CS were supported in part by the NSF grant 2023239. 
CG's research was partially supported by
INRIA Associate Teams project, FONDECYT 1210362 grant, ANID Anillo ACT210005 grant, and National Center for Artificial Intelligence CENIA FB210017, Basal ANID. Part of this work was done while CG was at the University of Twente. JD was supported by the NSF grant 2007757, by the Office of Naval Research under contract number N00014-22-1-2348, and by the Wisconsin Alumni Research Foundation. Part of this work was done while JD and CS were visiting Simons Institute for the Theory of Computing.

\bibliographystyle{plainnat}
\bibliography{reference}

\begin{thebibliography}{55}
\providecommand{\natexlab}[1]{#1}
\providecommand{\url}[1]{\texttt{#1}}
\expandafter\ifx\csname urlstyle\endcsname\relax
  \providecommand{\doi}[1]{doi: #1}\else
  \providecommand{\doi}{doi: \begingroup \urlstyle{rm}\Url}\fi

\bibitem[Alacaoglu and Malitsky(2022)]{alacaoglu2021stochastic}
Ahmet Alacaoglu and Yura Malitsky.
\newblock Stochastic variance reduction for variational inequality methods.
\newblock In \emph{Proc.~COLT'22}, 2022.

\bibitem[Allen-Zhu(2017)]{allen2017katyusha}
Zeyuan Allen-Zhu.
\newblock Katyusha: {The} first direct acceleration of stochastic gradient
  methods.
\newblock \emph{The Journal of Machine Learning Research}, 18\penalty0
  (1):\penalty0 8194--8244, 2017.

\bibitem[Arjevani et~al.(2020)Arjevani, Carmon, Duchi, Foster, Sekhari, and
  Sridharan]{arjevani2020second}
Yossi Arjevani, Yair Carmon, John~C Duchi, Dylan~J Foster, Ayush Sekhari, and
  Karthik Sridharan.
\newblock Second-order information in non-convex stochastic optimization: Power
  and limitations.
\newblock In \emph{Proc.~COLT'20}, 2020.

\bibitem[Beck(2017)]{beck2017first}
Amir Beck.
\newblock \emph{First-order methods in optimization}, volume~25.
\newblock SIAM, 2017.

\bibitem[Carmon et~al.(2019)Carmon, Jin, Sidford, and Tian]{carmon2019variance}
Yair Carmon, Yujia Jin, Aaron Sidford, and Kevin Tian.
\newblock Variance reduction for matrix games.
\newblock In \emph{Proc.~NeurIPS'19}, 2019.

\bibitem[Chavdarova et~al.(2019)Chavdarova, Gidel, Fleuret, and
  Lacoste-Julien]{chavdarova2019reducing}
Tatjana Chavdarova, Gauthier Gidel, Fran{\c{c}}ois Fleuret, and Simon
  Lacoste-Julien.
\newblock Reducing noise in {GAN} training with variance reduced extragradient.
\newblock In \emph{Proc.~NeurIPS'19}, 2019.

\bibitem[Chen and Luo(2022)]{chen2022near}
Lesi Chen and Luo Luo.
\newblock Near-optimal algorithms for making the gradient small in stochastic
  minimax optimization.
\newblock \emph{arXiv preprint arXiv:2208.05925}, 2022.

\bibitem[Cheval et~al.(2022)Cheval, Kohlenbach, and
  Leustean]{cheval2022modified}
Horatiu Cheval, Ulrich Kohlenbach, and Laurentiu Leustean.
\newblock On modified {Halpern} and {Tikhonov-Mann} iterations.
\newblock \emph{arXiv preprint arXiv:2203.11003}, 2022.

\bibitem[Contreras and Cominetti(2021)]{Contreras:2021}
Juan~Pablo Contreras and Roberto Cominetti.
\newblock Optimal error bounds for nonexpansive fixed-point iterations in
  normed spaces.
\newblock \emph{arXiv preprint, arXiv:2108.10969}, 2021.

\bibitem[Cutkosky and Orabona(2019)]{cutkosky2019momentum}
Ashok Cutkosky and Francesco Orabona.
\newblock Momentum-based variance reduction in non-convex {SGD}.
\newblock In \emph{Proc.~NeurIPS'19}, 2019.

\bibitem[Defazio et~al.(2014)Defazio, Bach, and
  Lacoste-Julien]{defazio2014saga}
Aaron Defazio, Francis Bach, and Simon Lacoste-Julien.
\newblock {SAGA}: {A} fast incremental gradient method with support for
  non-strongly convex composite objectives.
\newblock In \emph{Proc.~NeurIPS'14}, 2014.

\bibitem[Diakonikolas(2020)]{diakonikolas2020halpern}
Jelena Diakonikolas.
\newblock Halpern iteration for near-optimal and parameter-free monotone
  inclusion and strong solutions to variational inequalities.
\newblock In \emph{Proc.~COLT'20}, 2020.

\bibitem[Diakonikolas and Wang(2022)]{diakonikolas2021potential}
Jelena Diakonikolas and Puqian Wang.
\newblock Potential function-based framework for minimizing gradients in convex
  and min-max optimization.
\newblock \emph{SIAM Journal on Optimization}, 32\penalty0 (3):\penalty0
  1668--1697, 2022.

\bibitem[Diakonikolas et~al.(2021)Diakonikolas, Daskalakis, and
  Jordan]{diakonikolas2021efficient}
Jelena Diakonikolas, Constantinos Daskalakis, and Michael Jordan.
\newblock Efficient methods for structured nonconvex-nonconcave min-max
  optimization.
\newblock In \emph{Proc.~AISTATS'21}, 2021.

\bibitem[Dua and Graff(2017)]{Dua:2019}
Dheeru Dua and Casey Graff.
\newblock {UCI} {Machine} {Learning} {Repository}, 2017.
\newblock URL \url{http://archive.ics.uci.edu/ml}.

\bibitem[El~Ghaoui and Lebret(1997)]{el1997robust}
Laurent El~Ghaoui and Herv{\'e} Lebret.
\newblock Robust solutions to least-squares problems with uncertain data.
\newblock \emph{SIAM Journal on Matrix Analysis and Applications}, 18\penalty0
  (4):\penalty0 1035--1064, 1997.

\bibitem[Facchinei and Pang(2003)]{facchinei2007finite}
Francisco Facchinei and Jong-Shi Pang.
\newblock \emph{Finite-dimensional variational inequalities and complementarity
  problems}.
\newblock Springer Science \& Business Media, 2003.

\bibitem[Fang et~al.(2018)Fang, Li, Lin, and Zhang]{fang2018spider}
Cong Fang, Chris~Junchi Li, Zhouchen Lin, and Tong Zhang.
\newblock Spider: {Near}-optimal non-convex optimization via stochastic
  path-integrated differential estimator.
\newblock In \emph{Proc.~NeurIPS'18}, 2018.

\bibitem[Ghadimi and Lan(2016)]{ghadimi2016accelerated}
Saeed Ghadimi and Guanghui Lan.
\newblock Accelerated gradient methods for nonconvex nonlinear and stochastic
  programming.
\newblock \emph{Mathematical Programming}, 156\penalty0 (1-2):\penalty0 59--99,
  2016.

\bibitem[Golowich et~al.(2020)Golowich, Pattathil, Daskalakis, and
  Ozdaglar]{golowich2020last}
Noah Golowich, Sarath Pattathil, Constantinos Daskalakis, and Asuman Ozdaglar.
\newblock Last iterate is slower than averaged iterate in smooth convex-concave
  saddle point problems.
\newblock In \emph{Proc.~COLT'20}, 2020.

\bibitem[Halpern(1967)]{halpern1967fixed}
Benjamin Halpern.
\newblock Fixed points of nonexpanding maps.
\newblock \emph{Bulletin of the American Mathematical Society}, 73\penalty0
  (6):\penalty0 957--961, 1967.

\bibitem[Hamidieh(2018)]{HAMIDIEH2018346}
Kam Hamidieh.
\newblock A data-driven statistical model for predicting the critical
  temperature of a superconductor.
\newblock \emph{Computational Materials Science}, 154:\penalty0 346--354, 2018.

\bibitem[Iusem et~al.(2017)Iusem, Jofr{\'e}, Oliveira, and
  Thompson]{iusem2017extragradient}
Alfredo~N Iusem, Alejandro Jofr{\'e}, Roberto~Imbuzeiro Oliveira, and Philip
  Thompson.
\newblock Extragradient method with variance reduction for stochastic
  variational inequalities.
\newblock \emph{SIAM Journal on Optimization}, 27\penalty0 (2):\penalty0
  686--724, 2017.

\bibitem[Johnson and Zhang(2013)]{johnson2013accelerating}
Rie Johnson and Tong Zhang.
\newblock Accelerating stochastic gradient descent using predictive variance
  reduction.
\newblock In \emph{Proc.~NeurIPS'13}, 2013.

\bibitem[Kim(2019)]{kim2019accelerated}
Donghwan Kim.
\newblock Accelerated proximal point method and forward method for monotone
  inclusions.
\newblock \emph{arXiv preprint arXiv:1905.05149}, 2019.

\bibitem[Kohlenbach(2011)]{kohlenbach2011quantitative}
Ulrich Kohlenbach.
\newblock On quantitative versions of theorems due to {F.E.} {Browder} and {R}.
  {Wittmann}.
\newblock \emph{Advances in Mathematics}, 226\penalty0 (3):\penalty0
  2764--2795, 2011.

\bibitem[Kohlenbach and Leu{\c{s}}tean(2012)]{kohlenbach2012effective}
Ulrich Kohlenbach and Lauren{\c{t}}iu Leu{\c{s}}tean.
\newblock Effective metastability of {Halpern} iterates in {CAT}(0) spaces.
\newblock \emph{Advances in Mathematics}, 231\penalty0 (5):\penalty0
  2526--2556, 2012.

\bibitem[Korpelevich(1977)]{korpelevich1977extragradient}
GM~Korpelevich.
\newblock Extragradient method for finding saddle points and other problems.
\newblock \emph{Matekon}, 13\penalty0 (4):\penalty0 35--49, 1977.

\bibitem[Lee and Kim(2021)]{lee2021fast}
Sucheol Lee and Donghwan Kim.
\newblock Fast extra gradient methods for smooth structured
  nonconvex-nonconcave minimax problems.
\newblock In \emph{Proc.~NeurIPS'21}, 2021.

\bibitem[Lei et~al.(2017)Lei, Ju, Chen, and Jordan]{lei2017non}
Lihua Lei, Cheng Ju, Jianbo Chen, and Michael~I Jordan.
\newblock Non-convex finite-sum optimization via {SCSG} methods.
\newblock In \emph{Proc.~NeurIPS'17}, 2017.

\bibitem[Leustean(2007)]{leustean2007rates}
Laurentiu Leustean.
\newblock Rates of asymptotic regularity for {Halpern} iterations of
  nonexpansive mappings.
\newblock \emph{Journal of Universal Computer Science}, 13\penalty0
  (11):\penalty0 1680--1691, 2007.

\bibitem[Li et~al.(2021)Li, Bao, Zhang, and Richt{\'a}rik]{li2020page}
Zhize Li, Hongyan Bao, Xiangliang Zhang, and Peter Richt{\'a}rik.
\newblock {PAGE}: {A} simple and optimal probabilistic gradient estimator for
  nonconvex optimization.
\newblock In \emph{Proc.~ICML'21}, 2021.

\bibitem[Lieder(2021)]{lieder2019convergence}
Felix Lieder.
\newblock On the convergence rate of the {Halpern}-iteration.
\newblock \emph{Optimization Letters}, 15\penalty0 (2):\penalty0 405--418,
  2021.

\bibitem[Loizou et~al.(2021)Loizou, Berard, Gidel, Mitliagkas, and
  Lacoste-Julien]{loizou2021stochastic}
Nicolas Loizou, Hugo Berard, Gauthier Gidel, Ioannis Mitliagkas, and Simon
  Lacoste-Julien.
\newblock Stochastic gradient descent-ascent and consensus optimization for
  smooth games: {Convergence} analysis under expected co-coercivity.
\newblock In \emph{Proc.~NeurIPS'21}, 2021.

\bibitem[Minty(1962)]{minty1962monotone}
George~J Minty.
\newblock Monotone (nonlinear) operators in {Hilbert} space.
\newblock \emph{Duke Mathematical Journal}, 29\penalty0 (3):\penalty0 341--346,
  1962.

\bibitem[Nemirovski(2004)]{nemirovski2004prox}
Arkadi Nemirovski.
\newblock Prox-method with rate of convergence {$O(1/t)$} for variational
  inequalities with {L}ipschitz continuous monotone operators and smooth
  convex-concave saddle point problems.
\newblock \emph{SIAM Journal on Optimization}, 15\penalty0 (1):\penalty0
  229--251, 2004.

\bibitem[Nesterov(2007)]{nesterov2007dual}
Yurii Nesterov.
\newblock Dual extrapolation and its applications to solving variational
  inequalities and related problems.
\newblock \emph{Mathematical Programming}, 109\penalty0 (2-3):\penalty0
  319--344, 2007.

\bibitem[Nguyen et~al.(2017)Nguyen, Liu, Scheinberg, and
  Tak{\'a}{\v{c}}]{nguyen2017sarah}
Lam~M Nguyen, Jie Liu, Katya Scheinberg, and Martin Tak{\'a}{\v{c}}.
\newblock {SARAH}: A novel method for machine learning problems using
  stochastic recursive gradient.
\newblock In \emph{Proc.~ICML'17}, 2017.

\bibitem[Ouyang and Xu(2019)]{Ouyang2019}
Yuyuan Ouyang and Yangyang Xu.
\newblock Lower complexity bounds of first-order methods for convex-concave
  bilinear saddle-point problems.
\newblock \emph{Mathematical Programming}, Aug 2019.

\bibitem[Palaniappan and Bach(2016)]{palaniappan2016stochastic}
Balamurugan Palaniappan and Francis Bach.
\newblock Stochastic variance reduction methods for saddle-point problems.
\newblock In \emph{Proc.~NeurIPS'16}, 2016.

\bibitem[Pang(1997)]{pang1997error}
Jong-Shi Pang.
\newblock Error bounds in mathematical programming.
\newblock \emph{Mathematical Programming}, 79\penalty0 (1):\penalty0 299--332,
  1997.

\bibitem[Popov(1980)]{Popov1980}
L.~D. Popov.
\newblock A modification of the {Arrow-Hurwicz} method for search of saddle
  points.
\newblock \emph{Mathematical notes of the Academy of Sciences of the USSR},
  28\penalty0 (5):\penalty0 845--848, Nov 1980.

\bibitem[Reddi et~al.(2016)Reddi, Hefny, Sra, Poczos, and
  Smola]{reddi2016stochastic}
Sashank~J Reddi, Ahmed Hefny, Suvrit Sra, Barnabas Poczos, and Alex Smola.
\newblock Stochastic variance reduction for nonconvex optimization.
\newblock In \emph{Proc.~ICML'16}, 2016.

\bibitem[Rockafellar(1970)]{rockafellar1970monotone}
R~Tyrrell Rockafellar.
\newblock Monotone operators associated with saddle-functions and minimax
  problems.
\newblock \emph{Nonlinear Functional Analysis}, 18\penalty0 (part 1):\penalty0
  397--407, 1970.

\bibitem[Rockafellar(1976)]{rockafellar1976monotone}
R~Tyrrell Rockafellar.
\newblock Monotone operators and the proximal point algorithm.
\newblock \emph{SIAM Journal on Control and Optimization}, 14\penalty0
  (5):\penalty0 877--898, 1976.

\bibitem[Roulet and d'Aspremont(2020)]{roulet2020sharpness}
Vincent Roulet and Alexandre d'Aspremont.
\newblock Sharpness, restart, and acceleration.
\newblock \emph{SIAM Journal on Optimization}, 30\penalty0 (1):\penalty0
  262--289, 2020.

\bibitem[Sabach and Shtern(2017)]{Sabach:2017}
Shoham Sabach and Shimrit Shtern.
\newblock A first order method for solving convex bilevel optimization
  problems.
\newblock \emph{{SIAM} Journal on Optimization}, 27\penalty0 (2):\penalty0
  640--660, 2017.

\bibitem[Schmidt et~al.(2017)Schmidt, Le~Roux, and Bach]{schmidt2017minimizing}
Mark Schmidt, Nicolas Le~Roux, and Francis Bach.
\newblock Minimizing finite sums with the stochastic average gradient.
\newblock \emph{Mathematical Programming}, 162\penalty0 (1):\penalty0 83--112,
  2017.

\bibitem[Song et~al.(2020)Song, Jiang, and Ma]{song2020variance}
Chaobing Song, Yong Jiang, and Yi~Ma.
\newblock Variance reduction via accelerated dual averaging for finite-sum
  optimization.
\newblock In \emph{Proc.~NeurIPS'20}, 2020.

\bibitem[Stampacchia(1964)]{stampacchia1964formes}
Guido Stampacchia.
\newblock Formes bilineaires coercitives sur les ensembles convexes.
\newblock \emph{Acad\'emie des Sciences de Paris}, 258:\penalty0 4413--4416,
  1964.

\bibitem[Tran-Dinh and Luo(2021)]{tran2021halpern}
Quoc Tran-Dinh and Yang Luo.
\newblock Halpern-type accelerated and splitting algorithms for monotone
  inclusions.
\newblock \emph{arXiv preprint arXiv:2110.08150}, 2021.

\bibitem[Wittmann(1992)]{wittmann1992approximation}
Rainer Wittmann.
\newblock Approximation of fixed points of nonexpansive mappings.
\newblock \emph{Archiv der Mathematik}, 58\penalty0 (5):\penalty0 486--491,
  1992.

\bibitem[Yoon and Ryu(2021)]{Yoon2021OptimalGradientNorm}
Taeho Yoon and Ernest~K Ryu.
\newblock Accelerated algorithms for smooth convex-concave minimax problems
  with {$O(1/k^2)$} rate on squared gradient norm.
\newblock In \emph{Proc.~ICML'21}, 2021.

\bibitem[Zhou et~al.(2018{\natexlab{a}})Zhou, Xu, and Gu]{zhou2018finding}
Dongruo Zhou, Pan Xu, and Quanquan Gu.
\newblock Finding local minima via stochastic nested variance reduction.
\newblock \emph{arXiv preprint arXiv:1806.08782}, 2018{\natexlab{a}}.

\bibitem[Zhou et~al.(2018{\natexlab{b}})Zhou, Xu, and Gu]{zhou2018stochastic}
Dongruo Zhou, Pan Xu, and Quanquan Gu.
\newblock Stochastic nested variance reduction for nonconvex optimization.
\newblock In \emph{Proc.~NeurIPS'18}, 2018{\natexlab{b}}.

\end{thebibliography}

\appendix
\newpage

\section{Omitted proofs from Section~\ref{sec:prelim}}\label{appx:prelim-omitted}
\PAGElemma*
\begin{proof}
Using the definition of $\tF,$ conditional on $\cf_{k-1}$, we have for all $k \geq 1$ 
\begin{equation*}
\begin{aligned}
    \;& \ee\Big[\norm{\tF(\vu_{k}) - F(\vu_{k})}^2 \Big| \cf_{k-1}\Big] \\
    = \;& p_k\ee\Big[\Big\|\frac{1}{S_1^{(k)}} \sum_{i = 1}^{S_1^{(k)}} \hF(\vu_{k}, z^{(k)}_i) - F(\vu_{k})\Big\|^2 \Big| \cf_{k-1}\Big] \\
    & + (1 - p_k)\ee\Big[\Big\|\tF(\vu_{k - 1}) + \frac{1}{S_2^{(k)}} \sum_{i = 1}^{S_2^{(k)}} \left(\hF(\vu_{k}, z_i^{(k)}) - \hF(\vu_{k - 1}, z_i^{(k)})\right) - F(\vu_{k})\Big\|^2 \Big| \cf_{k-1}\Big], 
\end{aligned}
\end{equation*}
where $\cf_{k - 1} = \sigma(\{\tF(\vu_j)\}_{j \leq {k - 1}})$ is the natural filtration, as defined in Section~\ref{sec:prelim}. Note that both $\vu_{k - 1} \in \cf_{k - 1}$ and $\vu_k \in \cf_{k - 1}$ by the updating scheme considered in this paper, so we have 
\begin{equation}\label{eq:recursive-variance-proof}
\begin{aligned}
    \;& \ee\Big[\norm{\tF(\vu_{k}) - F(\vu_{k})}^2 \Big| \cf_{k-1}\Big] \\
    = \;& p_k\underbrace{\ee_{z^{(k)}}\Big[\Big\|\frac{1}{S_1^{(k)}} \sum_{i = 1}^{S_1^{(k)}} \hF(\vu_{k}, z^{(k)}_i) - F(\vu_{k})\Big\|^2\Big]}_{\mathcal{T}_1} \\
    & + (1 - p_k)\underbrace{\ee_{z^{(k)}}\Big[\Big\|\tF(\vu_{k - 1}) + \frac{1}{S_2^{(k)}} \sum_{i = 1}^{S_2^{(k)}} \left(\hF(\vu_{k}, z_i^{(k)}) - \hF(\vu_{k - 1}, z_i^{(k)})\right) - F(\vu_{k})\Big\|^2\Big]}_{\mathcal{T}_2}.
\end{aligned}
\end{equation}
Here we use $\ee_{z^{(k)}}$ to denote taking expectation with respect to the randomness of random seeds $z^{(k)}_i \overset{\text{i.i.d.}}{\sim} P_z$ sampled at iteration $k$. 

For the term $\mathcal{T}_1$, we have
\begin{equation}\label{ineq:ff-update}
\begin{aligned}
    \;& \ee_{z^{(k)}} \Big[ \Big\| \frac{1}{S_1^{(k)}} \sum_{i = 1}^{S_1^{(k)}} \hF(\vu_{k}, z^{(k)}_i) - F(\vu_{k}) \Big\|^2 \Big] \\
    \overset{(\romannumeral1)}{=} \;& \ee_{z^{(k)}} \Big[ \frac{1}{\big(S_1^{(k)}\big)^2} \sum_{i = 1}^{S_1^{(k)}} \norm{\hF(\vu_{k}, z^{(k)}_i) - F(\vu_{k})}^2 \Big] \leq \frac{\sigma^2}{S_1^{(k)}},
\end{aligned}
\end{equation}
where $(\romannumeral1)$ is due to $z_i^{(k)} \overset{\text{i.i.d.}}{\sim} P_z$ and $\ee\left[\hF(\vu_{k}, z^{(k)}_i)\right] = F(\vu_{k})$. 

For the term $\mathcal{T}_2$, we have 
\begin{equation*}
\begin{aligned}
    \;& \ee_{z^{(k)}}\Big[\Big\|\tF(\vu_{k - 1}) + \frac{1}{S_2^{(k)}} \sum_{i = 1}^{S_2^{(k)}} \left(\hF(\vu_{k}, z_i^{(k)}) - \hF(\vu_{k - 1}, z_i^{(k)})\right) - F(\vu_{k})\Big\|^2\Big] \\
    \overset{(\romannumeral1)}{=} \;& \ee_{z^{(k)}} \Big[\frac{1}{\big(S_2^{(k)}\big)^2}\Big\| \sum_{i = 1}^{S_2^{(k)}} \left[ \left(\hF(\vu_{k}, z_i^{(k)}) - \hF(\vu_{k - 1}, z_i^{(k)})\right) - \left(F(\vu_k) - F(\vu_{k - 1})\right) \right] \Big\|^2 \Big] \\
    & + \ee_{z^{(k)}} \Big[ \Big\| \tF(\vu_{k - 1}) - F(\vu_{k - 1})\Big\|^2\Big] \\
    \overset{(\romannumeral2)}{=} \;& \ee_{z^{(k)}} \Big[ \frac{1}{\big(S_2^{(k)}\big)^2}\sum_{i = 1}^{S_2^{(k)}} \Big\| \hF(\vu_{k}, z_i^{(k)}) - \hF(\vu_{k - 1}, z_i^{(k)}) - \left(F(\vu_k) - F(\vu_{k - 1})\right) \Big\|^2\Big] \\
    & + \ee_{z^{(k)}} \Big[ \Big\|\tF(\vu_{k - 1}) - F(\vu_{k - 1})\Big\|^2 \Big],
\end{aligned}
\end{equation*}
where $(\romannumeral1)$ and $(\romannumeral2)$ can be verified by expanding the square norm and using the assumption that all $z_i^{(k)}$ are i.i.d.~and $\hF(\vx, z_i^{(k)})$ is unbiased.  Since $\ee[\norm{X - \ee X}^2] \leq \ee[\norm{X}^2]$ for any random variable $X$, and using Assumption~\ref{assmpt:Lipschitz} for the stochastic queries, we have 
\begin{equation*}
    \begin{aligned}
    \;& \ee_{z^{(k)}} \Big[ \frac{1}{\big(S_2^{(k)}\big)^2}\sum_{i = 1}^{S_2^{(k)}} \Big\| \hF(\vu_{k}, z_i^{(k)}) - \hF(\vu_{k - 1}, z_i^{(k)}) - \left(F(\vu_k) - F(\vu_{k - 1})\right) \Big\|^2\Big] \\
    \leq \;& \frac{1}{\big(S_2^{(k)}\big)^2}\sum_{i = 1}^{S_2^{(k)}} \ee_{z_i^{(k)}}\Big[\norm{ \hF(\vu_{k}, z_i^{(k)}) - \hF(\vu_{k - 1}, z_i^{(k)}) }^2\Big] 
    \leq \frac{L^2\norm{\vu_k - \vu_{k - 1}}^2}{S_2^{(k)}}.
    \end{aligned}
\end{equation*}
So we obtain 
\begin{equation}\label{ineq:fs-update}
\begin{aligned}
    \;& \ee_{z^{(k)}} \Big[ \Big\| \tF(\vu_{k - 1}) + \frac{1}{S_2^{(k)}} \sum_{i = 1}^{S_2^{(k)}} \left(\hF(\vu_{k}, z_i^{(k)}) - \hF(\vu_{k - 1}, z_i^{(k)})\right) - F(\vu_{k}) \Big\|^2 \Big] \\
    \leq \;& \norm{\tF(\vu_{k - 1}) - F(\vu_{k - 1})}^2 + \frac{L^2\norm{\vu_k - \vu_{k - 1}}^2}{S_2^{(k)}}.
\end{aligned}
\end{equation}
Plugging Inequalities~\eqref{ineq:ff-update} and~\eqref{ineq:fs-update} into Eq.~\eqref{eq:recursive-variance-proof}, we have 
\begin{equation*}
\begin{aligned}
    \;& \ee\Big[\norm{\tF(\vu_{k}) - F(\vu_{k})}^2 \big| \cf_{k - 1}\Big] \\
    \leq \;& \frac{p_k \sigma^2}{S_1^{(k)}} + (1 - p_k)\norm{\tF(\vu_{k - 1}) - F(\vu_{k - 1})}^2 + \frac{(1 - p_k)L^2\norm{\vu_k - \vu_{k - 1}}^2}{S_2^{(k)}}.
\end{aligned}
\end{equation*}
Taking expectation with respect to all the randomness on both sides, and by the tower property of conditional expectations, we now obtain  
\begin{equation*}
\begin{aligned}
    \ee\Big[\norm{\tF(\vu_{k}) - F(\vu_{k})}^2\Big] \leq \;& p_k\sigma^2\ee\Big[\frac{1}{S_1^{(k)}}\Big] + (1 - p_k)\ee\Big[\norm{\tF(\vu_{k - 1}) - F(\vu_{k - 1})}^2\Big] \\
    & + (1 - p_k)L^2\ee\Big[\frac{\norm{\vu_{k} - \vu_{k - 1}}^2}{S_2^{(k)}}\Big],
\end{aligned}
\end{equation*}
which leads to the inequality in the lemma when $S_1^{(k)}$ are deterministic, thus completing the proof.
\end{proof}

\PAGEcoro*
\begin{proof}
We prove it by induction whose base step is 
\begin{equation*}
    \ee\Big[\norm{\tF(\vu_{1}) - F(\vu_{1})}^2\Big] \leq \frac{p_1\sigma^2}{S_1^{(1)}} \leq \frac{\epsilon^2}{8} \leq \epsilon^2,
\end{equation*}
where we use that $p_1 = 1$.

Assume that the result holds for all $j < k$; then by Lemma~\ref{lemma:recursive-variance-bound}, we have that at iteration $k$ 
\begin{equation*}
\begin{aligned}
    \;& \ee\Big[\norm{\tF(\vu_{k}) - F(\vu_{k})}^2\Big] \\
    \leq \;& \frac{p_k\sigma^2}{S_1^{(k)}} + (1 - p_k)\ee\Big[\norm{\tF(\vu_{k - 1}) - F(\vu_{k - 1})}^2\Big] + (1 - p_k)L^2\ee\Big[\frac{\norm{\vu_{k} - \vu_{k - 1}}^2}{S_2^{(k)}}\Big].
\end{aligned}
\end{equation*}
Plugging in our choice of $p_k$, $S_1^{(k)}$ and $S_2^{(k)}$, we have 
\begin{equation*}
\begin{aligned}
    \ee\Big[\norm{\tF(\vu_{k}) - F(\vu_{k})}^2\Big] \leq \;& \frac{p_k^2\epsilon^2}{8} + \frac{(1 - p_k)\epsilon^2}{k - 1} + \frac{p_k^2(1 - p_k)\epsilon^2}{8} \\
    \overset{(\romannumeral1)}{\leq} \;& \frac{p_k^2\epsilon^2}{4} + \frac{(1 - p_k)\epsilon^2}{k - 1} 
    = \left(\frac{1}{(k + 1)^2} + \frac{1}{k + 1}\right)\epsilon^2 
    \overset{(\romannumeral2)}{\leq} \frac{\epsilon^2}{k},
\end{aligned}
\end{equation*}
where $(\romannumeral1)$ is due to $\frac{p_k^2(1 - p_k)\epsilon^2}{8} \leq \frac{p_k^2\epsilon^2}{8}$, and $(\romannumeral2)$ is because $k(k + 2) \leq (k + 1)^2$. Hence, by induction, we can conclude that the result holds for all $k \geq 1$.
\end{proof}

\section{Omitted proofs from Section~\ref{sec:cocoercive}}\label{appx:omitted-cocoercive}

\subsection{Unconstrained settings}\label{appxSec:noVR}
Our argument for bounding the total number of stochastic queries to $F$ is based on the use of the following potential function, which was previously used for the deterministic case of Halpern iteration in \citep{diakonikolas2020halpern,diakonikolas2021potential},
\begin{equation}\label{eq:potential-function}
    \cc_{k} = \frac{A_k}{L_k}\lVert F(\vu_k) \rVert^2 + B_k\innp{F(\vu_k), \vu_k - \ui}, 
\end{equation}
where $\{ A_{k} \}_{k \geq 1}$ and $\{ B_{k} \}_{k \geq 1}$ are positive and non-decreasing sequences of real numbers, while the step size $\lambda_k$ is defined by $\lambda_k := \frac{B_k}{A_k + B_k}$. 
We start the proof by first justifying that a bound on the chosen potential function $\cc_k$ leads to a bound on $\|F(\vu_k)\|$ in expectation. The proof is a simple extension of~\citep[Lemma~4]{diakonikolas2020halpern} and is provided for completeness. 
\begin{restatable}{lemma}{errorCoco}
\label{lemma:error}
Given $k \geq 1,$ let $\cc_k$ be defined as in Eq.~\eqref{eq:potential-function} and let $\um$ be a solution to the monotone inclusion problem corresponding to $F$. If $\ee\left[\cc_k\right] \leq \ee\left[\ce_k\right]$ for some error term $\ce_k$, then 
  \begin{equation}
      \ee\left[\norm{F(\vu_k)}^2\right] \leq \frac{B_k L_k}{A_k}\norm{\ui - \vu^*}\ee\left[\norm{F(\vu_k)}\right] + \frac{L_k}{A_k}\ee\left[\ce_k\right],
  \end{equation}
  where the expectation is taken with respect to all random queries to $F$. 
\end{restatable}
\begin{proof}
By the definition of $\cc_k$, we have
\begin{equation*}
\begin{aligned}
    \ee\left[\norm{ F(\vu_k) }^2\right] \leq \;& \frac{B_k L_k}{A_k}\ee\left[\innp{F(\vu_k), \vu_0 - \vu_{k}}\right] + \frac{L_k}{A_k}\ee\left[\ce_k\right] \\
    = \;& \frac{B_k L_k}{A_k}\ee\left[\innp{F(\vu_k), \vu_0 - \vu^* + \vu^* - \vu_{k}}\right] + \frac{L_k}{A_k}\ee\left[\ce_k\right] \\
    = \;& \frac{B_k L_k}{A_k}\ee\left[\innp{F(\vu_k), \vu_0 - \vu^*}\right] + \frac{B_k L_k}{A_k}\ee\left[\innp{F(\vu_k), \vu^* - \vu_{k}}\right] + \frac{L_k}{A_k}\ee\left[\ce_k\right].
\end{aligned}
\end{equation*}
Since $\um$ is a solution to the monotone inclusion problem, as discussed in Section~\ref{sec:prelim}, it is also a weak VI (or MVI) solution, and thus 
\begin{equation*}
    (\forall k \geq 0) \quad \innp{F(\vu_k), \um - \vu_{k}} \leq 0.
\end{equation*}
As a result, 
\begin{equation*}
    \begin{aligned}
        \ee\left[\norm{ F(\vu_k) }^2\right] \leq \;& \frac{B_k L_k}{A_k}\ee\left[\innp{F(\vu_k), \ui - \um}\right] + \frac{L_k}{A_k}\ee\left[\ce_k\right] \\
        \overset{(\romannumeral1)}{\leq} \;& \frac{B_k L_k}{A_k}\ee\left[\norm{F(\vu_k)}\norm{\ui - \um}\right] + \frac{L_k}{A_k}\ee\left[\ce_k\right] \\
        \overset{(\romannumeral2)}{=} \;& \frac{B_k L_k}{A_k}\norm{\ui - \um}\ee\left[\norm{F(\vu_k)}\right] + \frac{L_k}{A_k}\ee\left[\ce_k\right],
    \end{aligned}
\end{equation*}
where we use Cauchy-Schwarz inequality for $(\romannumeral1)$, while $(\romannumeral2)$ holds because $\norm{\ui - \um}$ involves no randomness.
\end{proof}

Using Lemma \ref{lemma:error}, our goal now is to show that we can provide a bound on $\ee[\cc_k]$ by appropriately choosing the algorithm parameters. In the deterministic setup, it is sufficient to choose $L_k = \co(L)$ and $\lambda_k = \co(\frac{1}{k})$ to ensure that $\{ A_{k}\cc_k \}_{k \geq 1}$ is monotonically non-increasing, which immediately leads to  $\cc_k \leq \frac{A_1}{A_k}\cc_1$. In the stochastic setup considered here, we follow the same motivation, but need to deal with additional error terms caused by the stochastic access to $F$. 

We assume throughout that $L$ is known, and make the following assumption on the choice of $\{ A_{k} \}_{k \geq 1}$, $\{ B_{k} \}_{k \geq 1}$, and $\{ L_{k} \}_{k \geq 1}$, and provide a corresponding bound on the change of $\cc_k$ in Lemma~\ref{lemma:diffCk}.
%

\begin{assumption}\label{asp:para}
 $\{ L_k \}_{k \geq 1}$ is a sequence of positive reals such that $L_k \geq L$ for all $k \in \nn$. Sequences $\{ A_{k} \}_{k \geq 1}$ and $\{ B_{k} \}_{k \geq 1}$ are positive and non-decreasing, satisfying the following for all $k \geq 2$:
\begin{equation*}
    \frac{B_{k - 1}}{A_k} = \frac{B_k}{A_k + B_k}, \qquad \frac{1}{L_k}\left(1 - \frac{2B_k}{A_k + B_k}\right) = \frac{A_{k - 1}}{A_k L_{k - 1}}.
\end{equation*}
\end{assumption}

\begin{restatable}{lemma}{diffCoco}
\label{lemma:diffCk}
Let $\cc_k$ be defined as in 
Eq.~\eqref{eq:potential-function}, where $\{ A_{k} \}_{k \geq 1}$ and $\{ B_{k} \}_{k \geq 1}$ satisfy Assumption \ref{asp:para}. Let $L_k = 2L$ for all $k \geq 1$. Then, for any $k \geq 2$, we have
\begin{equation*}
    \begin{aligned}
    \;& \cc_{k} - \cc_{k - 1} \leq \frac{A_k}{2L}\norm{ F(\vu_{k - 1}) - \tF(\vu_{k - 1}) }^2 + \frac{A_k - A_{k - 1}}{2L}\innp{F(\vu_{k - 1}), F(\vu_{k - 1}) - \tF(\vu_{k - 1})}.
    \end{aligned}
\end{equation*}
\end{restatable}
\begin{proof}
By the definition of $\cc_{k}$, we have 
\begin{equation*}
\begin{aligned}
    \cc_{k} - \cc_{k - 1} = \;& \frac{A_k}{L_k}\norm{ F(\vu_k) }^2 + B_k\innp{F(\vu_k), \vu_k - \ui} \\
    & - \frac{A_{k - 1}}{L_{k - 1}}\norm{ F(\vu_{k - 1}) }^2 - B_{k - 1}\innp{F(\vu_{k - 1}), \vu_{k - 1} - \ui}.
\end{aligned}
\end{equation*}
Since the operator $F$ is cocoercive with parameter $\frac{1}{L}$, we have
\begin{equation*}
    \begin{aligned}
        \;& \innp{F(\vu_{k}) - F(\vu_{k - 1}), \vu_{k} - \vu_{k - 1}} \\
        \geq \;& \frac{1}{L}\norm{ F(\vu_{k}) - F(\vu_{k - 1}) }^2 \\
        = \;& \frac{1}{L_k}\norm{ F(\vu_{k}) - F(\vu_{k - 1}) }^2 + \Big(\frac{1}{L} - \frac{1}{L_k}\Big)\norm{ F(\vu_{k}) - F(\vu_{k - 1}) }^2 \\
        = \;& \frac{1}{L_k}\norm{F(\vu_{k})}^2 - \frac{2}{L_k}\innp{F(\vu_{k}), F(\vu_{k - 1})} + \frac{1}{L_k}\norm{ F(\vu_{k - 1})}^2 \\
        & + \Big(\frac{1}{L} - \frac{1}{L_k}\Big)\norm{F(\vu_{k}) - F(\vu_{k - 1})}^2.
    \end{aligned}
\end{equation*}
By rearranging, we obtain
\begin{align*}
    \frac{1}{L_k}\norm{F(\vu_{k})}^2 \leq \;& \innp{F(\vu_{k}), \vu_{k} - \vu_{k - 1} + \frac{2}{L_k}F(\vu_{k - 1})} - \innp{F(\vu_{k - 1}), \vu_{k} - \vu_{k - 1}} \\
    & - \frac{1}{L_k}\norm{ F(\vu_{k - 1}) }^2 - \left(\frac{1}{L} - \frac{1}{L_k}\right)\norm{ F(\vu_{k}) - F(\vu_{k - 1}) }^2.
\end{align*}
Multiplying $A_k$ on both sides and plugging into $\cc_k - \cc_{k - 1}$, we have
\begin{equation*}
\begin{aligned}
    \cc_{k} - \cc_{k - 1} \leq \;& \innp{F(\vu_{k}), A_k(\vu_{k} - \vu_{k - 1}) + \frac{2A_k}{L_k}F(\vu_{k - 1}) + B_k(\vu_k - \ui)} \\
    & - \innp{F(\vu_{k - 1}), A_k(\vu_{k} - \vu_{k - 1}) + B_{k - 1}(\vu_{k - 1} - \ui)} \\
    & - \Big(\frac{A_k}{L_k} + \frac{A_{k - 1}}{L_{k - 1}}\Big)\norm{ F(\vu_{k - 1}) }^2 - A_k\Big(\frac{1}{L} - \frac{1}{L_k}\Big)\norm{ F(\vu_{k}) - F(\vu_{k - 1}) }^2.
\end{aligned}
\end{equation*}
Since $\lambda_k = \frac{B_k}{A_k + B_k}$, we have 
\begin{equation*}
    \vu_{k} = \frac{B_k}{A_k + B_k}\ui + \frac{A_k}{A_k + B_k}\Big(\vu_{k - 1} - \frac{2}{L_{k}}\tF(\vu_{k - 1})\Big),
\end{equation*}
which leads to $A_k(\vu_{k} - \vu_{k - 1}) + \frac{2A_k}{L_k}F(\vu_{k - 1}) + B_k(\vu_k - \ui) =  \frac{2A_k}{L_k}\big(F(\vu_{k - 1}) - \tF(\vu_{k - 1})\big)$.
Further, as $\frac{B_{k - 1}}{A_k} = \frac{B_k}{A_k + B_k}$ by Assumption \ref{asp:para}, we have 
\begin{equation*}
    \begin{aligned}
        \;& \innp{F(\vu_{k - 1}), A_k(\vu_{k} - \vu_{k - 1}) + B_{k - 1}(\vu_{k - 1} - \ui)} \\
        = \;& A_k\innp{F(\vu_{k - 1}), \vu_{k} - \frac{B_{k - 1}}{A_k}\ui - \frac{A_k - B_{k - 1}}{A_k}\vu_{k-1}} \\
        = \;& A_k\innp{F(\vu_{k - 1}), \vu_{k} - \frac{B_k}{A_k + B_k}\ui - \frac{A_k}{A_k + B_k}\vu_{k-1}} \\
        = \;& - A_k\innp{F(\vu_{k - 1}), \frac{2A_k}{L_k(A_k + B_k)}\tF(\vu_{k - 1})}.
    \end{aligned}
\end{equation*}
Moreover, by Assumption \ref{asp:para}, we have $\frac{1}{L_k}\Big(1 - \frac{2B_k}{A_k + B_k}\Big) = \frac{A_{k - 1}}{A_k L_{k - 1}}$, so we obtain 
\begin{equation*}
\begin{aligned}
    \;& \innp{F(\vu_{k - 1}), A_k(\vu_{k} - \vu_{k - 1}) + B_{k - 1}(\vu_{k - 1} - \ui)} \\
    = \;& - A_k\innp{F(\vu_{k - 1}), \frac{2A_k}{L_k(A_k + B_k)}\tF(\vu_{k - 1})} \\
    = \;& - \innp{F(\vu_{k - 1}), \Big(\frac{A_k}{L_k} + \frac{A_{k -1}}{L_{k - 1}}\Big)\tF(\vu_{k - 1})}.
\end{aligned}
\end{equation*}
Since by hypothesis $L_k=2L$ for all $k\geq 1$,
we have 
\begin{equation*}
    \begin{aligned}
        \cc_{k} - \cc_{k - 1} \leq \;& \innp{F(\vu_{k}), \frac{A_k}{L}(F(\vu_{k - 1}) - \tF(\vu_{k - 1}))} + \innp{F(\vu_{k - 1}), \frac{A_k + A_{k - 1}}{2L}\tF(\vu_{k - 1})} \\
        & - \frac{A_k + A_{k - 1}}{2L}\norm{F(\vu_{k - 1})}^2 - \frac{A_k}{2L}\norm{ F(\vu_{k}) - F(\vu_{k - 1})}^2 \\
        \overset{(\romannumeral1)}{=} \;& \frac{A_k}{L}\innp{F(\vu_{k}) - F(\vu_{k - 1}), F(\vu_{k - 1}) - \tF(\vu_{k - 1})} - \frac{A_k}{2L}\norm{F(\vu_{k}) - F(\vu_{k - 1})}^2 \\
        & + \innp{F(\vu_{k - 1}), \frac{A_k - A_{k - 1}}{2L}\left(F(\vu_{k - 1}) - \tF(\vu_{k - 1})\right)}, 
    \end{aligned}
\end{equation*}
where $(\romannumeral1)$ is derived by rearranging and grouping terms. Using that $2\innp{p, q} - \norm{p}^2 \leq \norm{q}^2$ 
holds for any $p, q \in \rr^d$, we finally obtain 
\begin{align*}
    \cc_k - \cc_{k - 1} \leq \;& \frac{A_k}{2L}\norm{F(\vu_{k - 1}) - \tF(\vu_{k - 1})}^2 + \frac{A_k - A_{k - 1}}{2L}\innp{F(\vu_{k - 1}), F(\vu_{k - 1}) - \tF(\vu_{k - 1})},
\end{align*}
thus completing the proof.
\end{proof}

By Lemma \ref{lemma:diffCk}, if we choose $A_k = \co(k^2)$ and $B_k = \co(k)$ satisfying Assumption \ref{asp:para}, and take sufficiently large size of samples queried to ensure that $\ee\big[\big\|{F(\vu_k) - \tF(\vu_k)}\big\|^2\big] \leq \frac{\epsilon^2}{k}$ for $k \geq 0$, then we can obtain $\co(1/k)$ expected convergence rate in the norm of the operator by induction. Observe that we do not need an assumption that $\tF$ is an unbiased estimator of $F$ for any point except for the initial one; all that is needed is that the second moment of the estimation error, $\|F(\vu_k) - \tF(\vu_k)\|_2^2$, is bounded.

\rateCoco*
\begin{proof}
Observe first that the chosen sequence of numbers $A_k, B_k$ satisfies Assumption~\ref{asp:para}, and thus Lemma~\ref{lemma:diffCk} applies. 
Observe further that, by Jensen's Inequality, 
\begin{equation*}
    \ee[\norm{F(\vu_k))}] \leq \Big(\ee[\norm{F(\vu_k)}^2]\Big)^{\frac{1}{2}}.
\end{equation*}
and, thus, to prove the theorem, it suffices to show that there exists $\Lambda_0$ and $\Lambda_1$ such that for all $k \geq 1$ 
\begin{equation*}
     \Big(\ee[\norm{F(\vu_k)}^2]\Big)^{\frac{1}{2}} \leq \frac{\Lambda_0}{k} + \Lambda_1\epsilon.
\end{equation*}

We prove this claim by induction on $k$. For the base case $k = 1$, in which $\vu_1 = \vu_0 - \frac{1}{2L}\tF(\vu_0)$, we have 
\begin{equation}\label{eq:C_1-coco}
\cc_1 = \frac{1}{L}\norm{F(\vu_1)}^2 + 2\innp{F(\vu_1), \vu_1 - \vu_0} = \frac{1}{L}\Big(\norm{F(\vu_1)}^2 - \innp{F(\vu_1), \tF(\vu_0)}\Big).
\end{equation}
Further, since the operator $F$ is cocoercive with parameter $\frac{1}{L}$, it is also cocoercive with parameter $\frac{1}{2L}$, and thus we have 
\begin{equation*}
    \norm{F(\vu_1) - F(\vu_0)}^2 \leq 2L\innp{F(\vu_1) - F(\vu_0), \vu_1 - \vu_0} = \innp{F(\vu_1) - F(\vu_0), -\tF(\vu_0)}.
\end{equation*}
Expanding and rearranging the terms, we have 
\begin{equation*}
    \norm{F(\vu_1)}^2 \leq \innp{F(\vu_0), \tF(\vu_0) - F(\vu_0)} + 2\innp{F(\vu_1), F(\vu_0)} - \innp{F(\vu_1), \tF(\vu_0)}.
\end{equation*}
Recall that, by assumption, $\ee[\tF(\vu_0)] = F(\vu_0)$. Subtracting $\innp{F(\vu_1), \tF(\vu_0)}$ from both sides in the last inequality and taking expectation with respect to all the  randomness on both sides, we have 
\begin{equation*}
\begin{aligned}
    \;& \ee\Big[\norm{F(\vu_1)}^2 - \innp{F(\vu_1), \tF(\vu_0)}\Big] \\
    \leq \;& \ee\Big[\innp{F(\vu_0), \tF(\vu_0) - F(\vu_0)} + 2\innp{F(\vu_1), F(\vu_0)} - 2\innp{F(\vu_1), \tF(\vu_0)}\Big] \\
    = \;& 2\ee\Big[\innp{F(\vu_1), F(\vu_0) - \tF(\vu_0)}\Big] \\
    \overset{(\romannumeral1)}{\leq} \;& \ee\Big[\frac{1}{2}\big\|{F(\vu_1)}\big\|^2 + 2\big\|{F(\vu_0) - \tF(\vu_0)}\big\|^2\Big],
\end{aligned}
\end{equation*}
where for $(\romannumeral1)$ we use Young's inequality. Plugging into Eq.~\eqref{eq:C_1-coco}, we obtain that 
\begin{equation*}
    \ee[\cc_1] \leq \frac{1}{L}\ee\Big[\frac{1}{2}\norm{F(\vu_1)}^2 + 2\big\|F(\vu_0) - \tF(\vu_0)\big\|^2\Big].
\end{equation*}
Note that $A_1 = B_1 = 2$ and $L_1 = 2L$, by Lemma \ref{lemma:error} we have 
\begin{equation*}
\begin{aligned}
    \ee[\norm{F(\vu_1)}^2] \leq \;& \frac{B_1L_1}{A_1}\norm{\vu_0 - \vu^*}\ee[\norm{F(\vu_1)}] + \frac{L_1}{A_1}\frac{1}{L}\ee\Big[\frac{1}{2}\norm{F(\vu_1)}^2 + 2\big\|{F(\vu_0) - \tF(\vu_0)}\big\|^2\Big] \\
    = \;& 2L\norm{\vu_0 - \vu^*}\ee[\norm{F(\vu_1)}] + \ee\Big[\frac{1}{2}\norm{F(\vu_1)}^2 + 2\big\|{F(\vu_0) - \tF(\vu_0)}\big\|^2\Big].
\end{aligned}
\end{equation*}

Subtracting $\ee[\frac{1}{2}\norm{F(\vu_1)}^2]$ on both sides and using that (by Jensen's inequality) $\ee[\norm{F(\vu_1)}] \leq \big(\ee[\norm{F(\vu_1)}^2]\big)^{\frac{1}{2}}$ and (by assumption) $\ee[\|{F(\vu_0) - \tF(\vu_0)}\|^2] \leq \frac{\epsilon^2}{8}$, we have 
\begin{equation*}
    \ee[[\norm{F(\vu_1)}^2] \leq 4L\norm{\vu_0 - \vu^*}\big(\ee[\norm{F(\vu_1)}^2]\big)^{\frac{1}{2}} + \frac{\epsilon^2}{2},
\end{equation*}
which is a quadratic inequality in  $(\ee[\norm{F(\vu_1)}^2])^{\frac{1}{2}}$. Bounding the solution to this quadratic inequality by its larger root, we have 
\begin{equation*}
    \begin{aligned}
        (\ee[\norm{F(\vu_1)}^2])^{\frac{1}{2}} \leq \;& 2L\norm{\vu_0 - \vu^*} + \frac{1}{2}\sqrt{16L^2\norm{\vu_0 - \vu^*}^2 + 2\epsilon^2} \\
    \leq \;& 2L\norm{\vu_0 - \vu^*} + \frac{1}{2}(4L\norm{\vu_0 - \vu^*} + \sqrt{2}\epsilon) \\
    \leq \;& 4L\norm{\vu_0 - \vu^*} + \epsilon \\
    \leq\; & \Lambda_0 + \Lambda_1\epsilon.
    \end{aligned}
\end{equation*}
This completes the proof for the base case. Moreover, we can get a bound for $\ee[\cc_1]$ as follows 
\begin{equation*}
    \begin{aligned}
        \ee[\cc_1] \leq \;& \frac{1}{L}\ee\Big[\frac{1}{2}\norm{F(\vu_1)}^2 + 2\norm{F(\vu_0) - \tF(\vu_0)}^2\Big] \\
        \overset{(\romannumeral1)}{\leq} \;& \frac{1}{2L}\Big(24L^2\norm{\vu_0 - \vu^*}^2 + \frac{3}{2}\epsilon^2\Big) + \frac{2}{L}\frac{\epsilon^2}{8} \\
        = \;& 12L\norm{\vu_0 - \vu^*}^2 + \frac{\epsilon^2}{L},
    \end{aligned}
\end{equation*}
where $(\romannumeral1)$ can be verified by the bound we get above for $\ee[\norm{F(\vu_1)}^2]$ and by applying Young's inequality and that, by assumption, $\ee[\|{F(\vu_0) - \tF(\vu_0)}\|^2] \leq \frac{\epsilon^2}{8}$. 

For the inductive hypothesis, assume that the result holds for all $1 \leq i \leq k - 1$, and consider iteration $k$. By Lemma~\ref{lemma:diffCk}, we have for $\forall i \geq 2$
\begin{equation*}
    \begin{aligned}
        \cc_{i} - \cc_{i - 1} \leq \;& \frac{A_i}{2L}\norm{F(\vu_{i - 1}) - \tF(\vu_{i - 1})}^2 + \frac{A_i - A_{i - 1}}{2L}\innp{F(\vu_{i - 1}), F(\vu_{i - 1}) - \tF(\vu_{i - 1})} \\
        \overset{(\romannumeral1)}{\leq} \;& \frac{5i(i + 1)}{2L}\norm{F(\vu_{i - 1}) - \tF(\vu_{i - 1})}^2 + \frac{i}{8L(i + 1)}\norm{F(\vu_{i - 1})}^2,
    \end{aligned}
\end{equation*}
where we use Young's inequality and $A_i = i(i+1)$ for $(\romannumeral1)$. Taking expectation with respect to all randomness on both sides and telescoping from $i = 2$ to $k$, we obtain 
\begin{equation}
\begin{aligned}
    \ee[\cc_k] \leq \;& \ee\Big[\cc_1 + \sum_{i = 2}^{k}\Big( \frac{5i(i + 1)}{2L}\big\|{F(\vu_{i - 1}) - \tF(\vu_{i - 1})}\big\|^2 + \frac{i}{8L(i + 1)}\norm{F(\vu_{i - 1})}^2 \Big)\Big]  \\
    \leq \;& \ee\Big[\sum_{i = 2}^{k}\Big( \frac{5i(i + 1)}{2L}\norm{F(\vu_{i - 1}) - \tF(\vu_{i - 1})}^2 + \frac{i}{8L(i + 1)}\norm{F(\vu_{i - 1})}^2 \Big)\Big]  \\
    & +  12L\norm{\vu_0 - \vu^*}^2 + \frac{\epsilon^2}{L}. 
\end{aligned}\label{eq:thm-coco-Ck-bnd}
\end{equation}
Using that, by assumption, for $k\geq 1,$ $\ee[\|F(\vu_k) - \tF(\vu_k)\|^2] \leq \frac{\epsilon^2}{k},$ we further have
\begin{equation}\label{eq:thm-coco-bnd-1}
\begin{aligned}
     \ee\Big[\sum_{i = 2}^{k}\frac{5i(i + 1)}{2L}\big\|{F(\vu_{i - 1}) - \tF(\vu_{i - 1})}\big\|^2\Big] 
    \leq \;& \sum_{i = 2}^{k}\frac{5i(i + 1)}{2L}\frac{\epsilon^2}{i - 1} \\
    \overset{(\romannumeral1)}{\leq} \;& \sum_{i = 2}^{k}\frac{5(i + 1)\epsilon^2}{L} \\
    =\; & \frac{5(k + 4)(k - 1)\epsilon^2}{2L},
\end{aligned}
\end{equation}
where $(\romannumeral1)$ is because $\frac{i}{i - 1} \leq 2$ for all $i \geq 2$. By induction, we have 
\begin{equation}\label{eq:thm-coco-bnd-2}
\begin{aligned}
     \ee\Big[\sum_{i = 2}^{k}\frac{i}{8L(i + 1)}\norm{F(\vu_{i - 1})}^2\Big] 
    \overset{(\romannumeral1)}{\leq} \;& \sum_{i = 2}^{k}\frac{1}{8L}\Big(\frac{2\Lambda_0^2}{(i - 1)^2} + 2\Lambda_1^2\epsilon^2\Big)\\
    \overset{(\romannumeral2)}{\leq} \;& \frac{1}{4 L}\Big(\Lambda_0^2\frac{\pi^2}{6} + (k - 1)\Lambda_1^2\epsilon^2\Big)\\
    =\; & \frac{1}{L}\Big(\frac{\Lambda_0^2\pi^2}{24} + \frac{(k - 1)\Lambda_1^2\epsilon^2}{4}\Big),
\end{aligned}
\end{equation}
where $(\romannumeral1)$ follows from induction and $\frac{i}{i + 1} \leq 1$, and $(\romannumeral2)$ is due to $\sum_{i = 2}^{k}\frac{1}{(i - 1)^2} \leq \sum_{i = 1}^{\infty}\frac{1}{i^2} = \frac{\pi^2}{6}$. Combining Eqs.~\eqref{eq:thm-coco-Ck-bnd}--\eqref{eq:thm-coco-bnd-2}, we get 
\begin{equation*}
\begin{aligned}
    \ee[\cc_k] \leq 12L\norm{\vu_0 - \vu^*}^2 + \frac{\epsilon^2}{L} + \frac{5(k + 4)(k - 1)\epsilon^2}{2L} + \frac{1}{L}\Big(\frac{\Lambda_0^2\pi^2}{24} + \frac{(k - 1)\Lambda_1^2\epsilon^2}{4}\Big).
\end{aligned}
\end{equation*}
Applying Lemma~\ref{lemma:error} to the bound on $\cc_k$ from the last inequality, we have 
\begin{equation*}
\begin{aligned}
    \;& \ee\big[\norm{F(\vu_k)}^2\big] \\
    \leq \;& \frac{B_kL_k}{A_k}\norm{\vu_0 - \vu^*}\ee[\norm{F(\vu_k)}] \\
    & + \frac{L_k}{A_k}\Big(12L\norm{\vu_0 - \vu^*}^2 + \frac{\epsilon^2}{L} + \frac{5(k + 4)(k - 1)\epsilon^2}{2L} + \frac{\Lambda_0^2\pi^2}{24 L} + \frac{(k - 1)\Lambda_1^2\epsilon^2}{4 L}\Big) \\
    = \;& \frac{2L}{k}\norm{\vu_0 - \vu^*}\ee[\norm{F(\vu_k)}] \\
    & + \frac{1}{k(k + 1)}\Big(24L^2\norm{\vu_0 - \vu^*}^2 + 2\epsilon^2 + 5(k + 4)(k - 1)\epsilon^2 + \frac{\Lambda_0^2\pi^2}{12} + \frac{(k - 1)\Lambda_1^2\epsilon^2}{2}\Big) \\
    \overset{(\romannumeral1)}{\leq} \;& \frac{2L}{k}\norm{\vu_0 - \vu^*}\ee[\norm{F(\vu_k)}] + \Big(\frac{24L^2\norm{\vu_0 - \vu^*}^2}{k^2} + \Big(8 + \frac{\Lambda_1^2}{2(k+1)}\Big)\epsilon^2 + \frac{\Lambda_0^2\pi^2}{12 k^2}\Big),
\end{aligned}
\end{equation*}
where $(\romannumeral1)$ is due to $\frac{1}{k(k + 1)} \leq \frac{1}{k^2}$, 
$\frac{5(k + 1)(k - 1)}{k(k + 1)} \leq 6$ and $\frac{k - 1}{k(k + 1)} \leq \frac{1}{k + 1}$. Since $\ee[\norm{F(\vu_k)}] \leq (\ee[\norm{F(\vu_k)}^2])^{\frac{1}{2}}$ by Jensen's inequality, we have 
\begin{equation*}
\begin{aligned}
    \;& \ee\big[\norm{F(\vu_k)}^2\big] \\
    \leq \;& \frac{2L}{k}\norm{\vu_0 - \vu^*}\Big(\ee\big[\norm{F(\vu_k)}^2\big]\Big)^{\frac{1}{2}} + \Big(\frac{24L^2\norm{\vu_0 - \vu^*}^2}{k^2} + \Big(8 + \frac{\Lambda_1^2}{2(k+1)}\Big)\epsilon^2 + \frac{\Lambda_0^2\pi^2}{12k^2}\Big),
\end{aligned}
\end{equation*}
which is a quadratic inequality with respect to $(\ee[\norm{F(\vu_k)}^2])^{\frac{1}{2}}$. Similarly as for $k = 1,$ bounding its solution by its larger root, we obtain 
\begin{equation*}
\begin{aligned}
    \;& \Big(\ee\big[\norm{F(\vu_k)}^2\big]\Big)^{\frac{1}{2}} \\
    \leq \;&  \frac{L}{k}\norm{\vu_0 - \vu^*} + \frac{1}{2}\sqrt{\frac{4L^2}{k^2}\norm{\vu_0 - \vu^*}^2 + 4\Big(\frac{24L^2\norm{\vu_0 - \vu^*}^2}{k^2} + \Big(8 + \frac{\Lambda_1^2}{2(k+1)}\Big)\epsilon^2 + \frac{\Lambda_0^2\pi^2}{12k^2}\Big)} \\
    \overset{(\romannumeral1)}{\leq} \;& \frac{2L}{k}\norm{\vu_0 - \vu^*} + \Big(\frac{5 L\norm{\vu_0 - \vu^*}}{k} + \sqrt{8 + \frac{\Lambda_1^2}{2(k+1)}}\epsilon + \frac{\Lambda_0\pi}{2\sqrt{3}k}\Big)
    \\
    = \;& \frac{7 L\norm{\vu_0 - \vu^*} + \frac{\Lambda_0\pi}{2\sqrt{3}}}{k} + \sqrt{8 + \frac{\Lambda_1^2}{2(k+1)}}\epsilon \\
    \overset{(\romannumeral2)}{\leq} \;& \frac{\Lambda_0}{k} + \Lambda_1\epsilon,
\end{aligned}
\end{equation*}
where $(\romannumeral1)$ is due to the fact that $\sqrt{\sum_{i = 1}^{n}X_i^2} \leq \sum_{i = 1}^{n}|X_i|$, and $(\romannumeral2)$ is because of our choice of $\Lambda_0, \Lambda_1$. Hence, the result also holds for the case $k$. Then by induction we know that the result holds for all $k \geq 1$.

Finally, when $k \geq \frac{2\Lambda_0}{\epsilon}$, we have $\frac{\Lambda_0}{k} \leq \epsilon / 2$. Also, since we have $\Lambda_1 = 4 \sqrt{\frac{2}{3}}< 3.5$, we obtain 
\begin{equation*}
    \ee[\norm{F(\vu_k)}] \leq \Big(\frac{1}{2} + 4 \sqrt{\frac{2}{3}}\Big)\epsilon \leq 4\epsilon.
\end{equation*}
Hence, the total number of iterations needed to attain $4\epsilon$ norm of the operator is  
\begin{equation*}
    N = \Big\lceil \frac{2\Lambda_0}{\epsilon} \Big\rceil = \co\Big(\frac{L\norm{\vu - \vu^*}}{\epsilon}\Big), 
\end{equation*}
thus completing the proof.
\end{proof}

\subsection{Constrained setting with a cocoercive operator}\label{appx:omitted-coco-constrained}
To extend the results to possibly constrained settings, similar to~\cite{diakonikolas2020halpern}, we make use of the operator mapping defined by
\begin{align}\label{eq:grad-mapping}
    \grad(\vu) = \eta\Big(\vu - \proj\big(\vu - \frac{1}{\eta}F(\vu)\big)\Big),
\end{align}
where $\cu \subseteq \rr^d$ is the closed convex constraint set and $\proj(\vu)$ is the projection operator. 
%
Operator $G_\eta$ is a valid proxy for approximating \eqref{def:MI}; see~\citep{diakonikolas2020halpern} for further details. 


The extension of our results to constrained stochastic settings is not immediate; the reason is that the stochastic query assumptions (Assumptions~\ref{assmpt:stoch-queries} and \ref{assmpt:multiQuery}) are made for the operator $F$, not $G_\eta.$ Nevertheless, as we show in this subsection, it is not hard to match the stochastic oracle complexity of the unconstrained setups by proving an additional auxiliary result that bounds the variance of an operator mapping corresponding to $\tF$ (Lemma~\ref{lemma:gradient-mapping-error}).  

We begin by recalling that whenever $F$ is $\frac{1}{L}$-cocoercive and $\eta \geq L$, the operator mapping $\grad$ is $\frac{3}{4\eta}$-cocoercive (see, e.g., \citep[Proposition~7]{diakonikolas2020halpern} and \citep[Lemma~10.11]{beck2017first}).
\begin{proposition} \label{prop:grad-cocoercive}
Let $F$ be $\frac{1}{L}$-cocoercive and let $\grad$ be defined as in Eq.~\eqref{eq:grad-mapping}, where $\eta \geq L$. Then $\grad$ is $\frac{3}{4\eta}$-cocoercive.
\end{proposition}

To state the variant of stochastic Halpern iteration for constrained settings, we also define the operator mapping corresponding to the stochastic estimate $\tF$ by 
\begin{align}\label{eq:stoc-grad-mapping}
    \tG_{\eta}(\vu) = \eta\Big(\vu - \proj\big(\vu - \frac{1}{\eta}\tF(\vu)\big)\Big).
\end{align}

In the following lemma, we bound the error between the stochastic operator mapping and true operator mapping by the variance of stochastic queries.
%
\begin{restatable}{lemma}{gmErrorCons}
\label{lemma:gradient-mapping-error}
Let $G_{\eta}(\cdot)$ and $\tG_{\eta}(\cdot)$ be defined as in Eq.~\eqref{eq:grad-mapping} and Eq.~\eqref{eq:stoc-grad-mapping},  respectively. Then, for any $\vu \in \cu$ and any $\eta > 0$, we have  
\begin{equation}
    \|{G_{\eta}(\vu) - \tG_{\eta}(\vu)}\|^2 \leq \|{F(\vu) - \tF(\vu)}\|^2.
\end{equation}
\end{restatable}
\begin{proof}
By the definition of gradient mapping, we have 
\begin{equation*}
\begin{aligned}
    \norm{G_{\eta}(\vu) - \tG_{\eta}(\vu)}^2 = \eta^2\Big\|\proj\Big(\vu - \frac{1}{\eta}F(\vu)\Big) - \proj\Big(\vu - \frac{1}{\eta}\tF(\vu)\Big)\Big\|^2.
\end{aligned}
\end{equation*}
Since the projection operator is non-expansive, we obtain 
\begin{equation*}
\begin{aligned}
    \norm{G_{\eta}(\vu) - \tG_{\eta}(\vu)}^2 \leq \;& \eta^2\Big\|\Big(\vu - \frac{1}{\eta}F(\vu)\Big) - \Big(\vu - \frac{1}{\eta}\tF(\vu)\Big)\Big\|^2 = \norm{F(\vu) - \tF(\vu)}^2, 
\end{aligned}
\end{equation*}
thus completing the proof.
\end{proof}

Similar to the unconstrained setup, we define the following stochastic Halpern iteration for the constrained setup: 
\begin{align}\label{eq:cons-halpern}
    \vu_{k + 1} = \lambda_{k + 1}\vu_0 + (1 - \lambda_{k + 1})\big(\vu_k - \tG_{L_k}(\vu_k) / L_{k + 1}\big), 
\end{align}
where $L_k \geq L$, $\forall k \geq 0$. By the cocoercivity of the operator mapping and the error bound in Lemma~\ref{lemma:gradient-mapping-error}, we can immediately obtain the results for the iteration complexity and stochastic oracle complexity as in the unconstrained case, by applying Theorem~\ref{thm:cocoercive} and Corollary~\ref{thm:complexity} to $G_L$ and $\tG_L$. This is summarized in the following Theorem~\ref{thm:coco-rate-constrained} and Corollary~\ref{cor:coco-constrained}. To prove these, we make use of the potential function as in the unconstrained settings 
\begin{equation}\label{eq:con-potential-function}
    \cc_{k} = \frac{A_k}{2L_k}\lVert G_{L_k}(\vu_k) \rVert^2 + B_k\innp{G_{L_k}(\vu_k), \vu_k - \ui}, 
\end{equation}
and first bound the change of $\cc_k$ in the following Lemma~\ref{lemma:ck-cons}. For short, we denote $G_L$ as $G$ below.

\begin{lemma}\label{lemma:ck-cons}
Let $\cc_k$ be defined as in Eq.~\eqref{eq:con-potential-function}, where $A_k$ and $B_k$ satisfy Assumption \ref{asp:para}. Assume that $L$ is already known and we set $L_k = L$ for any $k \geq 1$. Then for any $k \geq 2$, we have
\begin{equation*}
    \cc_k - \cc_{k - 1} \leq \frac{A_k}{L}\norm{G(\vu_{k - 1}) - \tG(\vu_{k - 1})}^2 + \frac{A_k - A_{k - 1}}{2L}\innp{G(\vu_{k - 1}), G(\vu_{k - 1}) - \tG(\vu_{k - 1})}.
\end{equation*}
\end{lemma}
\begin{proof}
By the definition of $\cc_k$, we have 
\begin{align*}
    \cc_k - \cc_{k - 1} = \;& \frac{A_k}{2L_k}\norm{G_{L_k}(\vu_k)}^2 + B_k\innp{G_{L_k}(\vu_k), \vu_k - \vu_0} \\
    & - \frac{A_{k - 1}}{2L_{k - 1}}\norm{G_{L_{k - 1}}(\vu_{k - 1})}^2 - B_{k - 1}\innp{G_{L_{k - 1}}(\vu_{k - 1}), \vu_{k - 1} - \vu_0}.
\end{align*}
Since $G_{L_k}$ is cocoercive with parameter $\frac{3}{4L_k}$ when $L_k \geq L$, we have 
\begin{align*}
    \;& \innp{G_{L_k}(\vu_k) - G_{L_k}(\vu_{k - 1}), \vu_k - \vu_{k - 1}} \\ 
    \geq \;& \frac{3}{4L_k}\norm{G_{L_k}(\vu_k) - G_{L_k}(\vu_{k - 1})}^2 \\
    = \;& \frac{1}{2L_k}\Big(\norm{G_{L_k}(\vu_k)}^2 - 2\innp{G_{L_k}(\vu_k), G_{L_k}(\vu_{k - 1})} + \norm{G_{L_k}(\vu_{k - 1})}^2\Big) \\
    & + \frac{1}{4L_k}\norm{G_{L_k}(\vu_k) - G_{L_k}(\vu_{k - 1})}^2.
\end{align*}
Multiplying $A_k$ on both sides and rearranging the terms, we obtain 
\begin{align*}
    \frac{A_k}{2L_k}\norm{G_{L_k}(\vu_k)}^2 \leq \;& \innp{G_{L_k}(\vu_k), A_k(\vu_k - \vu_{k - 1}) + \frac{A_k}{L_k}G_{L_k}(\vu_{k - 1})} \\
    & - \innp{G_{L_k}(\vu_{k - 1}), A_k(\vu_k - \vu_{k - 1})} \\ & - \frac{A_k}{2L_k}\norm{G_{L_k}(\vu_{k - 1})}^2 - \frac{A_k}{4L_k}\norm{G_{L_k}(\vu_k) - G_{L_k}(\vu_{k - 1})}^2.
\end{align*}
Plugging this into $\cc_k - \cc_{k - 1}$, we have 
\begin{align*}\label{ineq:diffCk-proof-main}
    \cc_k - \cc_{k - 1} \leq \;& \innp{G_{L_k}(\vu_k), A_k(\vu_k - \vu_{k - 1}) + \frac{A_k}{L_k}G_{L_k}(\vu_{k - 1}) + B_k(\vu_k - \vu_0)} \\
    & - \innp{G_{L_k}(\vu_{k - 1}), A_k(\vu_k - \vu_{k - 1})} + \innp{G_{L_{k - 1}}(\vu_{k - 1}), B_{k - 1}(\vu_{k - 1} - \vu_0)} \\
    & - \Big(\frac{A_k}{2L_k}\norm{G_{L_k}(\vu_{k - 1})}^2 + \frac{A_{k - 1}}{2L_{k - 1}}\norm{G_{L_{k - 1}}(\vu_{k - 1})}^2 \Big) \\
    & - \frac{A_k}{4L_k}\norm{G_{L_k}(\vu_k) - G_{L_k}(\vu_{k - 1})}^2.
\end{align*}
Since $\lambda_k = \frac{B_k}{A_k + B_k}$, we have 
\begin{align*}
    \vu_{k} = \frac{B_k}{A_k + B_k}\vu_0 + \frac{A_k}{A_k + B_k}\Big(\vu_{k - 1} - \tG_{L_{k - 1}}(\vu_{k - 1}) / L_{k}\Big), 
\end{align*}
which leads to $A_k(\vu_k - \vu_{k - 1}) + \frac{A_k}{L_k}G_{L_k}(\vu_{k - 1}) + B_k(\vu_k - \vu_0) = \frac{A_k}{L_k}\Big(G_{L_{k - 1}}(\vu_{k - 1}) - \tG_{L_{k - 1}}(\vu_{k - 1})\Big)$. 

Further, as $\frac{B_{k - 1}}{A_k} = \frac{B_k}{A_k + B_k}$ by https://www.overleaf.com/project/5fe36b9ad2991b26777b720dAssumption \ref{asp:para}, we have 
\begin{align*}
    \;& \innp{G_{L_k}(\vu_{k - 1}), A_k(\vu_k - \vu_{k - 1}) +  B_{k - 1}(\vu_{k - 1} - \vu_0)} \\
    = \;& A_k\innp{G_{L_k}(\vu_{k - 1}), \vu_k - \vu_{k - 1} +  \frac{B_{k - 1}}{A_k}(\vu_{k - 1} - \vu_0)} \\
    = \;& A_k\innp{G_{L_k}(\vu_{k - 1}), \vu_k - \frac{A_k}{A_k + B_k}\vu_{k - 1} - \frac{B_k}{A_k + B_k}\vu_0} \\
    = \;& A_k\innp{G_{L_k}(\vu_{k - 1}),  - \frac{A_k}{A_k + B_k}\tG_{L_{k - 1}}(\vu_{k - 1})/L_k}.
\end{align*}
Moreover, by Assumption \ref{asp:para}, we have $\frac{1}{L_k}(1 - \frac{2B_k}{A_k + B_k}) = \frac{A_{k - 1}}{A_k L_{k - 1}}$, so we obtain 
\begin{equation*}
\begin{aligned}
    \;& \innp{G_{L_k}(\vu_{k - 1}), A_k(\vu_k - \vu_{k - 1}) + B_{k - 1}(\vu_{k - 1} - \vu_0)} \\
    = \;& A_k\innp{G_{L_k}(\vu_{k - 1}),  - \frac{A_k}{A_k + B_k}\tG_{L_{k - 1}}(\vu_{k - 1})/L_k} \\
    = \;& -\frac{1}{2}\innp{G_{L_k}(\vu_{k - 1}), \Big(\frac{A_k}{L_k} + \frac{A_{k - 1}}{L_{k - 1}}\Big)\tG_{L_{k - 1}}(\vu_{k - 1})}. 
\end{aligned}
\end{equation*}
Having $L_k = L$ and denoting $G_{L} = G$ for short, we have 
\begin{equation*}
\begin{aligned}
    \cc_k - \cc_{k - 1} \leq \;& \innp{G(\vu_k), \frac{A_k}{L}\Big(G(\vu_{k - 1}) - \tG(\vu_{k - 1})\Big)} + \innp{G(\vu_{k - 1}), \frac{A_k + A_{k - 1}}{2L}\tG(\vu_{k - 1})} \\
    & - \frac{A_k + A_{k - 1}}{2L}\norm{G(\vu_{k - 1})}^2 - \frac{A_k}{4L}\norm{G(\vu_k) - G(\vu_{k - 1})}^2 \\
    = \;& \frac{A_k}{L}\innp{G(\vu_k), G(\vu_{k - 1}) - \tG(\vu_{k - 1})} - \frac{A_k}{4L}\norm{G(\vu_k) - G(\vu_{k - 1})}^2 \\
    & - \frac{A_k + A_{k - 1}}{2L}\innp{G(\vu_{k - 1}), G(\vu_{k - 1}) - \tG(\vu_{k - 1})} \\ 
    = \;& \frac{A_k}{L}\innp{G(\vu_k) - G(\vu_{k - 1}), G(\vu_{k - 1}) - \tG(\vu_{k - 1})} - \frac{A_k}{4L}\norm{G(\vu_k) - G(\vu_{k - 1})}^2 \\
    & + \frac{A_k - A_{k - 1}}{2L}\innp{G(\vu_{k - 1}), G(\vu_{k - 1}) - \tG(\vu_{k - 1})}.
\end{aligned}
\end{equation*}
Since $2\innp{p, q} + \norm{p}^2 \leq \norm{q}^2$ for any $p, q \in \rr^d$, we have 
\begin{align*}
    \cc_k - \cc_{k - 1} \leq \;& \frac{A_k}{L}\norm{G(\vu_{k - 1}) - \tG(\vu_{k - 1})}^2 + \frac{A_k - A_{k - 1}}{2L}\innp{G(\vu_{k - 1}), G(\vu_{k - 1}) - \tG(\vu_{k - 1})}, 
\end{align*}
thus completing the proof.
\end{proof}

\begin{theorem}\label{thm:coco-rate-constrained}
Given an arbitrary $\vu_0 \in \rr^d,$ suppose that iterates $\vu_k$ evolve according to Halpern iteration for the constrained setup from Eq.~\eqref{eq:cons-halpern} for $k \geq 1,$ where $L_k = L$ and $\lambda_k = \frac{1}{k+1}.$ 
 %
 Given $\epsilon > 0,$ if we have that $\ee[\|F(\vu_0) - \tF(\vu_0)\|^2] \leq \frac{\epsilon^2}{8}$ and 
  $\ee\big[\big\|{F(\vu_k) - \tF(\vu_k)}\big\|^2\big] \leq \frac{\epsilon^2}{k}$ for all $k \geq 1$, then for all $k \geq 1,$ %
\begin{equation}
    \ee[\norm{G(\vu_k))}] \leq \frac{\Lambda_0}{k} + \Lambda_1\epsilon,
\end{equation}
where $\Lambda_{0} = 20 L\norm{\ui - \um}$ and $\Lambda_1 = \sqrt{13}$. 
As a result, stochastic Halpern iteration from Eq.~\eqref{eq:Halpern} returns a point $\vu_k$ such that $\ee[\norm{G(\vu_k)}] \leq 5\epsilon$ after at most $N = \lceil \frac{2\Lambda_0}{\epsilon}\rceil = \co\big(\frac{L\|\vu_0 -\vu^*\|}{\epsilon}\big)$ iterations.
\end{theorem}
\begin{proof}
First note that since $\cu$ is convex and closed, and $\vu_k - \tG(\vu_k) / L = \proj\Big(\vu_k - \frac{1}{L}\tF(\vu_k)\Big)$, then we have for $\forall k > 0$.
\begin{equation*}
    \begin{aligned}
        \vu_{k + 1} = \;& \lambda_{k + 1}\vu_0 + (1 - \lambda_{k + 1})\Big(\vu_k - \tG(\vu_k) / L\Big) \\
        = \;& \lambda_{k + 1}\vu_0 + (1 - \lambda_{k + 1})\proj\Big(\vu_k - \frac{1}{L}\tF(\vu_k)\Big) \in \cu.
    \end{aligned}
\end{equation*}
Then we come to prove the convergence. By Jensen's Inequality, we have for $k \geq 1$ 
\begin{align*}
    \ee[\norm{G(\vu_k))}] \leq \Big(\ee[\norm{G(\vu_k)}^2]\Big)^{\frac{1}{2}}.
\end{align*}
So it suffices to show that there exists $\Lambda_0$ and $\Lambda_1$ such that for all $k \geq 1$ 
\begin{align*}
     \Big(\ee[\norm{G(\vu_k)}^2]\Big)^{\frac{1}{2}} \leq \frac{\Lambda_0}{k} + \Lambda_1\epsilon.
\end{align*}
We prove it by induction. First, we consider the basis case $k = 1$ in which $\vu_1 = \vu_0 - \frac{1}{2L}\tG(\vu_0)$, so we have $\cc_1 = \frac{1}{L}\norm{G(\vu_1)}^2 + 2\innp{G(\vu_1), \vu_1 - \vu_0} = \frac{1}{L}\Big(\norm{G(\vu_1)}^2 - \innp{G(\vu_1), \tG(\vu_0)}\Big)$. Also, since the operator $G$ is cocoercive with parameter $\frac{3}{4L}$, thus cocoercive with $\frac{1}{2L}$, we have 
\begin{align*}
    \norm{G(\vu_1) - G(\vu_0)}^2 \leq 2L\innp{G(\vu_1) - G(\vu_0), \vu_1 - \vu_0} = \innp{G(\vu_1) - G(\vu_0), -\tG(\vu_0)}.
\end{align*}
Expanding and rearranging the terms, we have 
\begin{align*}
    \norm{G(\vu_1)}^2 \leq \innp{G(\vu_0), \tG(\vu_0) - G(\vu_0)} + 2\innp{G(\vu_1), G(\vu_0)} - \innp{G(\vu_1), \tG(\vu_0)}.
\end{align*}
Subtracting $\innp{G(\vu_1), \tG(\vu_0)}$ and taking expectation with respect to all randomness on both sides, we have 
\begin{equation*}
\begin{aligned}
    \;& \ee\Big[\norm{G(\vu_1)}^2 - \innp{G(\vu_1), \tG(\vu_0)}\Big] \\
    \leq \;& \ee\Big[\innp{G(\vu_0), \tG(\vu_0) - G(\vu_0)} + 2\innp{G(\vu_1), G(\vu_0)} - 2\innp{G(\vu_1), \tG(\vu_0)}\Big] \\
    \overset{(\romannumeral1)}{\leq} \;& \ee\Big[\frac{1}{2}\norm{G(\vu_0)}^2 + \frac{1}{2}\norm{G(\vu_1)}^2 + \frac{5}{2}\norm{G(\vu_0) - \tG(\vu_0)}^2\Big],
\end{aligned}
\end{equation*}
where for $(\romannumeral1)$ we use Young's Inequality. Since $\vu^*$ is the solution of monotone inclusion, then we have $G(\vu^*) = 0$. So we have 
\begin{align*}
    \norm{G(\vu_0)}^2 = \norm{G(\vu_0) - G(\vu^*)}^2 \overset{(\romannumeral1)}{\leq} 10L^2\norm{\vu_0 - \vu^*}^2, 
\end{align*}
where $(\romannumeral1)$ can be verified by Young's Inequality and using the fact that the projection operator is non-expansive. Also using the results in Lemma \ref{lemma:gradient-mapping-error}, we obtain that 
\begin{equation*}
\ee[\cc_1] \leq \frac{1}{L}\ee\big[5L^2\norm{\vu_0 - \vu^*}^2 + \frac{1}{2}\norm{G(\vu_1)}^2 + \frac{5}{2}\norm{F(\vu_0) - \tF(\vu_0)}^2\big].
\end{equation*}
Proceeding similar to Lemma \ref{lemma:error}, we have 
\begin{align*}
    \ee[\norm{G(\vu_1)}^2] \leq \;& \frac{2B_1L_1}{A_1}\norm{\vu_0 - \vu^*}\ee[\norm{G(\vu_1)}] \\
    & + \frac{2L_1}{A_1 L}\ee\Big[5L^2\norm{\vu_0 - \vu^*}^2 + \frac{1}{2}\norm{G(\vu_1)}^2 + \frac{5}{2}\norm{F(\vu_0) - \tF(\vu_0)}^2\Big] \\
    = \;& 2L\norm{\vu_0 - \vu^*}\ee[\norm{G(\vu_1)}] \\
    & + \ee\Big[5L^2\norm{\vu_0 - \vu^*}^2 + \frac{1}{2}\norm{G(\vu_1)}^2 + \frac{5}{2}\norm{F(\vu_0) - \tF(\vu_0)}^2\Big].
\end{align*}
Subtracting $\ee[\frac{1}{2}\norm{G(\vu_1)}^2]$ on both sides and using the fact that $\ee[\norm{G(\vu_1)}] \leq \big(\ee[\norm{G(\vu_1)}^2]\big)^{\frac{1}{2}}$ and $\ee[\norm{F(\vu_0) - \tF(\vu_0)}^2] \leq \frac{\epsilon^2}{8}$, we have 
\begin{align*}
    \ee[[\norm{G(\vu_1)}^2] \leq 4L\norm{\vu_0 - \vu^*}\big(\ee[\norm{G(\vu_1)}^2]\big)^{\frac{1}{2}} + 10L^2\norm{\vu_0 - \vu^*}^2 + \frac{5\epsilon^2}{8},
\end{align*}
which is a quadratic function with respect to $(\ee[\norm{G(\vu_1)}^2])^{\frac{1}{2}}$. So by its larger root we have 
\begin{align*}
    (\ee[\norm{G (\vu_1)}^2])^{\frac{1}{2}} \leq \;& 2L\norm{\vu_0 - \vu^*} + \frac{1}{2}\sqrt{56L^2\norm{\vu_0 - \vu^*}^2 + \frac{5}{2}\epsilon^2} \\
    \leq \;& 2L\norm{\vu_0 - \vu^*} + \frac{1}{2}(2\sqrt{14}L\norm{\vu_0 - \vu^*} + \frac{\sqrt{10}}{2}\epsilon) \\
    \leq \;& 6L\norm{\vu_0 - \vu^*} + \epsilon \leq \Lambda_0 + \Lambda_1\epsilon.
\end{align*}
So the result holds for the basis case. Moreover, we can get a bound for $\ee[\cc_1]$ as follows 
\begin{align*}
    \ee[\cc_1] \leq \;& \frac{1}{L}\ee\Big[5L^2\norm{\vu_0 - \vu^*}^2 + \frac{1}{2}\norm{G(\vu_1)}^2 + \frac{5}{2}\norm{F(\vu_0) - \tF(\vu_0)}^2\Big] \\
    \overset{(\romannumeral1)}{\leq} \;& 5L\norm{\vu_0 - \vu^*}^2 + \frac{1}{2L}\Big(50L^2\norm{\vu_0 - \vu^*}^2 + \frac{5}{4}\epsilon^2\Big) + \frac{5}{2L}\frac{\epsilon^2}{8} \\
    \leq \;& 30L\norm{\vu_0 - \vu^*}^2 + \frac{\epsilon^2}{L},
\end{align*}
where $(\romannumeral1)$ can be verified by using the bound we get above for $\ee[\norm{G(\vu_1)}^2]$ and applying Young's Inequaltiy, and the fact that $\ee[\norm{F(\vu_0) - \tF(\vu_0)}^2] \leq \frac{\epsilon^2}{8}$. 

Assume that the result holds for all $1 \leq i \leq k - 1$, then we come to prove the case $k$. By Lemma \ref{lemma:ck-cons} we have for $\forall i \geq 2$ 
\begin{align*}
    \cc_{i} - \cc_{i - 1} \leq \;& \frac{A_i}{L}\norm{G(\vu_{i - 1}) - \tG(\vu_{i - 1})}^2 + \frac{A_i - A_{i - 1}}{2L}\innp{G(\vu_{i - 1}), G(\vu_{i - 1}) - \tG(\vu_{i - 1})} \\
    \overset{(\romannumeral1)}{\leq} \;& \frac{3i(i + 1)}{L}\norm{G(\vu_{i - 1}) - \tG(\vu_{i - 1})}^2 + \frac{i}{8L(i + 1)}\norm{G(\vu_{i - 1})}^2 \\
    \overset{(\romannumeral2)}{\leq} \;& \frac{3i(i + 1)}{L}\norm{F(\vu_{i - 1}) - \tF(\vu_{i - 1})}^2 + \frac{i}{8L(i + 1)}\norm{G(\vu_{i - 1})}^2, 
\end{align*}
where we use Young's Inequality and $A_i = i(i + 1)$ for $(\romannumeral1)$, and $(\romannumeral2)$ is due to Lemma \ref{lemma:gradient-mapping-error}. Taking expectation with respect to all randomness on both sides and telescoping from $i = 2$ to $k$, we obtain 
\begin{align*}
    \ee[\cc_k] \leq \;& \ee\Big[\cc_1 + \sum_{i = 2}^{k}\Big( \frac{3i(i + 1)}{L}\norm{F(\vu_{i - 1}) - \tF(\vu_{i - 1})}^2 + \frac{i}{8L(i + 1)}\norm{G(\vu_{i - 1})}^2 \Big)\Big] \\
    \leq \;& 30L\norm{\vu_0 - \vu^*}^2 + \frac{\epsilon^2}{L} + \ee\Big[\sum_{i = 2}^{k}\frac{i}{8L(i + 1)}\norm{G(\vu_{i - 1})}^2\Big]\\
    & + \ee\Big[\sum_{i = 2}^{k}\frac{3i(i + 1)}{L}\norm{F(\vu_{i - 1}) - \tF(\vu_{i - 1})}^2\Big].
\end{align*}
By Corollary~\ref{cor:variance-bound}, we have 
\begin{align*}
    \;& \ee\Big[\sum_{i = 2}^{k}\frac{3i(i + 1)}{L}\norm{F(\vu_{i - 1}) - \tF(\vu_{i - 1})}^2\Big] \\
    \leq \;& \sum_{i = 2}^{k}\frac{3i(i + 1)}{L}\frac{\epsilon^2}{i - 1} \overset{(\romannumeral1)}{\leq} \sum_{i = 2}^{k}\frac{6(i + 1)\epsilon^2}{L} = \frac{3(k + 4)(k - 1)\epsilon^2}{L},
\end{align*}
where $(\romannumeral1)$ is because $\frac{i}{i - 1} \leq 2$ for all $i \geq 2$. By induction, we have 
\begin{align*}
    \;& \ee\Big[\sum_{i = 2}^{k}\frac{i}{8L(i + 1)}\norm{G(\vu_{i - 1})}^2\Big] \\
    \overset{(\romannumeral1)}{\leq} \;& \sum_{i = 2}^{k}\frac{1}{8L}\Big(2\frac{\Lambda_0^2}{(i - 1)^2} + 2\Lambda_1^2\epsilon^2\Big) \\
    \overset{(\romannumeral2)}{\leq} \;& \frac{1}{4L}\Big(\Lambda_0^2\frac{\pi^2}{6} + (k - 1)\Lambda_1^2\epsilon^2\Big) = \frac{1}{L}\Big(\frac{\Lambda_0^2\pi^2}{24} + \frac{(k - 1)\Lambda_1^2\epsilon^2}{4}\Big),
\end{align*}
where $(\romannumeral1)$ follows from induction and $\frac{i}{i + 1} \leq 1$, and $(\romannumeral2)$ is due to $\sum_{i = 2}^{k}\frac{1}{(i - 1)^2} \leq \sum_{i = 1}^{\infty}\frac{1}{i^2} = \frac{\pi^2}{6}$. We now obtain 
\begin{align*}
    \ee[\cc_k] \leq 30L\norm{\vu_0 - \vu^*}^2 + \frac{\epsilon^2}{L} + \frac{3(k + 4)(k - 1)\epsilon^2}{L} + \frac{1}{L}\Big(\frac{\Lambda_0^2\pi^2}{24} + \frac{(k - 1)\Lambda_1^2\epsilon^2}{4}\Big).
\end{align*}
By the same derivation of Lemma \ref{lemma:error}, we have 
\begin{align*}
    \ee\big[\norm{G(\vu_k)}^2\big] \leq \;& \frac{B_kL_k}{A_k}\norm{\vu_0 - \vu^*}\ee[\norm{G(\vu_k)}] \\
    & + \frac{L_k}{A_k}\Big(30L\norm{\vu_0 - \vu^*}^2 + \frac{\epsilon^2}{L} + \frac{3(k + 4)(k - 1)\epsilon^2}{L} + \frac{\Lambda_0^2\pi^2}{24L} + \frac{(k - 1)\Lambda_1^2\epsilon^2}{4L}\Big) \\
    = \;& \frac{L}{k}\norm{\vu_0 - \vu^*}\ee[\norm{G(\vu_k)}] + \frac{30L^2\norm{\vu_0 - \vu^*}^2}{k(k + 1)}\\
    & + \frac{1}{k(k + 1)}\Big(\epsilon^2 + 3(k + 4)(k - 1)\epsilon^2 + \frac{\Lambda_0^2\pi^2}{24} + \frac{(k - 1)\Lambda_1^2\epsilon^2}{4}\Big) \\
    \overset{(\romannumeral1)}{\leq} \;& \frac{L}{k}\norm{\vu_0 - \vu^*}\ee[\norm{G(\vu_k)}] + \Big(\frac{30L^2\norm{\vu_0 - \vu^*}^2}{k^2} + (11 + \frac{\Lambda_1^2}{4k})\epsilon^2 + \frac{\Lambda_0^2\pi^2}{24k^2}\Big),
\end{align*}
where $(\romannumeral1)$ is due to $\frac{1}{k(k + 1)} \leq \frac{1}{k^2}$, $\frac{3(k + 1)(k - 1)}{k(k + 1)} \leq 10$ and $\frac{k - 1}{k(k + 1)} \leq \frac{1}{k}$. Since $\ee[\norm{G(\vu_k)}] \leq (\ee[\norm{G(\vu_k)}^2])^{\frac{1}{2}}$ by Jensen's Inequality, we have 
\begin{align*}
    \ee\big[\norm{G(\vu_k)}^2\big] \leq \;& \frac{L}{k}\norm{\vu_0 - \vu^*}\Big(\ee\big[\norm{G(\vu_k)}^2\big]\Big)^{\frac{1}{2}} \\
    & + \Big(\frac{30L^2\norm{\vu_0 - \vu^*}^2}{k^2} + (11 + \frac{\Lambda_1^2}{4k})\epsilon^2 + \frac{\Lambda_0^2\pi^2}{24k^2}\Big),
\end{align*}
which is a quadratic function with respect to $(\ee[\norm{G(\vu_k)}^2])^{\frac{1}{2}}$. So by its larger root we obtain 
\begin{align*}
    \;& \Big(\ee\big[\norm{G(\vu_k)}^2\big]\Big)^{\frac{1}{2}} \\
    \leq \;& \frac{L}{2k}\norm{\vu_0 - \vu^*} + \frac{1}{2}\sqrt{\frac{L^2}{k^2}\norm{\vu_0 - \vu^*}^2 + 4\Big(\frac{30L^2\norm{\vu_0 - \vu^*}^2}{k^2} + (11 + \frac{\Lambda_1^2}{4k})\epsilon^2 + \frac{\Lambda_0^2\pi^2}{24k^2}\Big)} \\
    \overset{(\romannumeral1)}{\leq} \;& \frac{L}{k}\norm{\vu_0 - \vu^*} + \Big(\frac{\sqrt{30}L\norm{\vu_0 - \vu^*}}{k} + \sqrt{11 + \frac{\Lambda_1^2}{4k}}\epsilon + \frac{\Lambda_0\pi}{2\sqrt{6}k}\Big) \\
    = \;& \frac{(1 + \sqrt{30})L\norm{\vu_0 - \vu^*} + \frac{\Lambda_0\pi}{2\sqrt{6}}}{k} + \sqrt{11 + \frac{\Lambda_1^2}{4k}}\epsilon \\
    \overset{(\romannumeral2)}{\leq} \;& \frac{\Lambda_0}{k} + \Lambda_1\epsilon,
\end{align*}
where $(\romannumeral1)$ is due to the fact that $\sqrt{\sum_{i = 1}^{n}X_i^2} \leq \sum_{i = 1}^{n}|X_i|$, and $(\romannumeral2)$ is because of our choice of $\Lambda_0, \Lambda_1$. Hence, the result also holds for the case $k$. Then by induction we know that the result holds for all $k \geq 1$.

Finally, when $k \geq \frac{2\Lambda_0}{\epsilon}$, we have $\frac{\Lambda_0}{k} \leq \epsilon / 2$. Also, since we have $\Lambda_1 = \sqrt{13}$, we obtain 
\begin{align*}
    \ee[\norm{G(\vu_k)}] \leq \Big(\frac{1}{2} + \sqrt{13}\Big)\epsilon \leq 5\epsilon.
\end{align*}
Hence, the total number of iterations needed to attain $5\epsilon$ norm of the operator is 
\begin{align*}
    N = \Big\lceil \frac{2\Lambda_0}{\epsilon} \Big\rceil \leq \frac{2\Lambda_0}{\epsilon} + 1 = \frac{\Delta}{\epsilon},
\end{align*}
thus completing the proof.
\end{proof}

\begin{corollary}\label{cor:coco-constrained}
Given an arbitrary $\vu_0 \in \rr^d,$ suppose that iterates $\vu_k$ evolve according to Halpern iteration from Eq.~\eqref{eq:cons-halpern} for $k \geq 1,$ where  $L_k = L$, and $\lambda_k = \frac{1}{k+1}$. Assume further that the stochastic estimate $\tF(\vu)$ is defined according to Eq.~\eqref{eq:PAGE}, with its parameters set according to Corollary~\ref{cor:variance-bound}. 
 Then, given any $\epsilon > 0,$ stochastic Halpern iteration from Eq.~\eqref{eq:cons-halpern} returns a point $\vu_k$ such that $\ee[\norm{G(\vu_k)}] \leq 5\epsilon$ with at most $\co(\frac{\sigma^2L\norm{\vu_0 - \vu^*} + L^3\norm{\vu_0 - \vu^*}^3}{\epsilon^3})$ oracle queries to $\hF$
in expectation. 
\end{corollary}
\begin{proof}
Let $m_k$ be the number of stochastic queries made by the estimator from Eq.~\eqref{eq:PAGE} at iteration $k$. Conditional on $\cf_k$ and using Corollary~\ref{cor:variance-bound}, since each stochastic gradient mapping $\tG(\vu_k)$ only involves one PAGE invariant stochastic estimate $\tF(\vu_k)$, we have 
\begin{align*}
    \ee\big[m_{k + 1} | \cf_{k - 1}\big] = \;& p_k\Big\lceil\frac{8\sigma^2}{p_k\epsilon^2}\Big\rceil + 2(1 - p_k)\Big\lceil\frac{8L^2\norm{\vu_k - \vu_{k - 1}}^2}{p_k^2\epsilon^2}\Big\rceil \\
    \overset{(\romannumeral1)}{\leq} \;& p_k\Big(\frac{8\sigma^2}{p_k\epsilon^2} + 1\Big) + 2(1 - p_k)\Big(\frac{8L^2\norm{\vu_k - \vu_{k - 1}}^2}{p_k^2\epsilon^2} + 1\Big), 
\end{align*}
where $(\romannumeral1)$ is due to the fact that $\lceil x \rceil \leq x + 1$ for any $x \in \rr$. Taking expectation with respect to all randomness on both sides, and rearranging the terms, we obtain \begin{align*}
    \ee[m_{k + 1}] \leq \frac{8\sigma^2}{\epsilon^2} + \frac{16(1 - p_k)L^2\ee\big[\norm{\vu_k - \vu_{k - 1}}^2\big]}{p_k^2\epsilon^2} + 2.
\end{align*}
By the same derivation as Lemma \ref{lemma:uk}, we have 
\begin{align}\label{ineq:uk-cons}
    \norm{\vu_k - \vu_{k - 1}}^2 \leq 
    \left\{\begin{array}{ll}
         \frac{1}{4L^2}\norm{\tG(\vu_0)}^2 & \text{if } k = 1, \\
         & \\
         \frac{2k^2}{L^2(k + 1)^2}\norm{\tG(\vu_{k - 1})}^2 + \sum_{i = 0}^{k - 2}\frac{2(i + 1)^2}{k(k + 1)^2L^2}\norm{\tG(\vu_i)}^2 & \text{if } k \geq 2.
    \end{array}\right.
\end{align}
By the corollary assumptions, we have $\ee[\norm{G(\vu_i)}^2] \leq \co(\frac{L^2\norm{\vu_0 - \vu^*}^2}{i^2})$ for $i \leq k - 1$ by Theorem~\ref{thm:coco-rate-constrained}. Then we obtain 
\begin{equation*}
\begin{aligned}
    \ee\Big[\norm{\tG(\vu_i)}^2\Big] \leq \;& 2\ee\Big[\norm{G(\vu_i)}^2\Big] + 2\ee\Big[\norm{\tG(\vu_i) - G(\vu_i)}^2\Big] \\
    \overset{(\romannumeral1)}{\leq} \;& 2\ee\Big[\norm{G(\vu_i)}^2\Big] + 2\ee\Big[\norm{\tF(\vu_i) - F(\vu_i)}^2\Big] \\
    \leq \;& \co\Big(\frac{L^2\norm{\vu_0 - \vu^*}^2}{i^2}\Big), 
\end{aligned}
\end{equation*}
where $(\romannumeral1)$ is due to Lemma~\ref{lemma:gradient-mapping-error}. 

Plugging it into Inequality~\eqref{ineq:uk-cons}, we have $\ee[\norm{\vu_k - \vu_{k - 1}}^2] = \co(\frac{\norm{\vu_0 - \vu^*}^2}{k^2})$, which leads to 
\begin{equation*}
    \ee[m_{k + 1}] = \co\Big(\frac{\sigma^2 + L^2\norm{\vu_0 - \vu^*}^2}{\epsilon^2}\Big) 
\end{equation*}
using $p_k = \frac{2}{k + 1} = \co(1/k)$. 

Further, by Theorem~\ref{thm:coco-rate-constrained}, the total number of iterations to attain $5\epsilon$ norm of the operator in expectation is $N = \co(\frac{L\norm{\vu_0 - \vu^*}}{\epsilon})$ and $m_1 = S_1^{(0)} = \co(\frac{\sigma^2}{\epsilon^2})$, we conclude that the total number of stochastic queries to $F$ is 
\begin{equation*}
    \ee[M] = \ee\Big[\sum_{k = 1}^N m_k\Big] = \co\Big(\frac{\sigma^2L\norm{\vu_0 - \vu^*} + L^3\norm{\vu_0 - \vu^*}^3}{\epsilon^3}\Big), 
\end{equation*}
thus completing the proof.
\end{proof}


\section{Omitted proofs from Section~\ref{sec:two_step_Halpern}}\label{appx:monotone}
We use the potential function, previously used by~\citep{tran2021halpern}, 
\begin{equation}\label{eq:mono-vk}
\begin{aligned}
\cv_k := \;& A_{k}\|F(\vu_{k})\|^{2} + B_{k}\innp{F(\vu_{k}), \vu_{k} - \vu_{0}} + c_{k} L^{2}\|\vu_{k} - \vv_{k-1}\|^{2}, 
\end{aligned}
\end{equation}
prove Theorem~\ref{thm:mono-rate}. Here $A_k$, $B_k$ and $c_k$ are positive parameters to be determined later. We start by bounding the change of $\cv_k$ under the following assumption on the parameters.
\begin{assumption}\label{asp:para-mono}
 $\lambda_k \in \left[0, 1\right)$, $\eta_k > 0$, and $A_k$, $B_k$ and $c_k$ are positive parameters satisfying $B_{k + 1} = \frac{B_k}{1 - \lambda_k}$, $A_k = \frac{B_k \eta_k}{2\lambda_k}$,
\begin{equation}\label{eq:asp-mono}
    0 < \eta_{k + 1} = \frac{\big(1 - \lambda_{k}^{2} - M\eta_{k}^{2}\big) \lambda_{k+1} \eta_{k}}{\big(1 - M\eta_{k}^{2}\big) \big(1 - \lambda_{k}\big) \lambda_{k}}, \quad M\eta_k^2 + \lambda_k^2 < 1, \quad \text{and } \quad \eta_{k + 1} \leq \frac{\lambda_{k + 1}\big(1 - \lambda_k\big)}{M \lambda_k \eta_k},
\end{equation}
where $M = 3L^2(2 + \theta)$ and $\theta > 0$ is some parameter that can be determined later.
\end{assumption}

The following lemma gives a bound on the difference between the potential function values at two consecutive iterations with the control of the parameters above. 
\begin{restatable}{lemma}{diffMono}
\label{lemma:mono-diff}
Let $\cv_k$ be defined as in Eq.~\eqref{eq:mono-vk}, where the parameters satisfy Assumption \ref{asp:para-mono}. Then the difference of potential function between two consecutive iterations can be bounded by
\begin{equation}
\begin{aligned}
\cv_{k + 1} - \cv_k  \leq \;& -L^{2}\bigg(\frac{\theta A_{k}}{M \eta_{k}^{2}} - c_{k+1}\bigg)\|\vu_{k+1} - \vv_{k}\|^{2} - L^{2}(c_{k} - A_{k})\|\vu_{k} - \vv_{k-1}\|^{2} \\
& + \frac{2A_k}{M\eta_k^2}\norm{F(\vv_k) - \tF(\vv_k)}^2 + A_k\norm{F(\vv_{k - 1}) - \tF(\vv_{k - 1})}^2.
\end{aligned}
\end{equation}
\end{restatable}
 \begin{proof}
By the iteration scheme in Eq.~\eqref{eq:mono-uk}, we can deduce the following identities:
\begin{equation}\label{eq:mono-uk-diff}
\left\{\begin{aligned} 
\vu_{k+1} - \vu_{k} &= \lambda_{k}(\vu_{0} - \vu_{k}) - \eta_{k} \tF(\vv_{k}) \\ 
\vu_{k+1} - \vu_{k} &= \frac{\lambda_{k}}{1 - \lambda_{k}}(\vu_{0} - \vu_{k+1}) - \frac{\eta_{k}}{1 - \lambda_{k}} \tF(\vv_{k}) \\
\vu_{k+1} - \vv_{k} &= -\eta_{k}\left(\tF\left(\vv_{k}\right) - \tF\left(\vv_{k-1}\right)\right) \end{aligned}\right.
\end{equation}
Further, by the definition of the potential function $\mathcal{V}_k$, we can write
\begin{equation}\label{eq:mono-vk-diff}
\begin{aligned} 
\cv_k - \cv_{k+1}
= \;& \underbrace{A_{k}\left\|F\left(\vu_{k}\right)\right\|^{2} - A_{k+1}\left\|F\left(\vu_{k+1}\right)\right\|^{2}}_{\mathcal{T}_{[1]}} \\
& + \underbrace{B_{k}\innp{ F\left(\vu_{k}\right), \vu_{k} - \vu_{0}} - B_{k+1}\innp{ F\left(\vu_{k+1}\right), \vu_{k+1} - \vu_{0}}}_{\mathcal{T}_{[2]}} \\
& + c_{k} L^{2}\left\|\vu_{k} - \vv_{k-1}\right\|^{2}-c_{k+1} L^{2}\left\|\vu_{k+1} - \vv_{k}\right\|^{2}.
\end{aligned}
\end{equation}
To obtain the claimed bound, in the rest of the proof we focus on bounding $\mathcal{T}_{[1]}$ and $\mathcal{T}_{[2]}.$

To bound $\mathcal{T}_{[1]}$, by the Lipschitz continuity of $F$, we have 
\begin{equation*}
  \left\|F\left(\vu_{k + 1}\right) - F\left(\vv_k\right)\right\|^2 \leq L^2\left\|\vu_{k + 1} - \vv_k\right\|^2 = L^2\eta_k^2\left\|\tF\left(\vv_k\right) - \tF\left(\vv_{k - 1}\right)\right\|^2,
\end{equation*}
where in the last step we used the third identity from Eq.~\eqref{eq:mono-uk-diff}.
Further, for any $\theta > 0$ 
\begin{equation}\label{ineq:T1-bound-mono}
\begin{aligned}
\;& \left\|F\left(\vu_{k+1}\right) - \tF\left(\vv_{k}\right)\right\|^{2} + \theta L^{2}\left\|
\vu_{k+1} - \vv_{k}\right\|^{2} \\
\leq \;& 2\left\|F\left(\vu_{k+1}\right) - F\left(\vv_{k}\right)\right\|^{2} + 2\left\|F\left(\vv_{k}\right) - \tF\left(\vv_{k}\right)\right\|^{2} + \theta L^{2}\left\|
\vu_{k+1} - \vv_{k}\right\|^{2} \\
\leq & \eta_{k}^{2} L^{2}(2 + \theta)\left\|\tF\left(\vv_{k}\right) - \tF\left(\vv_{k-1}\right)\right\|^{2} + 2\left\|F\left(\vv_{k}\right) - \tF\left(\vv_{k}\right)\right\|^{2},
\end{aligned}
\end{equation}
where again in the last step we used the third identity from Eq.~\eqref{eq:mono-uk-diff}.
Notice that 
\begin{equation*}
\begin{aligned}
     \;& \left\|\tF\left(\vv_{k}\right) - \tF\left(\vv_{k-1}\right)\right\|^{2} \\
     = \;& \left\|\tF\left(\vv_{k}\right) - F\left(\vu_k\right) + F\left(\vu_k\right) - F\left(\vv_{k - 1}\right) + F\left(\vv_{k - 1}\right) - \tF\left(\vv_{k-1}\right)\right\|^{2} \\
    \leq \;& 3\left\|\tF\left(\vv_{k}\right) - F\left(\vu_k\right)\right\|^{2} + 3\left\|F\left(\vu_k\right) - F\left(\vv_{k - 1}\right)\right\|^{2} + 3\left\|F\left(\vv_{k - 1}\right) - \tF\left(\vv_{k - 1}\right)\right\|^{2} \\
    \leq \;& 3\Big(\left\|F\left(\vu_k\right)\right\|^2 - 2\innp{F\left(\vu_k\right), \tF\left(\vv_k\right)} + \left\|\tF\left(\vv_k\right)\right\|^2\Big) + 3L^2\left\|\vu_k - \vv_{k - 1}\right\|^2 \\
    & + 3\left\|F\left(\vv_{k - 1}\right) - \tF\left(\vv_{k - 1}\right)\right\|^{2}.
\end{aligned}
\end{equation*}
Let $M := 3L^2(2 + \theta)$. Expanding the term $\|{F(\vu_{k + 1}) - \tF(\vv_k)}\|^2$ on the LHS in Inequality~\eqref{ineq:T1-bound-mono} and combining with the inequality above, we have 
\begin{equation*}
\begin{aligned}
   \;& \left\|F\left(\vu_{k+1}\right)\right\|^{2} + \left\|\tF\left(\vv_{k}\right)\right\|^{2} - 2\innp{F\left(\vu_{k+1}\right), \tF\left(\vv_{k}\right)} +\theta L^{2}\left\|\vu_{k+1} - \vv_{k}\right\|^{2} \\
   \leq \;& M\eta_k^2\left(\left\|F\left(\vu_k\right)\right\|^2 - 2\innp{F\left(\vu_k\right), \tF\left(\vv_k\right)} + \left\|\tF\left(\vv_k\right)\right\|^2\right) + ML^2\eta_k^2\left\|\vu_k - \vv_{k - 1}\right\|^2 \\
   & + M\eta_k^2\left\|F\left(\vv_{k - 1}\right) - \tF\left(\vv_{k - 1}\right)\right\|^{2} +  2\left\|F\left(\vv_{k}\right) - \tF\left(\vv_{k}\right)\right\|^{2}.
\end{aligned}
\end{equation*}
Multiplying both sides by $\frac{A_k}{M\eta_k^2}$, rearranging this inequality and subtracting $A_{k + 1}\norm{F(\vu_{k + 1})}^2$ on both sides, we obtain 
\begin{equation}\label{ineq:mono-t1}
\begin{aligned}
\mathcal{T}_{[1]} = \;& A_k\left\|F\left(\vu_k\right)\right\|^2 - A_{k + 1}\norm{F(\vu_{k + 1})}^2   \\
\geq \;& \left(\frac{A_k}{M\eta_k^2} - A_{k + 1}\right)\left\|F\left(\vu_{k + 1}\right)\right\|^2 + \frac{A_k\left(1 - M\eta_k^2\right)}{M\eta_k^2}\left\|\tF\left(\vv_{k}\right)\right\|^2 \\
& - \frac{2A_k\left(1 - M\eta_k^2\right)}{M\eta_k^2}\innp{F\left(\vu_{k + 1}\right), \tF\left(\vv_k\right)} - 2A_k\innp{F\left(\vu_{k + 1}\right) - F\left(\vu_k\right), \tF\left(\vv_k\right)} \\
& + \frac{A_k\theta L^2}{M\eta_k^2}\norm{\vu_{k + 1} - \vv_k}^2 - A_k L^2\norm{\vu_k - \vv_{k - 1}}^2 \\
& - \frac{2A_k}{M\eta_k^2}\norm{F(\vv_k) - \tF(\vv_k)}^2 - A_k\norm{F(\vv_{k - 1}) - \tF(\vv_{k - 1})}^2.
\end{aligned}
\end{equation}

To bound $\mathcal{T}_{[2]}$, notice that $F$ is monotone, so we have 
\begin{equation*}
\innp{F\left(\vu_{k+1}\right), \vu_{k+1} - \vu_{k}} \geq \innp{F\left(\vu_{k}\right), \vu_{k+1} - \vu_{k}}.
\end{equation*}
Using the first line in Eq.~\eqref{eq:mono-uk-diff} for the RHS and the second line for the LHS, we can obtain 
\begin{equation*}
\begin{aligned}
\frac{\lambda_{k}}{1-\lambda_{k}}\innp{F\left(\vu_{k+1}\right), \vu_{0} - \vu_{k+1}} \geq \;&  \lambda_{k}\innp{F\left(\vu_{k}\right), \vu_{0} - \vu_{k}} - \eta_{k}\innp{F\left(\vu_{k}\right), \tF\left(\vv_{k}\right)} \\
& + \frac{\eta_{k}}{1 - \lambda_{k}}\innp{F\left(\vu_{k + 1}\right), \tF\left(\vv_{k}\right)}.
\end{aligned}
\end{equation*}
Multiplying both sides by $\frac{B_k}{\lambda_k}$ and using that $B_{k + 1} = \frac{B_k}{1 - \lambda_k}$ by Assumption~\ref{asp:para-mono}, we have 
\begin{equation}\label{ineq:mono-t2}
\begin{aligned} 
\mathcal{T}_{[2]} \geq \;& \frac{B_{k} \eta_{k}}{\lambda_{k}\left(1 - \lambda_{k}\right)}\innp{ F\left(\vu_{k+1}\right), \tF\left(\vv_{k}\right)} - \frac{B_{k} \eta_{k}}{\lambda_{k}}\innp{F\left(\vu_{k}\right), \tF\left(\vv_{k}\right)} \\ 
= \;& B_{k+1} \eta_{k}\innp{F\left(\vu_{k+1}\right), \tF\left(\vv_{k}\right)} + \frac{B_{k} \eta_{k}}{\lambda_{k}}\innp{ F\left(\vu_{k+1}\right) - F\left(\vu_{k}\right), \tF\left(\vv_{k}\right)}.
\end{aligned}
\end{equation}
Combining Inequalities~\eqref{ineq:mono-t1} and \eqref{ineq:mono-t2} and plugging the bounds into Eq.~\eqref{eq:mono-vk-diff}, we obtain 
\begin{equation*}
    \begin{aligned}
        \cv_k - \cv_{k + 1} \geq \;& \left(\frac{A_k}{M\eta_k^2} - A_{k + 1}\right)\left\|F\left(\vu_{k + 1}\right)\right\|^2 + \frac{A_k\left(1 - M\eta_k^2\right)}{M\eta_k^2}\left\|\tF\left(\vv_{k}\right)\right\|^2 \\
        & - 2\left(\frac{A_k\left(1 - M\eta_k^2\right)}{M\eta_k^2} - \frac{B_{k + 1}\eta_k}{2}\right)\innp{F\left(\vu_{k + 1}\right), \tF\left(\vv_k\right)} \\
        & + \left(\frac{B_k\eta_k}{\lambda_k} - 2A_k\right)\innp{F\left(\vu_{k + 1}\right) - F\left(\vu_k\right), \tF\left(\vv_k\right)} \\
        & + L^2\left(\frac{A_k\theta}{M\eta_k^2} - c_{k + 1}\right)\norm{\vu_{k + 1} - \vv_k}^2 + L^2\left(c_k - A_k\right)\norm{\vu_k - \vv_{k - 1}}^2 \\
        & - \frac{2A_k}{M\eta_k^2}\norm{F(\vv_k) - \tF(\vv_k)}^2 - A_k\norm{F(\vv_{k - 1}) - \tF(\vv_{k - 1})}^2.
    \end{aligned}
\end{equation*}
By Assumption~\ref{asp:para-mono}, we choose $A_k = \frac{B_k\eta_k}{2\lambda_k}$. Define: 
\begin{equation*}
\left\{
\begin{aligned}
S_{k}^{11} := \;& \frac{A_{k}}{M \eta_{k}^{2}} - A_{k+1} = \frac{B_{k}}{2M\lambda_{k}\eta_{k}} - \frac{B_{k}\eta_{k+1}}{2\left(1-\lambda_{k}\right)\lambda_{k+1}} \\
S_{k}^{22} := \;& \frac{A_{k}\left(1 - M\eta_{k}^{2}\right)}{M\eta_{k}^{2}} = \frac{B_{k}\left(1 - M\eta_{k}^{2}\right)}{2M\eta_{k}\lambda_{k}} \\
S_{k}^{12} := \;& \frac{A_{k}\left(1 - M\eta_{k}^{2}\right)}{M\eta_{k}^{2}} - \frac{B_{k+1}\eta_{k}}{2} = \frac{\left(1 - \lambda_{k} - M\eta_{k}^{2}\right)B_{k}}{2M\left(1 - \lambda_{k}\right)\lambda_{k}\eta_{k}}.
\end{aligned}\right.
\end{equation*}
Then, we obtain 
\begin{equation*}
    \begin{aligned}
        \cv_k - \cv_{k + 1} \geq \;& S_k^{11}\left\|F\left(\vu_{k + 1}\right)\right\|^2 + S_k^{22}\left\|\tF\left(\vv_{k}\right)\right\|^2 - 2S_k^{12}\innp{F\left(\vu_{k + 1}\right), \tF\left(\vv_k\right)} \\
        & + L^2\left(\frac{A_k\theta}{M\eta_k^2} - c_{k + 1}\right)\norm{\vu_{k + 1} - \vv_k}^2 + L^2\left(c_k - A_k\right)\norm{\vu_k - \vv_{k - 1}}^2 \\
        & - \frac{2A_k}{M\eta_k^2}\norm{F(\vv_k) - \tF(\vv_k)}^2 - A_k\norm{F(\vv_{k - 1}) - \tF(\vv_{k - 1})}^2.
    \end{aligned}
\end{equation*}


Suppose that $S_k^{11} \geq 0$, $S_k^{22} \geq 0$ and $\sqrt{S_k^{11}S_k^{22}} = S_k^{12}$. Then, we can conclude 
\begin{equation*}
    \begin{aligned}
        \cv_k - \cv_{k + 1} = \;& \norm{\sqrt{S_k^{11}}F\left(\vu_{k + 1}\right) - \sqrt{S_k^{22}}\tF\left(\vv_{k}\right)}^2 \\
        & + L^2\left(\frac{A_k\theta}{M\eta_k^2} - c_{k + 1}\right)\norm{\vu_{k + 1} - \vv_k}^2 + L^2\left(c_k - A_k\right)\norm{\vu_k - \vv_{k - 1}}^2 \\ 
        & - \frac{2A_k}{M\eta_k^2}\norm{F(\vv_k) - \tF(\vv_k)}^2 - A_k\norm{F(\vv_{k - 1}) - \tF(\vv_{k - 1})}^2 \\ 
        \geq \;& L^2\left(\frac{A_k\theta}{M\eta_k^2} - c_{k + 1}\right)\norm{\vu_{k + 1} - \vv_k}^2 + L^2\left(c_k - A_k\right)\norm{\vu_k - \vv_{k - 1}}^2 \\
        & - \frac{2A_k}{M\eta_k^2}\norm{F(\vv_k) - \tF(\vv_k)}^2 - A_k\norm{F(\vv_{k - 1}) - \tF(\vv_{k - 1})}^2.
    \end{aligned}
\end{equation*}
To complete the proof, let us argue that the assumptions that $S_k^{11} \geq 0$, $S_k^{22} \geq 0$ and $\sqrt{S_k^{11}S_k^{22}} = S_k^{12}$ we made above are valid. 
First, notice that $S_k^{11} \geq 0$ is equivalent to $\eta_{k + 1} \leq \frac{\lambda_{k + 1}\left(1 - \lambda_k\right)}{M\lambda_k\eta_k}$, and $S_k^{22} \geq 0$ is equivalent to $M\eta_k^2 \leq 1$, which are both included in Assumption~\ref{asp:para-mono}. Moreover, since $B_k > 0,$  $\sqrt{S_k^{11}S_k^{22}} = S_k^{12}$ is equivalent to 
\begin{equation*}
\frac{\left(1-M \eta_{k}^{2}\right)}{M \eta_{k}} \cdot\left(\frac{1}{M \eta_{k}}-\frac{\lambda_{k} \eta_{k+1}}{\left(1-\lambda_{k}\right) \lambda_{k+1}}\right) =\left(\frac{1-\lambda_{k}-M \eta_{k}^{2}}{M\left(1-\lambda_{k}\right) \eta_{k}}\right)^{2}, 
\end{equation*}
which is further equivalent to $\eta_{k+1} = \frac{\lambda_{k+1}\left(1-M \eta_{k}^{2}-\lambda_{k}^{2}\right)}{\lambda_{k}\left(1-\lambda_{k}\right)\left(1-M \eta_{k}^{2}\right)} \cdot \eta_{k}$, provided that $M \eta_{k}^{2} + \lambda_{k}^{2} \leq 1$. Both these inequalities hold by Assumption~\ref{asp:para-mono}, thus completing the proof.
\end{proof}

Motivated by Assumption \ref{asp:para-mono} and Lemma \ref{lemma:mono-diff}, we make the choice of $\lambda_k$ and $\eta_k$ as 
\begin{equation}\label{eq:mono-eta}
\lambda_{k}:=\frac{1}{k+2} \quad \text { and } \quad \eta_{k+1}:=\frac{\left(1-\lambda_{k}^{2}-M \eta_{k}^{2}\right) \lambda_{k+1} \eta_{k}}{\left(1-M \eta_{k}^{2}\right)\left(1-\lambda_{k}\right) \lambda_{k}},
\end{equation}
where $M = 3L^2\left(2 + \theta\right)$ and $0 < \eta_0 < \frac{1}{\sqrt{2M}}$. The sequence $\{\eta_k\}_{k \geq 1}$ given by Eq.~\eqref{eq:mono-eta} is actually non-increasing and has a positive limit. We summarize this result in the following lemma for completeness, and the proof can be found in \citep{tran2021halpern}.

\begin{lemma}\label{lemma:mono-eta}
Given $M > 0$, the sequence $\{\eta_k\}$ generated by Eq.~\eqref{eq:mono-eta} is non-increasing, i.e. $\eta_{k+1} \leq \eta_{k} \leq \eta_{0}<\frac{\sqrt{3}}{2\sqrt{M}}$. Moreover, if $0 < \eta_0 < \frac{1}{\sqrt{2M}}$, we have that $\eta_{*} := \lim _{k \rightarrow \infty} \eta_{k}$ exists and
\begin{equation}
    \eta_{*} \geq \underline{\eta}:=\frac{\eta_{0}\left(1-2 M \eta_{0}^{2}\right)}{1-M \eta_{0}^{2}} > 0.
\end{equation}
\end{lemma}

We now prove the results for the iteration complexity and the corresponding oracle complexity for Algorithm~\ref{alg:monotone}.
\rateMono*
\begin{proof}
We start with verifying that the conditions in Eq.~\eqref{eq:asp-mono} of Lemma \ref{lemma:mono-diff} are all satisfied. By Eq.~\eqref{eq:mono-eta} and Lemma \ref{lemma:mono-eta}, we know that $\{\eta_k\}$ is non-increasing and $\eta_* = \lim _{k \rightarrow \infty} \eta_{k} > 0$, so the first condition in Eq.~\eqref{eq:asp-mono} is satisfied. Also, as $0 < \eta_k \leq \eta_0 \leq \frac{1}{\sqrt{2M}}$, we have $M\eta_k^2 \leq M\eta_0^2 \leq \frac{1}{2} < 1 - \frac{1}{(k + 2)^2}$. So the second condition in Eq.~\eqref{eq:asp-mono} holds. Moreover, since $\eta_{k + 1} \leq \eta_k$, the third condition holds if $\eta_k^2 \leq \frac{\lambda_{k + 1}\left(1 - \lambda_k\right)}{M\lambda_k} = \frac{k + 1}{M(k + 3)}$. Due to the fact that $\frac{k + 1}{M(k + 3)} \geq \frac{1}{3M}$ and $\eta_k \leq \eta_0$ for all $k \geq 1$, we can have this condition hold if $\eta_0 \leq \frac{1}{\sqrt{3M}}$. Hence all the conditions hold with our parameter update and letting $\eta_0 \leq \frac{1}{\sqrt{3M}}$.

Let $c_k = A_k$, then we obtain 
\begin{equation*}
\begin{aligned}
\;& L^{2}\left(\frac{A_{k}\theta}{M\eta_{k}^{2}} - c_{k + 1}\right)\norm{\vu_{k + 1} - \vv_k}^{2} + L^{2}\left(c_{k} - A_{k}\right)\norm{\vu_k - \vv_{k - 1}}^{2} \\
= \;& L^{2}\left(\frac{A_{k}\theta}{M\eta_{k}^{2}} - A_{k + 1}\right)\norm{\vu_{k + 1} - \vv_k}^{2} \\
= \;& \frac{L^{2} B_{k}}{2}\left(\frac{\theta}{M \lambda_{k} \eta_{k}}-\frac{\eta_{k+1}}{\lambda_{k+1}\left(1-\lambda_{k}\right)}\right)\norm{\vu_{k + 1} - \vv_k}^{2}.
\end{aligned}
\end{equation*}
Here we choose the parameters such that $\eta_{k} \eta_{k+1} \leq \frac{\theta \lambda_{k+1}\left(1-\lambda_{k}\right)}{M \lambda_{k}}=\frac{\theta(k+1)}{M(k+3)}$ to ensure $\left(\frac{\theta}{M \lambda_{k} \eta_{k}}-\frac{\eta_{k+1}}{\lambda_{k+1}\left(1-\lambda_{k}\right)}\right) \geq 0$. 
Since $\eta_{k +1} \leq \eta_k$, the required inequality holds if $\eta_k \leq \sqrt{\frac{\theta(k+1)}{M(k+3)}}$, which is satisfied if we let $\eta_0 \leq \sqrt{\frac{\theta}{3M}}$ as $\eta_k \leq \eta_0$ for all $k \geq 1$.

Combining the two conditions on $\eta_0$, and choosing $\theta = 1$, we have 
\begin{equation*}
\eta_{0} \leq \frac{1}{\sqrt{3 M}} = \frac{1}{3L \sqrt{2 + \theta}} = \frac{1}{3L \sqrt{3}},
\end{equation*}
which is required by Algorithm~\ref{alg:monotone} and thus satisfied. 

Hence, with $0 < \eta_0 \leq \frac{1}{3L \sqrt{3}}$, we have 
\begin{equation}\label{ineq:mono-vk}
    \begin{aligned}
        \cv_k - \cv_{k + 1} \geq \;& L^2\left(\frac{A_k\theta}{M\eta_k^2} - c_{k + 1}\right)\norm{\vu_{k + 1} - \vv_k}^2 + L^2\left(c_k - A_k\right)\norm{\vu_k - \vv_{k - 1}}^2 \\
        & - \frac{2A_k}{M\eta_k^2}\norm{F(\vv_k) - \tF(\vv_k)}^2 - A_k\norm{F(\vv_{k - 1}) - \tF(\vv_{k - 1})}^2 \\
        \geq \;& - \frac{2A_k}{M\eta_k^2}\norm{F(\vv_k) - \tF(\vv_k)}^2 - A_k\norm{F(\vv_{k - 1}) - \tF(\vv_{k - 1})}^2.
    \end{aligned}
\end{equation}
Consider $\cc_k = A_k\norm{F(\vu_k)}^2 + B_k\innp{F(\vu_k), \vu_k - \vu_0}$. Then: 
\begin{equation*}
    \begin{aligned}
        \cc_k
        \overset{(\romannumeral1)}{\geq} \;& A_k\norm{F(\vu_k)}^2 + B_k\innp{F(\vu_k) - F(\vu^*), \vu_k - \vu^*} + B_k\innp{F(\vu_k), \vu^* - \vu_0} \\
        \overset{(\romannumeral2)}{\geq} \;& A_k\norm{F(\vu_k)}^2 - \frac{A_k}{2}\norm{F(\vu_k)}^2 - \frac{B_k^2}{2A_k}\norm{\vu_0 - \vu^*}^2 \\
        = \;& \frac{A_k}{2}\norm{F(\vu_k)}^2 - \frac{B_k^2}{2A_k}\norm{\vu_0 - \vu^*}^2, 
    \end{aligned}
\end{equation*}
where $(\romannumeral1)$ is due to $\vu^*$ being the solution to the monotone inclusion problem so an \eqref{def:SVI} solution as well, and we use monotonicity and Young's Inequality for $(\romannumeral2)$. So we obtain 
\begin{equation*}
    \begin{aligned}
        \frac{A_k}{2}\norm{F(\vu_k)}^2 + A_k L^2\norm{\vu_k - \vv_{k - 1}}^2 \leq \;&  \cc_k + A_k L^2\norm{\vu_k - \vv_{k - 1}}^2 + \frac{B_k^2}{2A_k}\norm{\vu_0 - \vu^*}^2 \\
        = \;& \cv_k + \frac{B_k^2}{2A_k}\norm{\vu_0 - \vu^*}^2.
    \end{aligned}
\end{equation*}

Since $B_{k + 1} = \frac{B_k}{1 - \lambda_k}$ and $\lambda_k = \frac{1}{k + 2}$, we have $B_k = (k + 1)B_0$ for any $B_0 > 0$. Then we obtain $A_k = c_k = \frac{B_k\eta_k}{2\lambda_k} = \frac{B_0(k + 1)(k + 2)\eta_k}{2}$. By Lemma \ref{lemma:mono-eta}, we know that $0 < \underline{\eta} \leq \eta_* \leq \eta_k \leq \eta_0$, so $\frac{B_0(k + 1)(k + 2)\underline{\eta}}{2} \leq A_k = c_k \leq \frac{B_0(k + 1)(k + 2)\eta_0}{2}$. By Inequality~\eqref{ineq:mono-vk} and noticing $\vv_{-1} = \vu_0$, we have 
\begin{equation*}
    \begin{aligned}
        \;& \frac{B_0\underline{\eta}(k + 1)(k + 2)}{4}\norm{F(\vu_k)}^2 + \frac{B_0\underline{\eta}L^2(k + 1)(k + 2)}{2}\norm{\vu_k - \vv_{k - 1}}^2 \\
        \leq \;& \frac{A_k}{2}\norm{F(\vu_k)}^2 + A_k L^2\norm{\vu_k - \vv_{k - 1}}^2 \\
        \leq \;& \cv_k + \frac{B_k^2}{2A_k}\norm{\vu_0 - \vu^*}^2 \\
        \leq \;& \cv_{k - 1} + \frac{B_k^2}{2A_k}\norm{\vu_0 - \vu^*}^2 + \frac{2A_{k - 1}}{M\eta_{k - 1}^2}\norm{F(\vv_{k - 1}) - \tF(\vv_{k - 1})}^2 + A_{k - 1}\norm{F(\vv_{k - 2}) - \tF(\vv_{k - 2})}^2.
    \end{aligned}
\end{equation*}
Unrolling this recursive bound down to $\cv_0$, we obtain 
\begin{equation*}
    \begin{aligned}
        \;& \frac{B_0\underline{\eta}(k + 1)(k + 2)}{4}\norm{F(\vu_k)}^2 + \frac{B_0\underline{\eta}L^2(k + 1)(k + 2)}{2}\norm{\vu_k - \vv_{k - 1}}^2 \\
        \leq \;& \cv_0 + \frac{B_k^2}{2A_k}\norm{\vu_0 - \vu^*}^2 + \sum_{i = 0}^{k - 1}\frac{2A_{i}}{M\eta_{i}^2}\norm{F(\vv_{i}) - \tF(\vv_{i})}^2 + \sum_{i = 0}^{k - 1}A_{i}\norm{F(\vv_{i - 1}) - \tF(\vv_{i - 1})}^2 \\
        \overset{(\romannumeral1)}{\leq} \;& B_0\eta_0\norm{F(\vu_0)}^2 + \frac{B_0(k + 1)^2}{\underline{\eta}(k + 1)(k + 2)}\norm{\vu_0 - \vu^*}^2 \\
        & + \sum_{i = 0}^{k - 1}\frac{B_0(i + 1)(i + 2)}{M\underline{\eta}}\norm{F(\vv_{i}) - \tF(\vv_{i})}^2 \\
        & + \sum_{i = 0}^{k - 1}\frac{B_0(i + 1)(i + 2)\eta_0}{2}\norm{F(\vv_{i - 1}) - \tF(\vv_{i - 1})}^2, 
    \end{aligned}
\end{equation*}
where we plug in the bound for $A_i$ and $\eta_i$ in $(\romannumeral1)$.
Taking expectation with respect to all randomness on both sides and using the variance bound from Corollary~\ref{cor:variance-bound}, we obtain that 
\begin{equation*}
\begin{aligned}
\;& \ee\left[\norm{F(\vu_k)}^2 + 2L^2\norm{\vu_k - \vv_{k - 1}}^2\right] \\
\leq \;& \frac{4}{B_0\underline{\eta}(k + 1)(k + 2)}\bigg[B_0\eta_0\norm{F(\vu_0)}^2 + \frac{B_0(k + 1)^2}{\underline{\eta}(k + 1)(k + 2)}\norm{\vu_0 - \vu^*}^2 \\ 
& + \sum_{i = 0}^{k - 1}\frac{B_0(i + 1)(i + 2)}{M\underline{\eta}}\ee \Big[\norm{F(\vv_{i}) - \tF(\vv_{i})}^2\Big] \\
& + \sum_{i = 0}^{k - 1}\frac{B_0(i + 1)(i + 2)\eta_0}{2}\ee \Big[\norm{F(\vv_{i - 1}) - \tF(\vv_{i - 1})}^2\Big]\bigg] \\
\overset{(\romannumeral1)}{\leq} \;& \frac{4}{\underline{\eta}(k + 1)(k + 2)}\bigg[\left(L^2\eta_0 + \frac{1}{\underline{\eta}}\right)\norm{\vu_0 - \vu^*}^2 + \left(\frac{2}{M\underline{\eta}} + \eta_0 + 3\eta_0\right)\frac{\epsilon^2}{8} \\
& + \sum_{i = 1}^{k - 1}\frac{(i + 1)(i + 2)}{M\underline{\eta}}\frac{\epsilon^2}{i} + \sum_{i = 2}^{k - 1}\frac{(i + 1)(i + 2)\eta_0}{2}\frac{\epsilon^2}{i - 1}\bigg] \\
\overset{(\romannumeral2)}{\leq} \;& \frac{4\left(L^2\eta_0\underline{\eta} + 1\right)\norm{\vu_0 - \vu^*}^2}{\underline{\eta}^2(k + 1)(k + 2)} + \frac{1 + 2M\underline{\eta}\eta_0}{M\underline{\eta}(k + 1)(k + 2)}\epsilon^2 \\
& + \frac{4(k - 1)(k + 4)}{M\underline{\eta}^2(k + 1)(k + 2)}\epsilon^2 + \frac{4\eta_0(k - 2)(k + 3)}{\underline{\eta}(k + 1)(k + 2)}\epsilon^2 \\
\overset{(\romannumeral3)}{\leq} \;& \frac{4\left(L^2\eta_0\underline{\eta} + 1\right)\norm{\vu_0 - \vu^*}^2}{\underline{\eta}^2(k + 1)(k + 2)} + \frac{\left(1 + 2M\underline{\eta}\eta_0\right)}{M\underline{\eta}^2(k + 1)(k + 2)}\epsilon^2 + \frac{4\left(1 + M\eta_0\underline{\eta}\right)}{M\underline{\eta}^2}\epsilon^2, 
\end{aligned}
\end{equation*}
where we use Lipschitz property and variance bounds by variance reduction for $(\romannumeral1)$. For $(\romannumeral2)$, we use the fact that $\frac{i + 1}{i} \leq 2$ and $\frac{i + 2}{i - 1} \leq 4$ and sum over $2(i + 1)$, respectively. Moreover, $(\romannumeral3)$ is due to $\frac{(k - 1)(k + 4)}{(k + 1)(k + 3)} \leq 1$ and $\frac{(k - 2)(k + 3)}{(k + 1)(k + 2)} \leq 1$ and by combining the last two terms.

When $k \geq \frac{\sqrt{\Lambda_0}}{\sqrt{\Lambda_1}\epsilon} = \co\Big(\frac{L\norm{\vu_0 - \vu^*}}{\epsilon}\Big)$, where $\Lambda_0 = \frac{4\left(L^2\eta_0\underline{\eta} + 1\right)\norm{\vu_0 - \vu^*}^2}{\underline{\eta}^2}$ and $\Lambda_1 = \frac{5\left(1 + M\underline{\eta}\eta_0\right)}{M\underline{\eta}^2}$, we have 
\begin{equation*}
    \begin{aligned}
        \ee\left[\norm{F(\vu_k)}^2 + 2L^2\norm{\vu_k - \vv_{k - 1}}^2\right] \leq \;& \frac{\Lambda_0}{(k + 1)(k + 2)} + \Lambda_1\epsilon^2 \leq 2\Lambda_1\epsilon^2.
    \end{aligned}
\end{equation*}
Claimed stochastic oracle complexity follows from Lemma~\ref{coro:mono-complexity} below.
\end{proof}
\begin{restatable}{lemma}{complexityMono}
\label{coro:mono-complexity}
Let $\vu_0 \in \rr^d$ be an arbitrary initial point and assume that iterates $\vu_k$ evolve according to Algorithm~\ref{alg:monotone}. Then, Algorithm~\ref{alg:monotone}
returns a point $\vu_N$ such that $\ee\big[\norm{F(\vu_N)}^2\big] \leq 2\Lambda_1\epsilon^2$ after at most $\co\big(\frac{\sigma^2L\norm{\vu_0 - \vu^*} + L^3\norm{\vu_0 - \vu^*}^3}{\epsilon^3}\big)$ stochastic queries to $F$.
\end{restatable}
\begin{proof}
Let $m_k$ be the number of stochastic queries made by the variance reduction method at iteration $k$ for $k \geq 1$. Conditional on $\cf_{k - 1}$, we have 
\begin{align*}
    \ee\big[m_{k + 1} | \cf_{k - 1}\big] = \;& \ee\Big[p_k S_1^{(k)} + 2(1 - p_k)S_2^{(k)} \Big| \cf_{k - 1}\Big] \\
    = \;& p_k\Big\lceil\frac{8\sigma^2}{p_k\epsilon^2}\Big\rceil + 2(1 - p_k)\Big\lceil\frac{8L^2\norm{\vv_k - \vv_{k - 1}}^2}{p_k^2\epsilon^2}\Big\rceil \\
    \overset{(\romannumeral1)}{\leq} \;& p_k\Big(\frac{8\sigma^2}{p_k\epsilon^2} + 1\Big) + 2(1 - p_k)\Big(\frac{8L^2\norm{\vv_k - \vv_{k - 1}}^2}{p_k^2\epsilon^2} + 1\Big), 
\end{align*}
where $(\romannumeral1)$ is due to the fact that $\lceil x \rceil \leq x + 1$ for any $x \in \rr$. Taking expectation with respect to all randomness on both sides, and rearranging the terms, we obtain \begin{align*}
    \ee[m_{k + 1}] \leq \frac{8\sigma^2}{\epsilon^2} + \frac{16(1 - p_k)L^2\ee\big[\norm{\vv_k - \vv_{k - 1}}^2\big]}{p_k^2\epsilon^2} + 2.
\end{align*}
With $m_{0} = m_{1} = S_1^{(0)} = \big\lceil\frac{8\sigma^2}{\epsilon^2}\big\rceil$, let $M$ be the total number of stochastic queries up to iteration $N$ such that $\ee[\norm{F(\vu_k)}] \leq 2\Lambda_1\epsilon$ for all $k \geq N$, we have 
\begin{align}
    \ee[M] = \ee\Big[\sum_{k = 0}^{N}m_k\Big] = \;& 2\Big\lceil\frac{8\sigma^2}{\epsilon^2}\Big\rceil + \ee\Big[\sum_{k = 2}^{N} m_k\Big] \notag \\
    \leq \;& \frac{16\sigma^2}{\epsilon^2} + 2 + \sum_{k = 1}^{N} \Big(\frac{8\sigma^2}{\epsilon^2} + \frac{16(1 - p_k)L^2\ee\big[\norm{\vv_k - \vv_{k - 1}}^2\big]}{p_k^2\epsilon^2} + 2\Big) \notag \\
    \overset{(\romannumeral1)}{\leq} \;& \frac{8\sigma^2\Delta}{\epsilon^3} + \frac{16\sigma^2}{\epsilon^2} + \frac{2\Delta}{\epsilon} + \frac{16L^2}{\epsilon^2}\sum_{k = 1}^{N}\frac{(1 - p_k)\ee\big[\norm{\vv_k - \vv_{k - 1}}^2\big]}{p_k^2}, 
\end{align}
where $(\romannumeral1)$ follows from $N \leq \frac{\Delta}{\epsilon}$ with $\Delta = \sqrt{\frac{\Lambda_0}{\Lambda_1}} + \epsilon$, and $1 - p_k \leq 1$. 

Then we come to bound $\ee\big[\norm{\vv_k - \vv_{k - 1}}^2\big]$. Notice that for $k \geq 1$
\begin{equation*}
    \begin{aligned}
        \vv_k - \vv_{k - 1} = \;& \vv_k - \vu_{k + 1} + \vu_{k + 1} - \vu_{k} + \vu_{k} - \vv_{k - 1} \\
        \overset{(\romannumeral1)}{=} \;& \eta_k\big(\tF(\vv_k) - \tF(\vv_{k - 1})\big) - \eta_{k - 1}\big(\tF(\vv_{k - 1}) - \tF(\vv_{k - 2})\big) + \vu_{k + 1} - \vu_{k} \\
        = \;& \eta_k\tF(\vv_k) - \big(\eta_k + \eta_{k - 1}\big)\tF(\vv_{k - 1}) + \eta_{k - 1}\tF(\vv_{k - 2}) + \vu_{k + 1} - \vu_{k},
    \end{aligned}
\end{equation*}
where $(\romannumeral1)$ is based on the third line in Eq.~\eqref{eq:mono-uk-diff}. To estimate $\vu_{k + 1} - \vu_k$, we recursively use the first line in Eq.~\eqref{eq:mono-uk-diff}, and obtain for $k \geq 2$,
\begin{equation*}
    \begin{aligned}
        \vu_{k + 1} - \vu_{k} = \;& \lambda_{k}\left(\vu_0 - \vu_{k}\right) - \eta_{k}\tF\left(\vv_{k}\right) \\
        = \;& \lambda_{k}\left(1 - \lambda_{k - 1}\right)\left(\vu_0 - \vu_{k - 1}\right) + \lambda_{k}\eta_{k - 1}\tF\left(\vv_{k - 1}\right) - \eta_{k}\tF\left(\vv_{k}\right) \\
        = \;& \lambda_{k}\sum_{i = 0}^{k - 2}\Big(\prod_{j = i + 1}^{k - 1}\left(1 - \lambda_j\right)\Big)\eta_i\tF(\vv_{i}) + \lambda_{k}\eta_{k - 1}\tF\left(\vv_{k - 1}\right) - \eta_{k}\tF\left(\vv_{k}\right) \\
        = \;& \sum_{i = 0}^{k - 2}\frac{i + 2}{(k + 2)(k + 1)}\eta_i\tF(\vv_i) + \frac{\eta_{k - 1}}{k + 2}\tF(\vv_{k - 1}) - \eta_{k}\tF(\vv_{k}).
    \end{aligned}
\end{equation*}
So we have 
\begin{align}
    \vu_{k + 1} - \vu_{k} = 
    \begin{cases}
         -\eta_0\tF(\vv_0) & \text{if } k = 0, \\
         & \\
         \lambda_1\eta_0\tF(\vv_0) - \eta_1\tF(\vv_1) & \text{if } k = 1, \\
         & \\
         \sum_{i = 0}^{k - 2}\frac{i + 2}{(k + 2)(k + 1)}\eta_i\tF(\vv_i) + \frac{\eta_{k - 1}}{k + 2}\tF(\vv_{k - 1}) - \eta_{k}\tF(\vv_{k}) & \text{if } k \geq 2.
    \end{cases}
\end{align}
Then we obtain for $k \geq 3$
\begin{equation*}
\begin{aligned}
    \;& \norm{\vv_k - \vv_{k - 1}}^2 \\
    = \;& \norm{\Big(\frac{\eta_{k - 1}}{k + 2} - \eta_k - \eta_{k - 1}\Big)\tF(\vv_{k - 1}) + \eta_{k - 1}\tF(\vv_{k - 2}) + \sum_{i = 0}^{k - 2}\frac{i + 2}{(k + 2)(k + 1)}\eta_i\tF(\vv_i)}^2 \\
    \leq \;& 3\Big(\frac{\eta_{k - 1}}{k + 2} - \eta_k - \eta_{k - 1}\Big)^2\norm{\tF(\vv_{k - 1})}^2 + 3\Big(\eta_{k - 1} + \frac{k\eta_{k - 2}}{(k + 1)(k + 2)}\Big)^2\norm{\tF(\vv_{k - 2})}^2 \\
    & + \sum_{i = 0}^{k - 3}\frac{3(i + 2)^2(k - 2)}{k^2(k + 1)^2}\eta_i^2\norm{\tF(\vv_i)}^2 \\
    \leq \;& \eta_{0}^2\Big(12\norm{\tF(\vv_{k - 1})}^2 + 5\norm{\tF(\vv_{k - 2})}^2\Big) + \sum_{i = 0}^{k - 3}\frac{3(i + 2)^2}{k(k + 1)^2}\eta_0^2\norm{\tF(\vv_i)}^2.
\end{aligned}
\end{equation*}
Taking expectation with respect to all randomness on both sides, we have 
\begin{equation*}\label{ineq:mono-euk-diff}
    \ee\norm{\vv_k - \vv_{k - 1}}^2 \leq \eta_{0}^2\Big(12\ee\norm{\tF(\vv_{k - 1})}^2 + 5\ee\norm{\tF(\vv_{k - 2})}^2\Big) + \sum_{i = 0}^{k - 3}\frac{3(i + 2)^2}{k(k + 1)^2}\eta_0^2\ee\norm{\tF(\vv_i)}^2.
\end{equation*}
Note that for $k \geq 1$, we have
\begin{equation*}
    \begin{aligned}
        \ee\norm{\tF(\vv_k)}^2 \leq \;& 2\ee\norm{F(\vu_{k + 1})}^2 + 4\ee\norm{F(\vu_{k + 1}) - F(\vv_k)}^2 + 4\ee\norm{\tF(\vv_k) - F(\vv_k)}^2 \\
        \overset{(\romannumeral1)}{\leq} \;& 2\ee\left(\norm{F(\vu_{k + 1})}^2 + 2L^2\norm{\vu_{k + 1} - \vv_k}^2\right) + 4\frac{\epsilon^2}{k} \\
        \overset{(\romannumeral2)}{\leq} \;& 2\left(\frac{\Lambda_0}{(k + 1)^2} + \Lambda_1\epsilon^2\right) + 4\frac{\epsilon^2}{k}, 
    \end{aligned}
\end{equation*}
where $(\romannumeral1)$ is due to the Lipschitz property and the variance bound, and we use the result of Theorem \ref{thm:mono-rate} for $(\romannumeral2)$.
Proceeding similarly, we have $\ee\norm{\tF(\vv_0)}^2 \leq  2\left(\Lambda_0 + \Lambda_1\epsilon^2\right) + \frac{\epsilon^2}{2}$, so we obtain for $k \geq 4$, 
\begin{equation*}
    \begin{aligned}
        \;& \ee\norm{\vv_k - \vv_{k - 1}}^2 \\
        \leq \;& \eta_{0}^2\Big(12\norm{\tF(\vv_{k - 1})}^2 + 5\norm{\tF(\vv_{k - 2})}^2\Big) + \sum_{i = 0}^{k - 3}\frac{3(i + 2)^2}{k(k + 1)^2}\eta_0^2\norm{\tF(\vv_i)}^2 \\
        \leq \;& \eta_{0}^2\left[24\left(\frac{\Lambda_0}{k^2} + \Lambda_1\epsilon^2 + 2\frac{\epsilon^2}{k - 1}\right) + 10\left(\frac{\Lambda_0}{(k - 1)^2} + \Lambda_1\epsilon^2 + 2\frac{\epsilon^2}{k - 2}\right)\right] \\
        & + \frac{12\eta_0^2}{k(k + 1)^2}\left[2\left(\Lambda_0 + \Lambda_1\epsilon^2\right) + \frac{\epsilon^2}{2}\right] + \sum_{i = 1}^{k - 3}\frac{6(i + 2)^2}{k(k + 1)^2}\eta_0^2\left(\frac{\Lambda_0}{(i + 1)^2} + \Lambda_1\epsilon^2 + 2\frac{\epsilon^2}{i}\right) \\
        \leq \;& 40\eta_0^2\frac{\Lambda_0}{(k - 1)^2} + 35\eta_0^2\left(\Lambda_1 + 1\right)\epsilon^2 + 6\eta_0^2\sum_{i = 1}^{k - 3}\left(\frac{4\Lambda_0}{k(k + 1)^2} + \frac{\Lambda_1\epsilon^2}{k} + \frac{2\epsilon^2}{k}\right) \\
        \leq \;& 40\eta_0^2\frac{\Lambda_0}{(k - 1)^2} + 35\eta_0^2\left(\Lambda_1 + 1\right)\epsilon^2 + 6\eta_0^2\left(\frac{4\Lambda_0}{(k + 1)^2} + \left(\Lambda_1 + 2\right)\epsilon^2\right) \\
        = \;& \eta_0^2\Lambda_0\Big(\frac{40}{(k - 1)^2} + \frac{24}{(k + 1)^2}\Big) + 41\eta_0^2\left(\Lambda_1 + 2\right)\epsilon^2.
    \end{aligned}
\end{equation*}
Since $p_k = \frac{2}{k + 1}$, we have for $k \geq 4$
\begin{equation*}
\begin{aligned}
    \frac{\ee\norm{\vv_k - \vv_{k - 1}}^2}{p_k^2} \leq \;& \frac{(k + 1)^2}{4}\left[\eta_0^2\Lambda_0\Big(\frac{40}{(k - 1)^2} + \frac{24}{(k + 1)^2}\Big) + 41\eta_0^2\left(\Lambda_1 + 2\right)\epsilon^2\right] \\
    \leq \;& 30\eta_0^2\Lambda_0 + 11\eta_0^2\left(\Lambda_1 + 2\right)\epsilon^2(k + 1)^2 \\
    \overset{(\romannumeral1)}{\leq} \;& 30\eta_0^2\Lambda_0 + 11\eta_0^2\left(\Lambda_1 + 2\right)\epsilon^2\frac{\Delta^2}{\epsilon^2} \\
    = \;& 30\eta_0^2\Lambda_0 + 11\eta_0^2\left(\Lambda_1 + 2\right)\Delta^2, 
\end{aligned}
\end{equation*}
where $(\romannumeral1)$ is due to $k \leq N < \frac{\Delta}{\epsilon}$.
So we obtain
\begin{equation*}
    \begin{aligned}
        \sum_{k = 4}^{N}\frac{(1 - p_k)\ee\big[\norm{\vv_k - \vv_{k - 1}}^2\big]}{p_k^2} \leq \;& \sum_{k = 1}^{N}\frac{\ee\big[\norm{\vv_k - \vv_{k - 1}}^2\big]}{p_k^2} \\
        \leq \;& \sum_{k = 1}^{N}\left[30\eta_0^2\Lambda_0 + 11\eta_0^2\left(\Lambda_1 + 2\right)\Delta^2\right] \\
        \leq \;& \frac{\Delta\left(30\eta_0^2\Lambda_0 + 11\eta_0^2\left(\Lambda_1 + 2\right)\Delta^2\right)}{\epsilon}.
    \end{aligned}
\end{equation*}
For $k = 3$, we have 
\begin{equation*}
    \begin{aligned}
        \frac{(1 - p_3)\ee\norm{\vv_3 - \vv_2}^2}{p_3^2} = \;& 2\ee\norm{\vv_3 - \vv_2}^2 \leq 2\Big(10\eta_0^2\Lambda_0 + 35\eta_0^2\left(\Lambda_1 + 2\right)\epsilon^2\Big).
    \end{aligned}
\end{equation*}
Moreover, we have 
\begin{equation*}
    \begin{aligned}
        \frac{(1 - p_2)\ee\norm{\vv_2 - \vv_1}^2}{p_2^2} = \;& \frac{3}{4}\ee\norm{\vv_2 - \vv_1}^2 \\
        \leq \;& \frac{3}{2}\ee\left[\Big(\frac{\eta_0}{6} + \eta_1\Big)^2\norm{\tF(\vu_0)}^2 + \Big(\frac{3\eta_1}{4} + \eta_2\Big)^2\norm{\tF(\vu_1)}^2\right] \\
        \leq \;& \frac{3}{2}\eta_0^2\left[\frac{49}{36}\left(2\Lambda_0 + 2\Lambda_1\epsilon^2 + \epsilon^2 / 2\right) + \frac{49}{16}\Big( \frac{\Lambda_0}{2} + 2\Lambda_1\epsilon^2 + 4\epsilon^2\Big)\right] \\
        \leq \;& 7\eta_0^2\Lambda_0 + 14\left(\Lambda_1 + 2\right)\epsilon^2.
    \end{aligned}
\end{equation*}
Note that $p_1 = 1$, so we have 
\begin{equation*}
    \begin{aligned}
        \ee\left[M\right] \leq \;& \frac{8\sigma^2\Delta}{\epsilon^3} + \frac{16\sigma^2}{\epsilon^2} + \frac{2\Delta}{\epsilon} + \frac{16L^2}{\epsilon^2}\sum_{k = 1}^{N}\frac{(1 - p_k)\ee\big[\norm{\vv_k - \vv_{k - 1}}^2\big]}{p_k^2} \\
        \leq \;& \frac{8\sigma^2\Delta}{\epsilon^3} + \frac{16\sigma^2}{\epsilon^2} + \frac{2\Delta}{\epsilon} + \frac{16L^2\Delta\left(24\eta_0^2\Lambda_0 + 5\eta_0^2\left(\Lambda_1 + 2\right)\Delta^2\right)}{\epsilon^3} \\
        & + \frac{32L^2\Big(10\eta_0^2\Lambda_0 + 35\eta_0^2\left(\Lambda_1 + 2\right)\epsilon^2\Big)}{\epsilon^2} + \frac{16L^2\Big(7\eta_0^2\Lambda_0 + 14\left(\Lambda_1 + 2\right)\epsilon^2\Big)}{\epsilon^2}\\
        = \;& \co\Big(\frac{\sigma^2L\norm{\vu_0 - \vu^*} + L^3\norm{\vu_0 - \vu^*}^3}{\epsilon^3}\Big), 
    \end{aligned}
\end{equation*}
where we assume without loss of generality that $L\norm{\vu_0 - \vu^*} \geq 1$, thus completing the proof.
\end{proof}

\section{Omitted proofs from Section~\ref{sec:str_monotone}}\label{appx:str_monotone}
\complexityStrong*
\begin{proof}
Let ${\cal G}_{k-1}$ be the natural filtration of all the random variables used up to (and including) the $(k-1)^\mathrm{th}$ outer loop. 
By Theorem \ref{thm:mono-rate}, we have 
\begin{equation}\label{ineq:rate-sharp}
    \ee\left[\norm{F(\vu_k)}^2 | {\cal G}_{k-1}\right] \leq \frac{\Lambda_0^{(k)}}{(K + 1)(K + 2)} + \Lambda_1^{(k)}\epsilon_k^2, 
\end{equation}
where 
$\Lambda_0^{(k)} =\frac{4(L^2\eta_0\underline{\eta} + 1)\norm{\vu_{k - 1} - \vu^*}^2}{\underline{\eta}^2}$ and $\Lambda_1^{(k)} = \frac{5\left(1 + M\underline{\eta}\eta_0\right)}{M\underline{\eta}^2}$. 

By the sharpness condition, we have 
\begin{equation*}
\begin{aligned}
  \norm{\vu_k - \vu^*}^2 \leq \;& \frac{1}{\mu}\innp{F(\vu_k) - F(\vu^*), \vu_k - \vu^*} \\
  \overset{(\romannumeral1)}{\leq}\;& \frac{1}{\mu}\innp{F(\vu_k), \vu_k - \vu^*} \\
  \overset{(\romannumeral2)}{\leq} \;& \frac{1}{\mu}\norm{F(\vu_k)}\norm{\vu_k - \vu^*}, 
\end{aligned}
\end{equation*}
where $(\romannumeral1)$ is because $\vu^*$ is a solution to \eqref{def:SVI}, and we use Cauchy-Schwarz inequality for $(\romannumeral2)$.
Taking expectation conditional on ${\cal G}_{k-1}$ on both sides, we have 
\begin{equation*}
    \ee\left[\norm{F(\vu_k)}^2 | {\cal G}_{k-1}\right] \geq \ee\left[\mu^2\norm{\vu_k - \vu^*}^2 | {\cal G}_{k-1}\right], 
\end{equation*}
which leads to 
\begin{equation*}
    \ee\left[\norm{\vu_k - \vu^*}^2 | {\cal G}_{k-1}\right] \leq \frac{1}{\mu^2}\left[\frac{\Lambda_0^{(k)}}{(K + 1)(K + 2)} + \Lambda_1^{(k)}\epsilon_k^2\right].
\end{equation*}
If we choose $K \geq \frac{4\sqrt{L^2\eta_0\underline{\eta} + 1}}{\mu\underline{\eta}}$, we have $\frac{\Lambda_0^{(k)}}{(K + 1)(K + 2)} \leq \frac{\mu^2\norm{\vu_{k - 1} - \vu^*}^2}{4}$. On the other hand, by our choice of $\epsilon_k$ in Algorithm~\ref{alg:sharp},  we obtain 
\begin{equation*}
    \Lambda_1^{(k)}\epsilon_k^2 \leq \frac{5\left(1 + M\underline{\eta}\eta_0\right)}{M\underline{\eta}^2} \frac{\mu^2\epsilon^2M\underline{\eta}^2}{20\left(1 + M\underline{\eta}\eta_0\right)} \leq \frac{\mu^2\epsilon^2}{4}.
\end{equation*}
So we have
\begin{equation*}
\ee\left[\norm{\vu_k - \vu^*}^2 | {\cal G}_{k-1}\right] \leq \frac{\norm{\vu_{k - 1} - \vu^*}^2}{4} + \frac{\epsilon^2}{4}.
\end{equation*}
Taking expectation with respect to all the randomness on both sides, we obtain 
\begin{equation}\label{ineq:recursive-sharp}
    \ee\big[\norm{\vu_k - \vu^*}^2\big] \leq \frac{\ee\big[\norm{\vu_{k - 1} - \vu^*}^2\big]}{4} + \frac{\epsilon^2}{4}.
\end{equation}
Recursively using Inequality~\eqref{ineq:recursive-sharp} till $k = 0$, we have 
\begin{equation*}
   \ee\big[\norm{\vu_k - \vu^*}^2\big] \leq \frac{1}{4^k}\norm{\vu_0 - \vu^*}^2 + \sum_{i = 1}^{k}\frac{\epsilon^2}{4^i} \leq  \frac{1}{4^k}\norm{\vu_0 - \vu^*}^2 + \sum_{i = 1}^{\infty}\frac{\epsilon^2}{4^i} \leq \frac{1}{4^k}\norm{\vu_0 - \vu^*}^2 + \frac{\epsilon^2}{3}.
\end{equation*}

Hence, after $\Big\lceil\log\Big(\frac{\sqrt{6}\norm{\vu_0 - \vu^*}}{2\epsilon}\Big)\Big\rceil$ outer loops, the  Algorithm~\ref{alg:sharp} can output a point $\vu_k$ such that $\ee\big[\norm{\vu_k - \vu^*}^2\big] \leq \epsilon^2$, as well as $\ee\big[\norm{F(\vu_{k})}^2\big] \leq L^2\epsilon^2$. More specifically, the total number of iterations such that the algorithm can return a point $\vu_k$ such that $\ee\norm{\vu_k - \vu^*}^2 \leq \epsilon^2$ will be 
\begin{equation*}
    \Big\lceil\log\Big(\frac{\sqrt{6}\norm{\vu_0 - \vu^*}}{2\epsilon}\Big)\Big\rceil\Big\lceil\frac{4\sqrt{L^2\eta_0\underline{\eta} + 1}}{\mu\underline{\eta}}\Big\rceil = \co\left(\frac{L}{\mu}\log\frac{\norm{\vu_0 - \vu^*}}{\epsilon}\right).
\end{equation*}

Next we come to bound the expected number of the stochastic oracle queries for each call to Algorithm $\mathcal{A}$. Denote $i$-th iterate in $k$-th call as $\vu_i^{(k)}$ and $\vv_i^{(k)}$, and let $K = \Big\lceil\frac{4\sqrt{L^2\eta_0\underline{\eta} + 1}}{\mu\underline{\eta}}\Big\rceil$, then proceeding as in the proof of Corollary~\ref{coro:mono-complexity}, we obtain  
\begin{equation}\label{ineq:strong-N}
\begin{aligned}
    \ee\big[M_k | {\cal G}_{k-1}\big] = \;& \ee\Big[\sum_{i = 0}^{K}m_i^{(k)} \Big| {\cal G}_{k-1} \Big] \\
    \leq \;& \frac{16\sigma^2}{\epsilon_k^2} + 2 + \sum_{i = 1}^{K} \Big(\frac{8\sigma^2}{\epsilon_k^2} + \frac{16(1 - p_i)L^2\ee\Big[\norm{\vv_i^{(k)} - \vv_{i - 1}^{(k)}}^2 \Big| {\cal G}_{k-1}  \Big]}{p_i^2\epsilon_k^2} + 2\Big) \\
    = \;& \frac{16\sigma^2}{\epsilon_k^2} + 2(K + 1) +  \frac{8\sigma^2K}{\epsilon_k^2} + \sum_{i = 1}^{K}\frac{16(1 - p_i)L^2\ee\Big[\norm{\vv_i^{(k)} - \vv_{i - 1}^{(k)}}^2 \Big| {\cal G}_{k-1}\Big]}{p_i^2\epsilon_k^2}, 
\end{aligned}
\end{equation}
where $M_k$ is the total number of queries at the $k^\mathrm{th}$ call.
Notice that $K = \Big\lceil\frac{4\sqrt{L^2\eta_0\underline{\eta} + 1}}{\mu\underline{\eta}}\Big\rceil = \co\left(
\frac{L}{\mu}\right)$ and $\epsilon_k^2 = \frac{\mu^2\epsilon^2M\underline{\eta}^2}{20\left(1 + M\underline{\eta}\eta_0\right)} = \co\left(\mu^2\epsilon^2\right)$, then it remains to bound $\frac{\ee\Big[\norm{\vv_i^{(k)} - \vv_{i - 1}^{(k)}}^2\Big]}{p_i^2}$ for $1 \leq i \leq K$. The proof of Lemma~\ref{coro:mono-complexity} shows that for $i \geq 4$  
\begin{equation*}
    \frac{\ee\Big[\norm{\vv_i^{(k)} - \vv_{i - 1}^{(k)}}^2 \Big| {\cal G}_{k-1}\Big]}{p_i^2} \leq 30\eta_0^2\Lambda_0^{(k)} + 11\eta_0^2\left(\Lambda_1^{(k)} + 2\right)\epsilon_k^2(i + 1)^2.
\end{equation*}
On the other hand, for $1 \leq i \leq 3$, we have $\sum_{i = 1}^{3}\frac{\ee\big[\norm{\vv_i^{(k)} - \vv_{i - 1}^{(k)}}^2\big]}{p_i^2} = o(\frac{L^2\norm{\vu_{k - 1} - \vu^*}^2}{\epsilon})$, we obtain 
\begin{equation*}
\begin{aligned}
    \sum_{i = 1}^{K}\frac{16(1 - p_i)L^2\ee\Big[\norm{\vv_i^{(k)} - \vv_{i - 1}^{(k)}}^2\Big| {\cal G}_{k-1}\Big]}{p_i^2\epsilon_k^2} =\;&  \co\Big(\frac{L^3\norm{\vu_{k - 1} - \vu^*}^2}{\mu^3\epsilon^2}\Big).
\end{aligned}
\end{equation*}
Combining last inequality with 
Inequality~\eqref{ineq:strong-N} and taking expectations on both sides, we obtain 
\begin{equation*}
    \ee[M_k] = \co\Big(\frac{\sigma^2(\mu + L) + L^3\ee\big[\norm{\vu_{k - 1} - \vu^*}^2\big]}{\mu^3\epsilon^2}\Big).
\end{equation*}
Telescoping from $k = 1$ to $N = \Big\lceil\log\Big(\frac{\sqrt{6}\norm{\vu_0 - \vu^*}}{2\epsilon}\Big)\Big\rceil$ and noticing that 
\begin{equation*}
    \ee\big[\norm{\vu_k - \vu^*}^2\big] \leq \frac{1}{4}\ee\big[\norm{\vu_{k - 1} - \vu^*}^2\big] + \frac{\epsilon^2}{4} \leq \frac{1}{4^k}\norm{\vu_0 - \vu^*}^2 + \frac{\epsilon^2}{3}, 
\end{equation*}
we have 
\begin{equation*}
    \sum_{k = 1}^{N}\ee\norm{\vu_{k - 1} - \vu^*}^2 \leq \norm{\vu_0 - \vu^*}^2\sum_{k = 1}^{\infty}\frac{1}{4^k} + \frac{N\epsilon^2}{3} \leq \frac{\norm{\vu_0 - \vu^*}^2}{3} + \frac{N\epsilon^2}{3}.
\end{equation*}
Hence, we finally arrive at 
\begin{equation*}
\begin{aligned}
    \ee\Big[\sum_{k = 1}^{N}M_k\Big] = \;& \co\Big(\frac{\sigma^2(\mu + L)\log(\norm{\vu_0 - \vu^*}/ \epsilon) + L^3\norm{\vu_0 - \vu^*}^2}{\mu^3\epsilon^2}\Big),
\end{aligned}
\end{equation*}
which completes the proof.
\end{proof}
\end{document}